\documentclass[a4paper,final]{siamltex}

\usepackage[english]{babel}
\usepackage{amssymb,amsfonts,amsmath,latexsym}
\usepackage[dvips]{graphicx}
\usepackage{algorithm}
\usepackage{algorithmic}
\usepackage{color}
\usepackage{url}
\usepackage{hyperref}

\newcommand{\bm}[1]{\mathbf{#1}}
\newcommand{\RR}{\mathbb{R}}
\newcommand{\cN}{{\mathcal{N}}}

\newcommand{\cR}{{\mathcal{R}}}
\newcommand{\cF}{{\mathcal{F}}}
\newcommand{\cG}{{\mathcal{G}}}

\newcommand{\ee}{{\mathrm{e}}}
\newcommand{\psib}{\boldsymbol\psi}
\newcommand{\dx}{\,dx}
\newcommand{\dt}{\,dt}
\newcommand{\dz}{\,dz}
\DeclareMathOperator{\Span}{span}
\DeclareMathOperator{\arcsinh}{arcsinh}

\newtheorem{remark}[theorem]{Remark}
\newtheorem{example}[theorem]{Example}

\begin{document}

\title{Regularized minimal-norm solution of an overdetermined system of first 
kind integral equations
}

\author{Patricia D\'iaz de Alba\thanks{Department of Mathematics, University of 
Salerno, via Giovanni Paolo II 132, 84084 Fisciano, Italy, 
\texttt{pdiazdealba@unisa.it}}
\and
Luisa Fermo\thanks{Department of Mathematics and Computer Science, University 
of Cagliari, via Ospedale 72, 09124 Cagliari, Italy, \texttt{fermo@unica.it, 
federica.pes@unica.it, rodriguez@unica.it}}
\and
Federica Pes\footnotemark[2]
\and
Giuseppe Rodriguez\footnotemark[2]
}

\maketitle

\begin{center}
\emph{Dedicated to Claude Brezinski on the occasion of his 80th birthday.}
\end{center}
\medskip

\begin{abstract}
Overdetermined systems of first kind integral equations appear in many
applications. When the right-hand side is discretized, the resulting finite-data
problem is ill-posed and admits infinitely many solutions. We propose a
numerical method to compute the minimal-norm solution in the presence of
boundary constraints. The algorithm stems from the Riesz representation theorem
and operates in a reproducing kernel Hilbert space. Since the resulting linear
system is strongly ill-conditioned, we construct a regularization method
depending on a discrete parameter. It is based on the expansion of the
minimal-norm solution in terms of the singular functions of the integral
operator defining the problem. Two estimation techniques are tested for the
automatic determination of the regularization parameter, namely, the discrepancy
principle and the L-curve method. Numerical results concerning two artificial
test problems demonstrate the excellent performance of the proposed method.
Finally, a particular model typical of geophysical applications, which
reproduces the readings of a frequency domain electromagnetic induction device,
is investigated. The results show that the new method is extremely effective
when the sought solution is smooth, but produces significant information even
for non-smooth solutions.
\end{abstract}



\section{Introduction}\label{sec:introduction}

Fredholm integral equations of the first kind model several physical problems 
arising in different contexts such as medical imaging, image processing, signal 
processing and geophysics. Their standard form is 
\begin{equation}\label{equationFirstkind}
\int_a^b k(x,t) f(t) \dt =g(x), \qquad x\in[c,d],
\end{equation}
where the right-hand  side $g$, usually given at a finite set of points 
$x=x_i$, $i=1,\ldots,n$, represents the experimental data, the kernel $k$, 
often analytically 
known, stands for the impulse response of the experimental equipment, and the 
function 
$f$ is the signal to recover. 

From a theoretical point of view, they are treated in a Hilbert space setting 
which typically coincides with the space of square-integrable functions. The 
corresponding integral operator 
$$
(Kf)(x)=\int_a^b k(x,t) f(t) \dt
$$
is a bounded linear operator from a Hilbert space $H_1$ into a Hilbert space
$H_2$, and a solution $f$ of \eqref{equationFirstkind} exists only if
the right-hand side $g$ belongs to the range of $K$, $\cR(K)\subset H_2$. 
Consequently,
the existence of the solution of \eqref{equationFirstkind} cannot be 
guaranteed for any right-hand side, but only for a restricted class of 
functions $g$ \cite{Groetsch2007}.
The uniqueness of the solution depends upon the structure of the null space of
the operator $K$, but even when it is ensured the problem is still ill-posed
since the stability is missing; see \cite[pag.~155]{Groetsch}.

In an experimental setting, $g$ is certainly an element of $\cR(K)$, as it
represents the data $g(x_i)$ produced by an operator $K$ which reproduces a
real situation. This leads to the integral equation with discrete data
\begin{equation}\label{equationDiscrete}
\int_a^b k(x_i,t) f(t) \dt =g(x_i), \quad i=1,\dots,n.
\end{equation}
However, even when $g\in\cR(K)$, the data values in \eqref{equationDiscrete} are
affected by perturbations due to measuring and rounding errors, so one cannot
be sure that the perturbed right-hand side lies exactly in the range of $K$. 
Moreover, the
solution of \eqref{equationDiscrete} is not unique and it does not depend
continuously on the data. In other words, a discretization
\eqref{equationDiscrete} of  equation \eqref{equationFirstkind} is an
ill-posed problem \cite{Hadamard,wing1991}. This fact makes its numerical
treatment rather delicate, especially if compared to the discretization of
integral equations of the second kind, a typical example of a well-posed 
problem~\cite{Atkinson}. 

The non-uniqueness of the solution of \eqref{equationDiscrete} can be stated as
follows. Let us consider the functions $k_i(t)=k(x_i,t)$, $i=1,\ldots,n$.
By the Gram--Schmidt process it is possible to construct a set of orthonormal
functions $\phi_j(t)$, $j=1,\ldots,\bar{n}\leq n$, such that
$$
\mathcal{S} = \Span\{\phi_1,\ldots,\phi_{\bar{n}}\} = \Span\{k_1,\ldots,k_n\}.
$$
Chosen any function $\psi(t)$ linearly independent of $k_i(t)$, $i=1,\ldots,n$,
the function
$$
\phi_{\bar{n}+1}(t) = \psi(t) - \sum_{j=1}^{\bar{n}} \langle \psi, \phi_j
\rangle \phi_j
$$
is orthogonal to $\mathcal{S}$, so that whenever $f(t)$ is a solution of
\eqref{equationDiscrete} also $f(t)+\alpha\phi_{\bar{n}+1}(t)$ is, for any
$\alpha\in\RR$.

The same considerations about ill-posedness can be repeated for a
system of linear integral equations of the first kind.
In this paper, we focus on overdetermined systems of linear integral equations,
e.g., two equations whose solution is a single unknown function.
According to our knowledge, this problem has not been addressed before in the
literature, although it arises in a variety of applications.
Indeed, specific physical systems can be observed by different devices, or by
the same device with different configurations. This fact results in writing
distinct equations with the same unknown. 

An example is given by the geophysical model presented in \cite{McNeill}; see
also Section \ref{sec:case study}. It reproduces the readings of a
ground conductivity meter, a device composed of two coils, a transmitter and
a receiver, placed at a fixed distance from each other.
The model consists of two
integral equations of the first kind involving the same unknown function, 
representing the electrical conductivity of the soil at a certain depth; see
equations \eqref{linmodel}. The first equation describes the situation in
which both coil axes are aligned vertically with respect to the ground level,
while the second one corresponds to the horizontal orientation of the coils.
This system has been studied in \cite{dfrv19}, under the assumption that the
values of the unknown function at the boundaries are known, either on the basis
of additional measurements or of known geophysical properties of the subsoil.

Further applications are the model considered in \cite{hendr02}, and the
Radon transform \cite{wang2019,wang2021}. In all these situations, the model is
written in terms of an overdetermined system and a priori boundary information
on the signal to recover may be known.

In this paper, motivated by these applications and with the purpose of
developing a method that can be applied to different physical models, we focus
on the following system of $m$ integral equations of the first kind
\begin{equation}\label{model}
\begin{cases}
\displaystyle
\int_a^b k_\ell(x,t) \, f(t) \dt = g_\ell(x), \quad \ell=1,\dots,m,  \quad x
\in [c_\ell,d_\ell], \\
f(a)=f_0, \ f(b)=f_1,
\end{cases}
\end{equation}
where $k_\ell$ and $g_\ell$ are the given kernel and right-hand side  of the 
$\ell$-th equation, respectively, and $f$ is the function to be determined 
satisfying known constraints at the boundary. 
Specifically, given the data at a finite (and often small, in applications) set
of points $x_{\ell,i}\in[c_\ell,d_\ell]$, $i=1,\dots,n_\ell$,
we aim at solving the problem with discrete data
\begin{equation}\label{model2}
\begin{cases}
\displaystyle
\int_a^b k_\ell(x_{\ell,i},t) \, f(t) \dt = g_\ell(x_{\ell,i}), \qquad
\ell=1,\dots,m, \quad i=1,\dots,n_\ell, \\
f(a)=f_0, \ f(b)=f_1.
\end{cases}
\end{equation}

As already observed, a discrete data integral problem as \eqref{model2} has
infinitely many solutions.
Since the data may not belong to the range of the operator, we reformulate it
as a minimal-norm least-squares problem and solve the latter in suitable
function spaces. 
While this approach is rather standard in functional analysis, it has never
been applied to an overdetermined system.
Moreover, as we will show, the corresponding algorithm proves to be very
accurate in the absence of experimental errors, if compared to other standard
approaches, and it naturally leads to an effective regularization technique,
when the data is affected by noise.

Specifically, we consider a reproducing kernel Hilbert space
where, by using the Riesz theory, the minimal-norm solution can be written as a
linear combination of the so-called Riesz representers. Then, the main issue is
to determine the Riesz functions as well as the coefficients of such a linear
combination.
The first ones, which are determined by the reproducing kernel, are expressed
in terms of integrals which need suitable quadrature schemes, whenever they
cannot be evaluated analytically.
The coefficients are obtained by solving a square ill-conditioned linear system.
If the data is only affected by rounding errors, this representation proves to
be accurate.
If the noise level is realistic, the error propagation
completely cancels the solution and a regularized approach is required.

To this end, we introduce a regularization method to solve problem
\eqref{model2}, based on a truncated expansion in terms of the singular
functions of the corresponding integral operator.
To improve stability, the singular system is not explicitly used in the
construction of the regularized solution, which is still represented as a
linear combination of the Riesz representers instead.
We prove that the coefficients of such regularized expansion are obtained by
applying the truncated eigenvalue decomposition to the initial ill-conditioned
linear system.
The truncation index is, in fact, a regularization parameter, which we
determine by different estimation approaches.
The effectiveness of the resulting solution method is confirmed by
numerical experiments, which involve both artificial examples and an integral
model reproducing the propagation of an electromagnetic field in the earth soil.

Reproducing kernel Hilbert spaces \cite{aronszajn} are a powerful and flexible
tool of functional analysis. They have been applied to many different fields,
such as numerical analysis \cite{cui}, optimization \cite{wahba2003},
statistics \cite{berlinet}, and machine learning \cite{cucker2002}.
In \cite{rs90,castro2012} they have been used in the numerical solution of
integral equations, in \cite{rs93a,sawano2008} to develop real inversion
methods for the Laplace transform, while \cite{evgeniou2000} discusses an
interesting application of reproducing kernels and radial basis functions to
machine learning problems.

In principle, the solution of \eqref{model2} could be handled by standard
projection methods using, e.g., splines or orthogonal polynomials.
Such an approach would produce, even in infinite arithmetics, an approximation
of the minimal-norm solution. On the contrary, the method here presented
constructs the exact solution to the problem. The approximation is introduced
in the algorithm by the floating point system and by the regularization
procedure.
Moreover, our method performs an implicit orthogonalization of the basis
functions which span the space containing the exact solution.
We note that a projection method for a particular system of integral equations
based on spline functions has been studied in \cite{dfrv19}.

We remark that a preliminary version of the procedure described in this paper,
still not completely motivated from a theoretical point of view, has been
applied by the same authors to the solution of a single equation in a specific
applicative context in \cite{dfpr21}.
The computation of minimal-norm solutions to nonlinear least-squares problems
is much more involved; see, e.g., \cite{pr20,pr21}.

The structure of the paper is as follows. In Section \ref{sec:preliminaries}, we
reformulate \eqref{model2} as a minimal-norm solution problem in
suitable Hilbert spaces. Then, in Section \ref{sec:method}, we develop a
solution method which leads to an ill-conditioned linear system, whose
regularized solution is characterized in Section~\ref{sec:regularization}.
In Section \ref{sec:tests}, we show the performance of our method by some
numerical examples, and in Section \ref{sec:case study} we conclude the paper
with the application of the proposed numerical approach to a geophysical model.

\section{Mathematical preliminaries}\label{sec:preliminaries}

\subsection{Statement of the problem}\label{sec:statement}

Let us consider problem \eqref{model2} and, from now on, let us assume that
$f_0=f_1=0$. This assumption does not affect the generality. Indeed, if it is
not fulfilled, by introducing the linear function
\begin{equation}\label{gamma}
\gamma(t) = \frac{b-t}{b-a} f_0 + \frac{t-a}{b-a} f_1,
\end{equation}
we can rewrite problem \eqref{model} into an equivalent one with vanishing
boundary conditions
\begin{equation}\label{model3}
\begin{cases}
\displaystyle
\int_a^b k_\ell(x,t) \, \xi(t) \dt = \varphi_\ell(x), \qquad 
\ell=1,\dots,m, \\
\xi(a)=0, \  \xi(b)=0,
\end{cases}
\end{equation}
where 
\begin{equation}\label{xi}
\begin{aligned}
\xi(t)=f(t)-\gamma(t), \qquad
\varphi_\ell(x) = g_\ell(x) - \int_a^b k_\ell(x,t) \, \gamma(t) \dt,
\end{aligned}
\end{equation}
are the new unknown function and right-hand side, respectively.

Let us now introduce the integral operators  
\begin{equation}\label{operator}
(K_\ell f)(x):= \int_a^b k_\ell(x,t) \, f(t) \dt, \qquad \ell=1,\dots,m,
\end{equation}
so that problem \eqref{model2} can be written as
\begin{equation}\label{modeloperator1}
\begin{cases}
\displaystyle
(K_\ell f)(x_{\ell,i}) = g_\ell(x_{\ell,i}), 
\qquad \ell=1,\dots,m,
\quad i=1,\dots,n_\ell, \\
f(a)=0, \  f(b)=0,
\end{cases}
\end{equation}
or, equivalently,
\begin{equation}\label{modeloperator}
\begin{cases}
\displaystyle
\bm{K}f = \bm{g}, \\
f(a)=0, \ f(b)=0,
\end{cases}
\end{equation}
where  
\begin{equation}\label{Kfg}
\bm{K}f= \begin{bmatrix}
\bm{K}_1f \\ \vdots \\ \bm{K}_mf
\end{bmatrix}, \qquad 
\bm{g} = \begin{bmatrix}
\bm{g}_1 \\ \vdots \\ \bm{g}_m
\end{bmatrix}, 
\end{equation}
and
$$
\begin{aligned}
\bm{K}_\ell f&=[(K_\ell f)(x_{\ell,1}),\dots,(K_\ell f)(x_{\ell,n_\ell})]^T, \\
\bm{g}_\ell&=[g_\ell(x_{\ell,1}),\dots,g_\ell(x_{\ell,n_\ell})]^T, 
\end{aligned}
$$
are vectors in $\RR^{n_\ell}$ for $\ell=1,\dots,m$. 

As already remarked in Section \ref{sec:introduction}, the above problem is 
ill-posed.
If the right-hand side does not belong to the range of the operator the
solution does not exist; this happens, in particular, when the data are
affected by errors. Moreover, the solution is not unique.
Because of this, we reformulate \eqref{modeloperator} in terms 
of the following least-squares problem
\begin{equation}\label{leastsquare}
\min_{f} \|\bm{K} f-\bm{g}\|_2^2,
\end{equation}
where $\| \cdot\|_2$ is the standard Euclidean norm.
Problem \eqref{leastsquare} has infinitely many solutions and among them we
look for a function $f(t)$ which satisfies
\begin{equation}\label{minfs}
\min \int_a^b \left(f''(t) \right)^2 \dt.
\end{equation}
We note that the curvature of a function $f$ at $t\in[a,b]$ is given by
$f''(t)(1+f'(t)^2)^{-3/2}$. If $f'$ is relatively small on the interval, then
\eqref{minfs} approximates the total curvature of the function on $[a,b]$, and
its minimization promotes the determination of a smooth solution.

In the space of square-integrable functions, this solution may not be unique.
It is necessary to introduce a suitable function space in which \eqref{minfs}
represents a strictly convex norm. In this way, the uniqueness of the solution
is ensured.

%
\begin{remark}\rm
Let us observe that in case $f$ does not satisfy homogeneous boundary
conditions, so that we have to reformulate the original problem  as
\eqref{model3}, from \eqref{gamma} and \eqref{xi}, we obtain
$$
\min \int_a^b \left(f''(t) \right)^2 \dt = \min \int_a^b \left( \xi''(t) 
\right)^2 \dt.
$$
This means that, after collocation, selecting the solution $f$ of \eqref{model2}
satisfying \eqref{minfs} corresponds to computing the minimal-norm solution of
\eqref{modeloperator1} in a suitable Hilbert space.
\end{remark}

\begin{remark}\rm
Approximating the solution of a problem by a smooth function is rather common
in applied mathematics. For example, 
let $a=x_0<x_1<\cdots<x_n=b$ and $y_i\in\RR$, $i=0,\ldots,n$, be given. 
The function $f$ such that $f(x_i)=y_i$, $i=0,\ldots,n$, is said to be an
interpolant.
It is well known \cite[Theorem~2.4.1.5]{sb91} that the interpolant which
minimizes \eqref{minfs} over all functions with absolutely continuous first
derivative and second derivative in $L^2[a,b]$ is an interpolating natural
cubic spline $s(x)$.
Here ``natural'' means that $s''(a)=s''(b)=0$.
We will show in Section~\ref{sec:method} that the smoothest solution of
\eqref{modeloperator1} can be  uniquely represented
as the expansion of basis functions depending upon the integral
operators $K_\ell$ and the collocation points $x_{\ell,i}$, for
$\ell=1,\ldots,m$ and $i=1,\ldots,n_\ell$.
\end{remark}

\subsection{Function spaces}\label{sec:funcspac}

Let us now introduce a function space for the solution of such a 
problem. 
Let $L^2$ be the Hilbert space of square-integrable functions $f: [a,b]
\rightarrow \RR$, equipped with the inner product
\begin{equation*}
\langle f,g \rangle_{L^2}=\int_a^b f(x) g(x) \dx,
\end{equation*}
and the induced norm
\begin{equation*}
\|f\|_{L^2}=\sqrt{\langle f,f \rangle_{L^2}}.
\end{equation*}
Let us also define the Hilbert space
\begin{equation*}
W= \{ f \in L^2: f(a)=f(b)=0,  f,f' \in AC([a,b]),  f'' \in L^2\},
\end{equation*}
where $AC([a,b])$ denotes the set of all functions $f$ that are absolutely
continuous on $[a,b]$, with inner product 
\begin{equation}\label{normW}
\langle f,g \rangle_{W} =
\langle f'',g'' \rangle_{L^2},
\end{equation}
and induced norm 
$$
\|f\|_W = \|f''\|_{L^2}.
$$
This is a norm for $W$. Indeed, $\|f\|_W=0$ if and only if
$f$ is a linear function (see, e.g., \cite[Section~2.4.1]{sb91}) and $f\in W$
implies $f(a)=f(b)=0$, so that $f\equiv 0$.

The space $W$ is a \emph{reproducing kernel Hilbert space} (RKHS). This means 
that each function $f$ belonging to $W$ can be written as
\begin{equation}\label{f}
f(y)=\langle G_y,f \rangle_W,
\end{equation}
where $G: [a,b]\times [a,b] \rightarrow \RR$ is a known bivariate function  such
that, for any $y \in [a,b]$,
$$G_y(x)\in W.$$
The function $G$ is called the \emph{reproducing kernel}.
Its expression is given by
$$
G(x,y) = G_y(x) = \int_a^b G_x''(z) G_y''(z)\dz,
$$ where
$$
G_y''(z)=\frac{\partial^2 G_y(z)}{\partial z^2}=
\begin{cases}
\dfrac{(z-a)(y-b)}{b-a}, & \quad a \leq z <y, \\
\dfrac{(y-a)(z-b)}{b-a}, & \quad y \leq z \leq b. \\
\end{cases}
$$
It is easy to check that from \eqref{normW} and \eqref{f} it follows 
\begin{equation*}
f(y)= \int_a^b G_y''(z) f''(z)\dz.
\end{equation*}
Further examples and properties of reproducing kernels
can be found in \cite{aronszajn,hille,yosida2}.

\subsection{Riesz theory}\label{sec:rieszth}

Let us now consider problem \eqref{modeloperator} in $W$. This means that the 
bounded linear functional $\bm{K}$ is such that
$$
\begin{aligned}
\bm{K}   :\  & W && \longrightarrow && \RR^{N_m} \\
& f && \longmapsto && \bm{K}  f,
\end{aligned}
$$
with
\begin{equation}\label{Kfj}
(\bm{K} f)_j =(K_\ell f)(x_{\ell,i}), \quad j=i+N_{\ell-1},
\quad \quad N_r=\sum_{k=1}^r n_k,
\end{equation}
$\ell=1,\ldots,m$, $i=1,\ldots,n_\ell$, and $N_0=0$.

By the Riesz representation theorem~\cite{yosida2}, there exist $N_m$ functions 
$\{\eta_{j}\}_{j=1}^{N_m} \in W$, named \emph{Riesz representers}, such that
the $j$th component of the array 
$ \bm{K} f$ is given by
\begin{equation}\label{Riesz}
(\bm{K} f)_j=\langle \eta_{j},f \rangle_{W}, \quad j=1,\ldots,N_m.
\end{equation}
Moreover, let us denote by $\bm{K}^*: \RR^{N_m} \to W$ the adjoint 
operator of $\bm{K}$, defined by
\begin{equation}\label{propadjoint}
\langle \bm{K} f, \bm{g} \rangle_2 = \langle f,\bm{K}^* \bm{g} \rangle_{W},
\end{equation}
where $\bm{g}\in\RR^{N_m}$ and $\langle \cdot, \cdot \rangle_2$ is the usual
Euclidean inner product in $\RR^{N_m}$.
Let us also introduce the null space of $\bm{K}$  
$$ \cN(\bm{K} )= \{f \in W : \bm{K} f=\bm{0} \},$$
and  its orthogonal complement
$$ \cN(\bm{K})^{\perp}= \{f \in W : \langle f,g \rangle_{W} = 0,  
\forall g \in \cN(\bm{K}) \}.$$
The latter space is spanned by the Riesz representers, as the following lemma
states.

\begin{lemma}\label{lemma1}
Let $\bm{K}$ be a bounded linear operator from a Hilbert space to a 
finite-dimensional Hilbert space, then $\cN(\bm{K})^\perp$ coincides 
with the range of the adjoint operator $\cR(\bm{K}^*)$
\begin{equation*}
\cN(\bm{K})^\perp = \cR(\bm{K}^*) =\{f \in W  :  
f=\bm{K}^* \bm{g}\ \text{ for }\ \bm{g} \in \RR^{N_m} \},
\end{equation*}
and, in addition, 
$$
\cN(\bm{K})^\perp= \Span \{ \eta_{1},  \dots, \eta_{N_m} \}.
$$
\end{lemma}

\begin{proof}
In \cite[Theorem 3.3.2]{Groetsch} it is proved that
$\cN(\bm{K})^\perp = \overline{\cR(\bm{K}^*)}$. 
In our case, $\cR(\bm{K}^*)$ is finite-dimensional, so the closure is
not needed.
For any $f\in W$ and $\bm{g}\in\RR^{N_m}$, we have
$$
\langle \bm{K} f, \bm{g} \rangle_2 = \sum_{\ell=1}^{m} \langle \bm{K}_\ell
f, \bm{g}_\ell \rangle_2 = \sum_{\ell=1}^{m} \sum_{i=1}^{n_\ell} (K_\ell 
f)(x_{\ell,i})
\,g_\ell(x_{\ell,i}).
$$
Then, by combining \eqref{Kfj} and \eqref{Riesz}, we can assert
\begin{equation*}
\begin{aligned}
\langle \bm{K} f, \bm{g} \rangle_2 &= \sum_{\ell=1}^{m} \sum_{i=1}^{n_\ell}
\langle \eta_{i+N_{\ell-1}},f \rangle_{W} \,  g_\ell(x_{\ell,i}) \\
&= \left \langle f, \sum_{\ell=1}^{m} \sum_{i=1}^{n_\ell} g_\ell(x_{\ell,i}) \,
\eta_{i+N_{\ell-1}}
\right \rangle_{W} = \left \langle f, \bm{K}^* \bm{g} \right \rangle_{W},
\end{aligned}
\end{equation*}
where the last equality follows by virtue of \eqref{propadjoint}.
This shows that any function in the range of $\bm{K}^*$ can be
expressed as a linear combination of the Riesz representers $\eta_j$,
$j=1,\ldots,N_m$.
\end{proof}

\section{Computing the minimal-norm solution}\label{sec:method}

In this section, we develop a projection method to compute the minimal-norm
solution of \eqref{modeloperator}.
As a consequence of Lemma~\ref{lemma1}, such a solution can be
expressed as a linear combination of the Riesz representers, as the following
theorem shows.

\begin{theorem}\label{teo1}
The minimal-norm solution $f^\dagger$ of~\eqref{modeloperator} is given by
\begin{equation}\label{normalphi}
f^\dagger = \sum_{\ell=1}^m 
\sum_{i=1}^{n_\ell} c_{i+N_{\ell-1}}  \eta_{\ell,i},
\end{equation}
with $\eta_{\ell,i}:=\eta_{i+N_{\ell-1}}$.
\end{theorem}

\begin{proof}
Since the minimal-norm solution $f^\dagger$ belongs to $\cN(\bm{K})^\perp$, 
from Lemma~\ref{lemma1} we can write 
\begin{equation*}
f^\dagger = \sum_{j=1}^{N_m} c_j \eta_j = \sum_{\ell=1}^m 
\sum_{i=1}^{n_\ell} c_{i+N_{\ell-1}}  \eta_{\ell,i}, \qquad 
\text{with } \eta_{\ell,i}:=\eta_{i+N_{\ell-1}}.
\end{equation*}
\end{proof}

The Riesz representers are functions in the space $W$, so we have
\begin{equation}\label{eta}
\eta_{\ell,i}(t) = \langle G_t,\eta_{i+N_{\ell-1}} \rangle_W
\quad\text{and}\quad \eta_{\ell,i}(a)=\eta_{\ell,i}(b)=0.
\end{equation}
Given the definition \eqref{normW} of the inner product,
to obtain the Riesz representers $\eta_{\ell,i}(t)$ the
expressions of their second derivatives $\eta_{\ell,i}''$ are needed, for
$\ell=1,\dots,m$ and $i=1,\ldots,n_\ell$.
To this end, we consider \eqref{operator} and write the unknown function $f$ by
\eqref{f}
\begin{align*}
(K_\ell f)(x_{\ell,i}) & = \int_a^b k_\ell(x_{\ell,i},t) \int_a^b G_t''(z)
\, f''(z)  \dz \dt  \\ 
& =\int_a^b f''(z) \int_a^b G_t''(z) \, k_\ell(x_{\ell,i},t) \dt \dz,
\end{align*}
from which, by \eqref{Riesz}, we deduce
\begin{equation}\label{etasecondo}
\eta_{\ell,i}''(z) =
\int_a^b  G_t''(z) \, k_\ell(x_{\ell,i},t) \dt, 
\end{equation}
for $\ell=1,\dots,m$ and $i=1,\ldots,n_\ell$.

Let us mention that, depending on the expression of the kernels $k_\ell$, the
above integrals may be analytically computed.
Whenever this is not possible, we employ a Gaussian quadrature formula of
suitable order to approximate \eqref{etasecondo}.
The following two examples illustrate both situations.
Here, we assume $m=2$, $n_1=n_2=n$, so that $N_m=2n$, and $x_{1,i}=x_{2,i}=x_i$,
for $i=1,\ldots,n$.

\begin{example}\label{example1}
\rm
Let us consider the system of integral equations
\begin{equation}\label{ex1}
\begin{cases}
\displaystyle
\int_0^{1} \frac{x}{t+1}  f(t) \dt  =  x\left(\log{4}-\frac{1}{2}\right), \\
\displaystyle
\int_0^{1\strut} \cos{(xt)}  f(t) \dt  = 
\frac{2}{x^3}\left(x\cos{x} + (x^2-1)\sin{x} \right),
\end{cases}
\end{equation}
with $x\in(0,1]$, whose exact solution is $f(t)= t^2+1$.
We introduce the function \eqref{gamma}
\[
\gamma(t)=t+1,
\]
to reformulate the original problem as the following one
\begin{equation*}
\begin{cases}
\displaystyle
\int_0^{1} \frac{x}{t+1}  \xi(t) \dt  =  
x\left(\log{4}-\frac{3}{2}\right), \\
\displaystyle
\int_0^{1\strut} \cos{(xt)}  \xi(t) \dt  =
\frac{1}{x^2}\left(\cos{x} +1 - \frac{2\sin{x}}{x} \right),
\end{cases}
\end{equation*}
where $\xi(t)=f(t)-\gamma(t)$ satisfies homogeneous boundary conditions.

From \eqref{etasecondo}, after some computation, we obtain, for $i=1,\ldots,n$,
\begin{align}\label{eta1S}
\eta_{1, i}''(z) &= x_{i} \left[ (1-z)\log(1+z) -z \log\left(\frac{4}{(1+z)^2} 
\right) \right], \\
\eta_{2, i}''(z) &= \frac{1}{x_{i}^2} \left(z\cos{x_{i}} - \cos{(x_{i}z)}-z+1
\right).
\label{eta2S}
\end{align}
Then, from \eqref{eta},
\begin{align}\label{eta1}
\eta_{1, i}(y) &= \frac{x_{i}}{36} \Bigl\{ 6(1+y)^3 \log{(1+y)} -y \left[ y^2 
(5+12\log{2} ) +15y+4 (9\log{2} -5) \right] \Bigr\}, \\
\eta_{2, i}(y) &= \frac{y(y-1)}{6x_{i}^2}\bigl[ (y+1)\cos{(x_{i})} -y+2 \bigr]
+ \frac{1}{x_i^4} \bigl[ y(1-\cos{(x_{i})})-1+\cos{(x_{i}y)} \bigr].
\label{eta2}
\end{align}

Figure~\ref{fig1} displays, in the top row, the functions $\eta_{1,i}$ (on the 
left) and $\eta''_{1,i}$ (on the right), while the bottom row depicts the
functions $\eta_{2,i}$ (on the left) and $\eta''_{2,i}$ (on the right) for
different collocation points $x_{\ell,i}$.
We see from Figure~\ref{fig1} that the Riesz functions satisfy the boundary
conditions, i.e., $\eta_{\ell,i}(0) = \eta_{\ell,i}(1) = 0$,
for $\ell = 1, 2$ and $i = 1,\ldots,5$.
From the same figure we can observe that, in this case, it also holds
$\eta''_{\ell,i}(0) = \eta''_{\ell,i}(1) = 0$.

\begin{figure}[ht] \centering
\includegraphics[width=.46\textwidth]{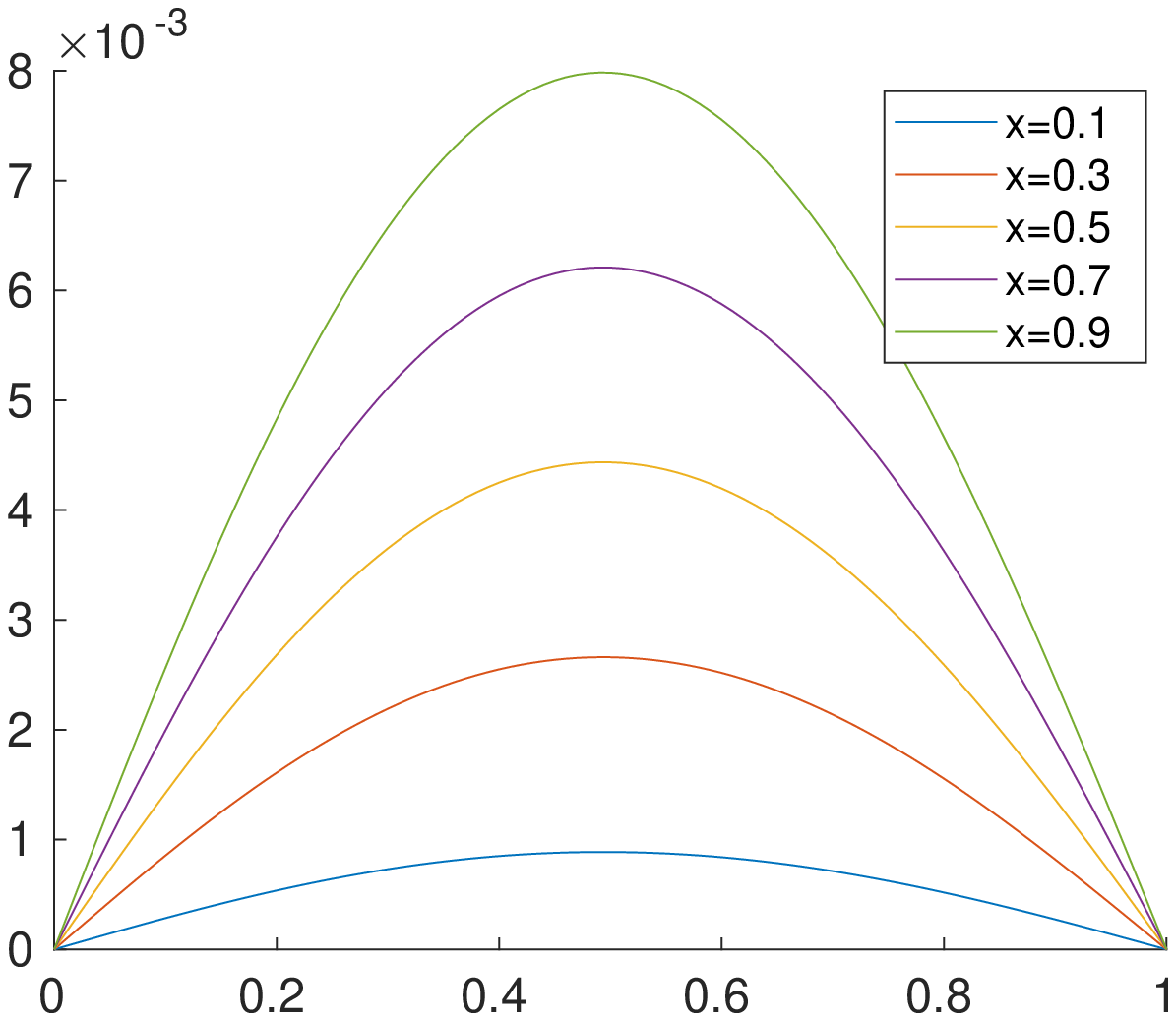}\hfill
\includegraphics[width=.46\textwidth]{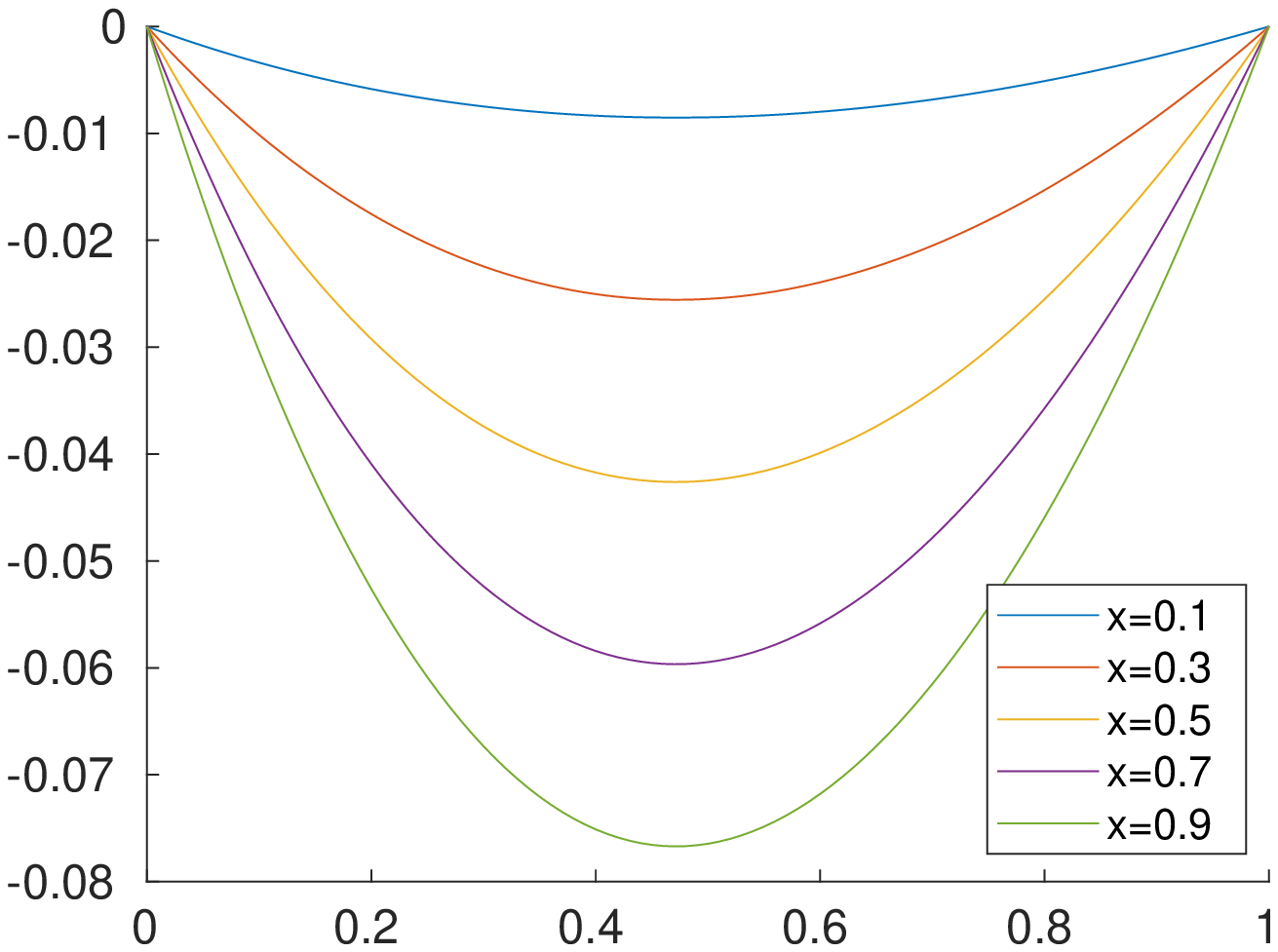}\\ \vspace{.5cm}
\includegraphics[width=.46\textwidth]{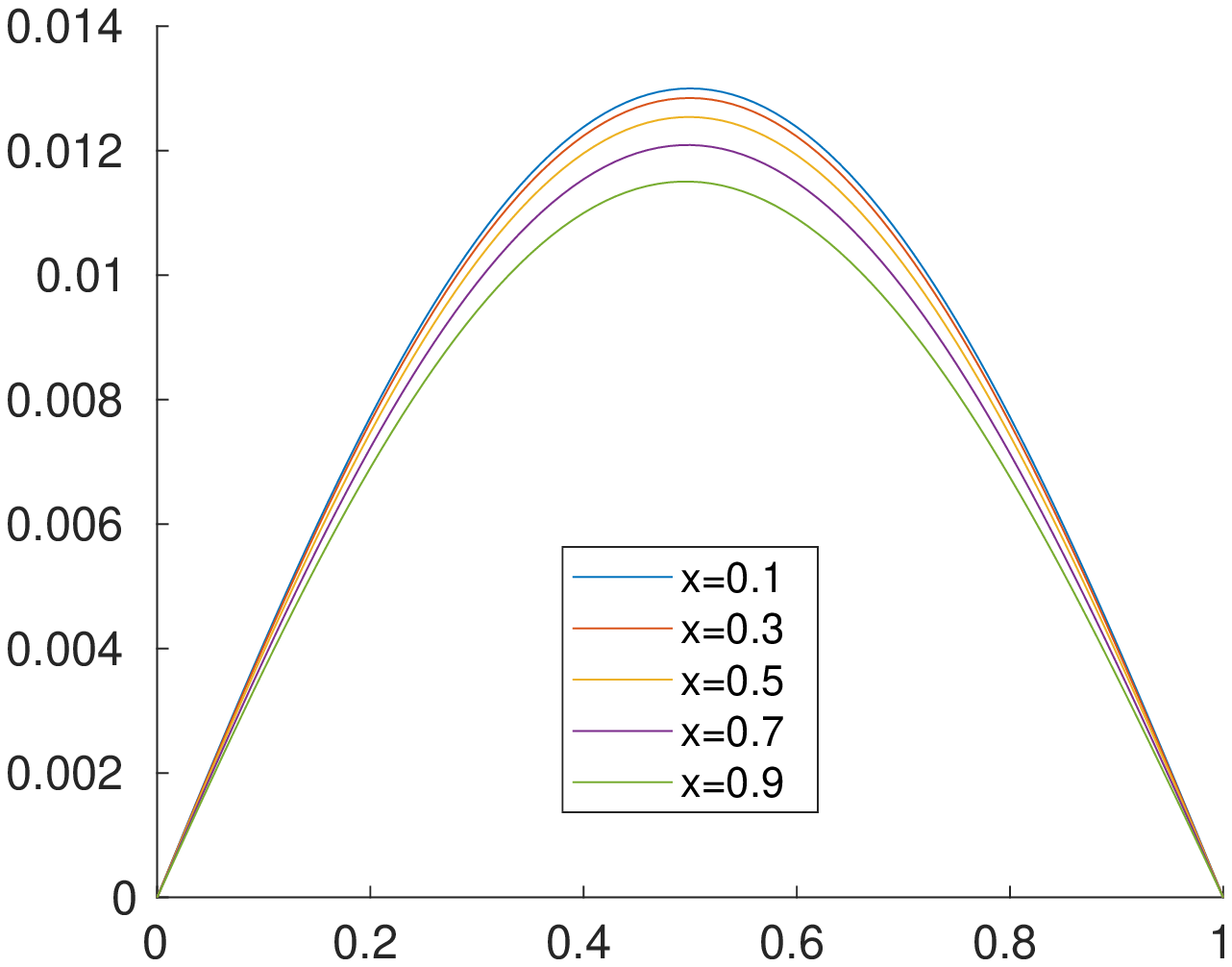}\hfill
\includegraphics[width=.46\textwidth]{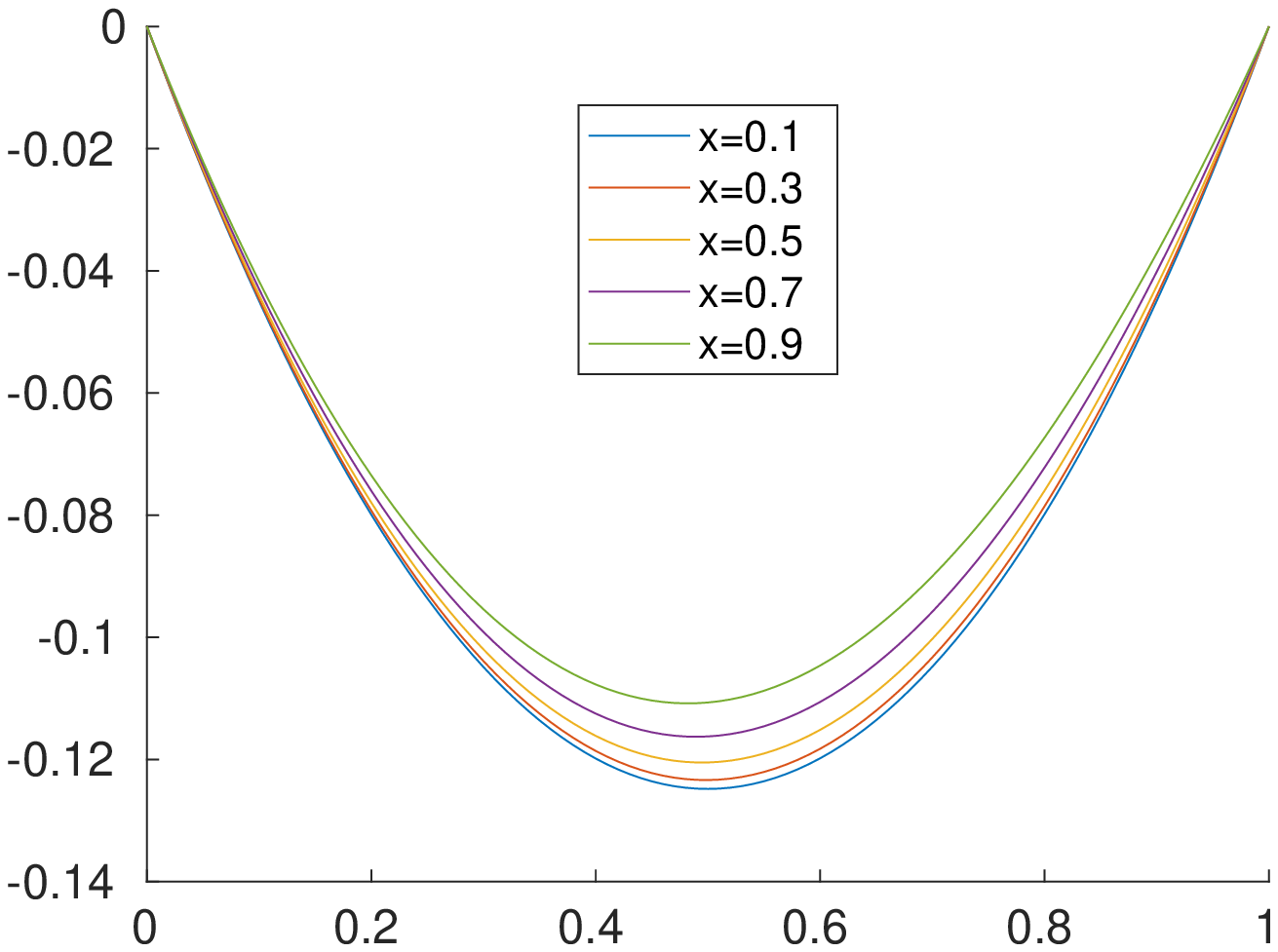}
\caption{Riesz functions for the system \eqref{ex1}: $\eta_{1,i}$ (top-left),
$\eta''_{1,i}$ (top-right), $\eta_{2,i}$ (bottom-left), and $\eta''_{2,i}$
(bottom-right), with $x_i = 0.1+0.2(i-1)$ for $i = 1,\ldots,5$.}
\label{fig1}
\end{figure}

\end{example}

\begin{example}\label{exbaart}
\rm
Let us consider the system
\begin{equation}\label{sysbaart}
\begin{cases}
\displaystyle
\int_0^{\pi} \ee^{x \cos{t}}  f(t) \dt  = 2 \frac{\sinh{x}}{x}, \\
\displaystyle
\int_0^{\pi\strut} (xt+\ee^{xt})  f(t) \dt  = \pi x+\frac{1+\ee^{\pi x}}{1+x^2},
\end{cases}
\end{equation}
with $x\in(0,\pi/2]$, whose exact solution is $f(t)=\sin t$. This system has
been obtained by coupling the well-known Baart test problem \cite{Hansentools}
with another equation having the same solution.

From \eqref{etasecondo} we have, for $i=1,\ldots,n$,
\begin{equation}\label{eta2Sex2}
\eta_{2, i}''(z) = \frac{z(1-\ee^{\pi x_{i}})}{\pi x_{i}^2} + \frac{x_{i} 
z(z^2-\pi^2)}{6} + \frac{\ee^{x_{i}z}-1}{x_{i}^2},
\end{equation}
and from~\eqref{eta}
\begin{equation}\label{eta2ex2}
\begin{aligned}
\eta_{2, i}(y) &= \frac{\pi^2x_{i}y}{36}\left(\frac{7}{10}\pi^2-y^2 \right) + 
\frac{y}{6\pi x_{i}^4}(1-\ee^{\pi x_{i}})(x_{i}^2y^2+6) \\
&\phantom{=\ } +\frac{\pi y}{6x_{i}^2}(\ee^{\pi 
x_{i}}+2) + \frac{y^2}{2}\left( \frac{x_{i}y^3}{60}-\frac{1}{x_{i}^2}\right)+ 
\frac{\ee^{x_{i}y}-1}{x_{i}^4}.
\end{aligned}
\end{equation}
The functions $\eta_{1, i}''(z)$ and $\eta_{1, i}(y)$ do not have an analytic
representation, so they should be approximated by a quadrature formula.
\end{example}

Let us now compute the coefficient of the expansion
\eqref{normalphi} of the minimal-norm solution.
By replacing $f$ in \eqref{modeloperator1} by \eqref{normalphi}, we obtain
\begin{equation*}
(K_\ell f^\dagger)(x_{\ell,i}) = g_\ell(x_{\ell,i}), 
\qquad \ell=1,\dots,m, \quad i=1,\dots,n_\ell, 
\end{equation*}
namely,
\begin{equation*}
\sum_{\ell=1}^m \sum_{k=1}^{n_\ell}  (K_\ell \eta_{\ell,k})(x_{\ell,i}) 
c_{k+N_{\ell-1}} = g_\ell(x_{\ell,i}),
\end{equation*}
where $\eta_{\ell,k}:=\eta_{k+N_{\ell-1}}$ and the integers $N_\ell$ are
defined in \eqref{Kfj}.
By renumbering the Riesz representers, we obtain the square linear system
\begin{equation*}
\sum_{j=1}^{N_m} (K_\ell \eta_j)(x_{\ell,i}) c_j = g_\ell(x_{\ell,i}),
\qquad \ell=1,\dots,m, \quad i=1,\dots,n_\ell.
\end{equation*}

Taking into account \eqref{Riesz}, the above linear system can 
be written in matrix form as
\begin{equation}
\cG \bm{c} = \bm{g},
\label{linsystem}
\end{equation}
where $\bm{g}$ is defined in \eqref{Kfg} and $\bm{c}=[c_j]_{j=1}^{N_m}$ is the
vector of the unknowns.
The Gram matrix $\cG\in\RR^{N_m\times N_m}$ is defined as
\begin{equation}
\cG=\begin{bmatrix}
\cG^1 & \Gamma^{1,2} & \cdots & \Gamma^{1,m} \\
(\Gamma^{1,2})^T & \cG^2 & & \vdots \\
\vdots & & \ddots & \vdots \\
(\Gamma^{1,m})^T & \cdots & \cdots & \cG^m 
\end{bmatrix},
\label{grammatrix}
\end{equation}
where the entries of the $m$ diagonal blocks $\cG^\ell$,
$\ell=1,\dots,m$, are 
\begin{equation}\label{gramel1}
\cG^\ell_{i j}=\langle \eta_{\ell,i},\eta_{\ell,j} \rangle_W, 
\end{equation}
and the off-diagonal blocks $\Gamma^{\ell,k}$, with $\ell,k=1,\dots,m$,
$k>\ell$, have entries
\begin{equation}\label{gramel2}
\Gamma^{\ell,k}_{i j}=\langle \eta_{\ell,i},\eta_{k, j} \rangle_W,
\end{equation}
for $i=1,\dots,n_\ell$ and $j=1,\dots,n_k$.

The inner products in \eqref{gramel1} and \eqref{gramel2} involve the second
derivatives $\eta_{\ell,i}''$.
Whenever they can be computed analytically, the elements of the Gram matrix
$\cG$ can be obtained by symbolic computation; we used the \texttt{integral}
function of Matlab. If this is not possible, a Gaussian quadrature formula is
adopted.

As it is well-known, the Gram matrix $\cG$ defined in \eqref{grammatrix} is 
symmetric positive definite. Then, a natural approach for solving system
\eqref{linsystem} would be to apply Cholesky factorization. However, 
as this linear system results from the discretization of an
ill-posed problem, the matrix $\cG$ is severely ill-conditioned.
Since experimental data is typically contaminated by noise, the numerical
solution of \eqref{linsystem} is subject to strong error propagation
and can deviate substantially from the exact solution.
Moreover, the numerical computation of the
Cholesky factorization may lead to computing the square root of small negative
quantities, making it impossible to construct the Cholesky factor.

We adopted a different approach, consisting of
writing the Gram matrix in terms of its spectral factorization
\cite{stewart}
\begin{equation}
\cG= U \Lambda U^T,
\label{eig}
\end{equation}
where the diagonal matrix 
$\Lambda=\diag(\lambda_1,\lambda_2,\dots,\lambda_{N_m})$
contains the eigenvalues of $\cG$ 
sorted by decreasing value,
and $U=[\bm{u}_1,\dots,\bm{u}_{N_m}]$ is the eigenvector matrix with
orthonormal columns.

Then, by employing \eqref{eig} in system \eqref{linsystem}, we
obtain the following representation for the coefficients
\begin{equation}\label{etacoef}
\bm{c}=[c_1,\ldots,c_{N_m}]^T 
= \sum_{\ell=1}^{N_m} \frac{\bm{u}_\ell^T\bm{g}}{\lambda_\ell}\bm{u}_\ell,
\end{equation}
of the minimal-norm solution 
\begin{equation}\label{fc}
f^\dagger = \sum_{j=1}^{N_m} c_j \eta_j,
\end{equation}
resulting from Theorem~\ref{teo1}.

\section{Regularized minimal-norm solution}\label{sec:regularization}

The severe ill-conditioning of the matrix $\cG$ produces a strong error
propagation in \eqref{etacoef} and, consequently, in the solution \eqref{fc}.
A regularized solution is needed, instead.

In what follows, it is convenient to write $f^\dagger$ as a linear 
combination of orthonormal functions.
The orthonormalization of a family of functions is a classical topic in
functional analysis.
The properties arising from the orthogonalization of the translates of a given
function, and the connections of such process to the factorization of the
associated Gram matrix have been investigated
in~\cite{gmrs95,gmrs98}, and later generalized to multivariate functions
in~\cite{gmrs00}.
A review of the available algorithms for the spectral factorization of infinite
Gram matrices is contained in \cite{gmrs97}.

The following theorem shows how an orthonormal expansion for the minimal-norm
solution can be constructed by \eqref{eig}, and gives the expression of such
orthonormal functions which are, in fact, the singular 
functions~\cite{Engl,Kress99} of the integral operator $\bm{K}$.

\begin{theorem}\label{theoremort}
The minimal-norm solution $f^\dagger$ of~\eqref{modeloperator} can be written as
a linear combination of orthonormal functions $\widehat{\eta}_\ell$
\begin{equation}\label{etaort}
f^\dagger = \sum_{\ell=1}^{N_m} \widehat{c}_\ell \widehat{\eta}_\ell,
\end{equation}
where
\begin{equation}\label{etaort2}
\widehat{c}_\ell = \frac{\bm{u}_\ell^T \bm{g}}{\sqrt{\lambda_\ell}},
\qquad
\widehat{\eta}_\ell = \sum_{j=1}^{N_m} \frac{u_{j \ell}}{\sqrt{\lambda_\ell}} 
\eta_j,
\qquad \ell=1,\ldots,N_m,
\end{equation}
and $u_{j\ell}$ denotes the $j$th component of the eigenvector $\bm{u}_\ell$
with eigenvalue $\lambda_\ell$ in the spectral factorization~\eqref{eig}.
Moreover, the set of the triplets $\left\{ \sqrt{\lambda_\ell}, 
\widehat{\eta}_\ell, \bm{u}_\ell \right\}$, $\ell=1,\ldots,N_m$, is the
singular system of the operator $\bm{K}$~\eqref{modeloperator}.
\end{theorem}

\begin{proof}
Starting from~\eqref{fc} and \eqref{etacoef}, changing the order of summation, 
we obtain
\[
f^\dagger = \sum_{j=1}^{N_m} c_j \eta_j
= \sum_{j=1}^{N_m} \sum_{\ell=1}^{N_m} \frac{\bm{u}_\ell^T 
\bm{g}}{\lambda_\ell} u_{j \ell} \eta_j
= \sum_{\ell=1}^{N_m} \frac{\bm{u}_\ell^T \bm{g}}{\sqrt{\lambda_\ell}}
\sum_{j=1}^{N_m} \frac{u_{j \ell}}{\sqrt{\lambda_\ell}} \eta_j.
\]
Equation \eqref{etaort} follows by defining $\widehat{c}_\ell$ and
$\widehat{\eta}_\ell$ as in \eqref{etaort2}.

Let us now prove the final statement of the theorem.
It is immediate to verify that the functions $\widehat{\eta}_\ell$,
$\ell=1,\ldots,N_m$, form an orthonormal basis for $\cN(\bm{K})^\perp$.
Indeed, letting $\cG_{ij}= \langle \eta_i, \eta_j \rangle_W$ be 
the elements of $\cG$, we have
\begin{align*}
\langle \widehat{\eta}_k,\widehat{\eta}_h \rangle_W 
&= \sum_{i=1}^{N_m} \sum_{j=1}^{N_m}
\frac{u_{ik}}{\sqrt{\lambda_k}} \frac{u_{jh}}{\sqrt{\lambda_h}}
\langle \eta_i, \eta_j \rangle_W
= \frac{1}{\sqrt{\lambda_k \lambda_h}} \sum_{i=1}^{N_m} u_{i k}
\sum_{j=1}^{N_m} \cG_{ij} u_{j h} \\
&= \frac{1}{\sqrt{\lambda_k \lambda_h}} (U^T \cG U)_{kh}
= \frac{1}{\sqrt{\lambda_k \lambda_h}} \Lambda_{kh} = \delta_{kh},
\end{align*}
where $\delta_{kh}$ is the Kronecker delta and, in the last equality, the
matrix $\cG$ is replaced by its spectral decomposition~\eqref{eig}.
The orthonormality of the vectors $\bm{u}_\ell$, $\ell=1,\ldots,N_m$,
immediately follows from factorization~\eqref{eig}.

From the definition~\eqref{Riesz} of the Riesz representers, we can write
\begin{align*}
(\bm{K}\widehat{\eta}_\ell)_j & = \langle \eta_j, \widehat{\eta}_\ell\rangle_W
= \sum_{s=1}^{N_m} \frac{u_{s\ell}}{\sqrt{\lambda_\ell}} \langle \eta_j, 
\eta_s\rangle_W
= \sum_{s=1}^{N_m} \frac{u_{s\ell}}{\sqrt{\lambda_\ell}} \cG_{js}
= \frac{1}{\sqrt{\lambda_\ell}} (\cG U)_{j\ell}\\
& = \frac{1}{\sqrt{\lambda_\ell}} (U \Lambda)_{j\ell}
= \sqrt{\lambda_\ell} \, u_{j\ell}, \qquad j=1,\ldots,N_m,
\end{align*}
where the spectral factorization~\eqref{eig} of $\cG$ is employed again.
Then, $\bm{K}\widehat{\eta}_\ell=\sqrt{\lambda_\ell}\bm{u}_\ell$.

Now, let $f\in W$. Then, $f=f_0+f_1$, with $f_0\in\cN(\bm{K})$, 
$f_1\in\cN(\bm{K})^\perp$, and
$$
f_1 = \sum_{j=1}^{N_m} \alpha_j \widehat{\eta}_j,
\qquad\text{with } \alpha_j=\langle f_1, \widehat{\eta}_j\rangle_W.
$$
By the definition~\eqref{propadjoint} of the adjoint operator, we obtain
\begin{align*}
\langle f, \bm{K}^*\bm{u}_\ell \rangle_W
& = \langle \bm{K}f, \bm{u}_\ell \rangle_2
= \langle \bm{K}f_1, \bm{u}_\ell \rangle_2
= \sum_{j=1}^{N_m} \alpha_j \langle \bm{K} \widehat{\eta}_j, \bm{u}_\ell 
\rangle_2 \\
&= \sum_{j=1}^{N_m} \alpha_j \left\langle \sqrt{\lambda_j} \bm{u}_j,
\bm{u}_\ell \right\rangle_2 
= \alpha_\ell \sqrt{\lambda_\ell}
= \langle f, \sqrt{\lambda_\ell} \widehat{\eta}_\ell \rangle_W,
\end{align*}
since $\alpha_\ell=\langle f_1, \widehat{\eta}_\ell\rangle_W
=\langle f, \widehat{\eta}_\ell\rangle_W$.
Then $\bm{K}^*\bm{u}_\ell=\sqrt{\lambda_\ell}\widehat{\eta}_\ell$.
It follows that
$$
\bm{K}^*\bm{K}\widehat{\eta}_\ell = \lambda_\ell \widehat{\eta}_\ell,
\qquad
\bm{K}\bm{K}^*\bm{u}_\ell = \lambda_\ell \bm{u}_\ell,
\qquad \ell=1,\ldots,N_m.
$$
This completes the proof.
\end{proof}

We remark that Theorem~\ref{theoremort} is applicable under the assumption
that the Gram matrix $\cG$ is positive definite. 
In practice, because of error propagation, the smallest numerical eigenvalues
of $\cG$ may become zero, or even negative. In this case, that is, if
$\lambda_{N_m}\leq 0$, we replace $N_m$ in all summations by an integer
$N<N_m$ such that $\lambda_N>0\geq \lambda_{N+1}$.

From~\eqref{etaort} and from the definition of $\widehat{\bm{c}}$ 
in~\eqref{etaort2}, it follows that
\begin{equation}\label{normfc}
\|f^\dagger\|_W = \|\widehat{\bm{c}}\|_2 = \|L\bm{c}\|_2,
\qquad
\text{with } L=\Lambda^{1/2}U^T,
\end{equation}
where the relation between $\bm{c}$ and 
$\widehat{\bm{c}}$ is obtained by \eqref{etaort2},
writing $\widehat{\bm{c}}$ in matrix form 
\[
\widehat{\bm{c}} = \Lambda^{-1/2}U^T\bm{g}
= \Lambda^{-1/2}U^T \cG\bm{c}
= \Lambda^{-1/2}U^T U\Lambda U^T \bm{c}
= \Lambda^{1/2} U^T \bm{c}.
\]
This expression for $\widehat{\bm{c}} $ is equivalent to solving the linear
system $\widehat{\cG}\, \widehat{\bm{c}} = \bm{g}$,
whose coefficient matrix $\widehat{\cG}=U\Lambda^{1/2}$ has a condition number
which is the square root of that of $\cG$.

\begin{figure}[ht] \centering
\includegraphics[width=.46\textwidth]{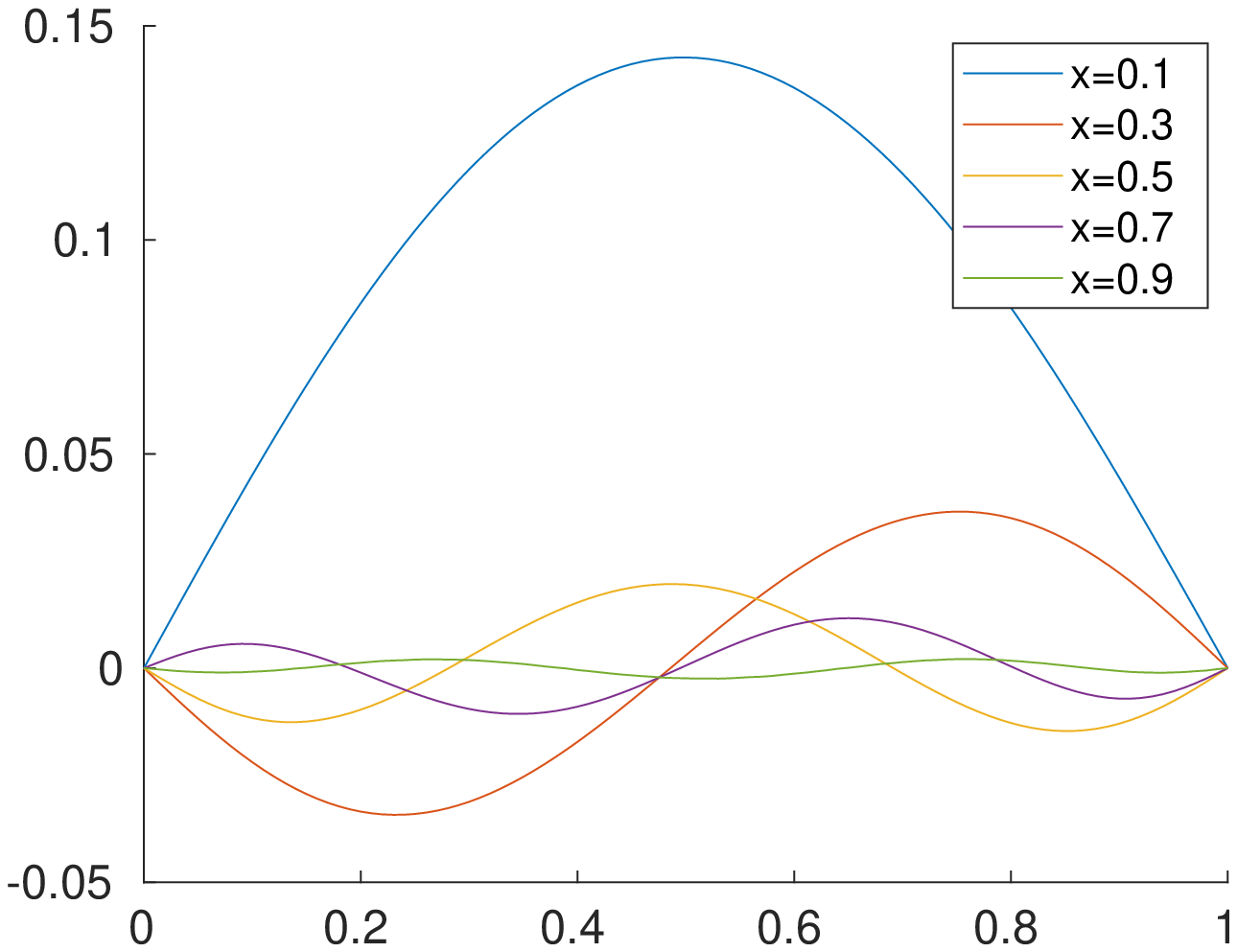}\hfill
\includegraphics[width=.46\textwidth]{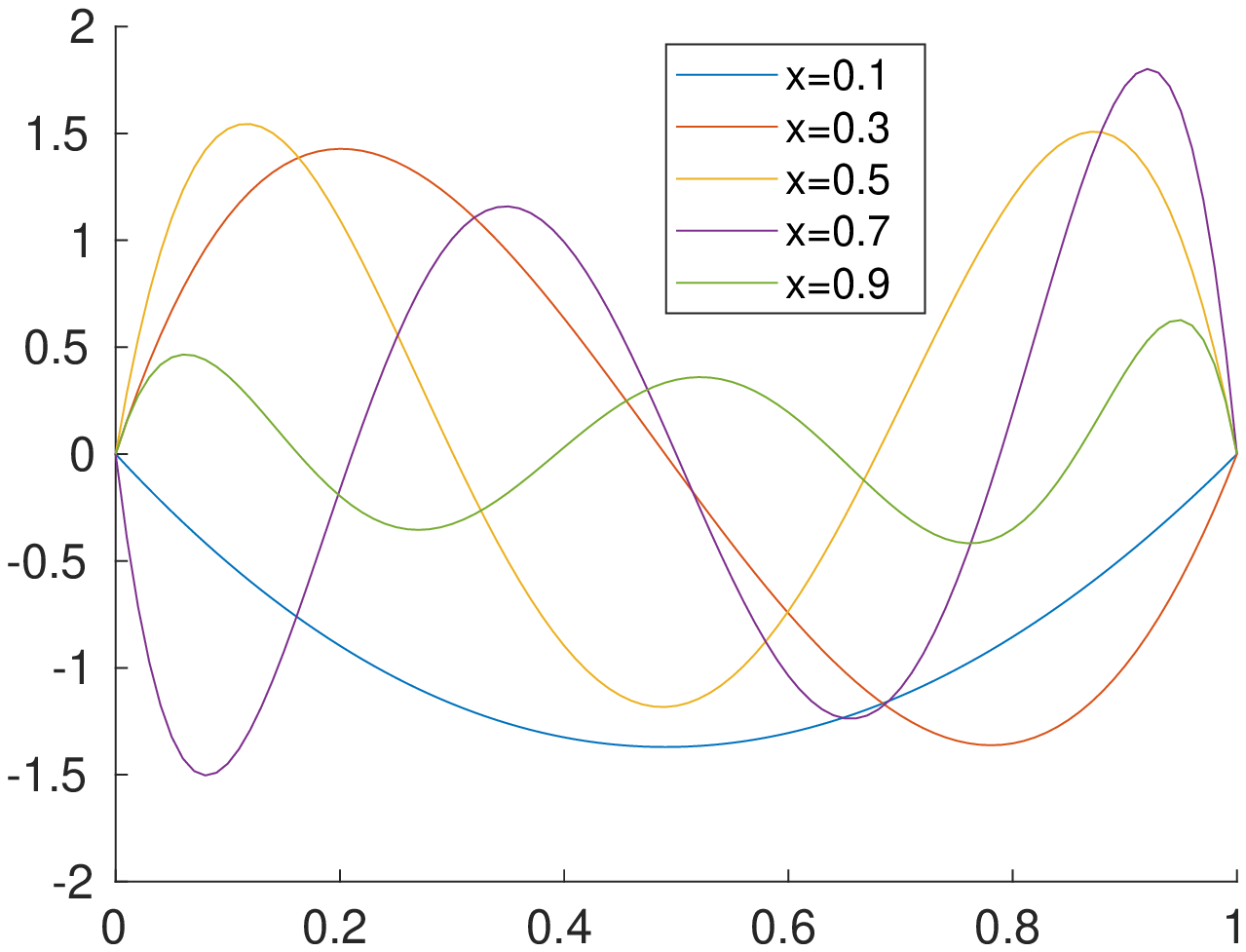}\\ \vspace{.5cm}
\includegraphics[width=.46\textwidth]{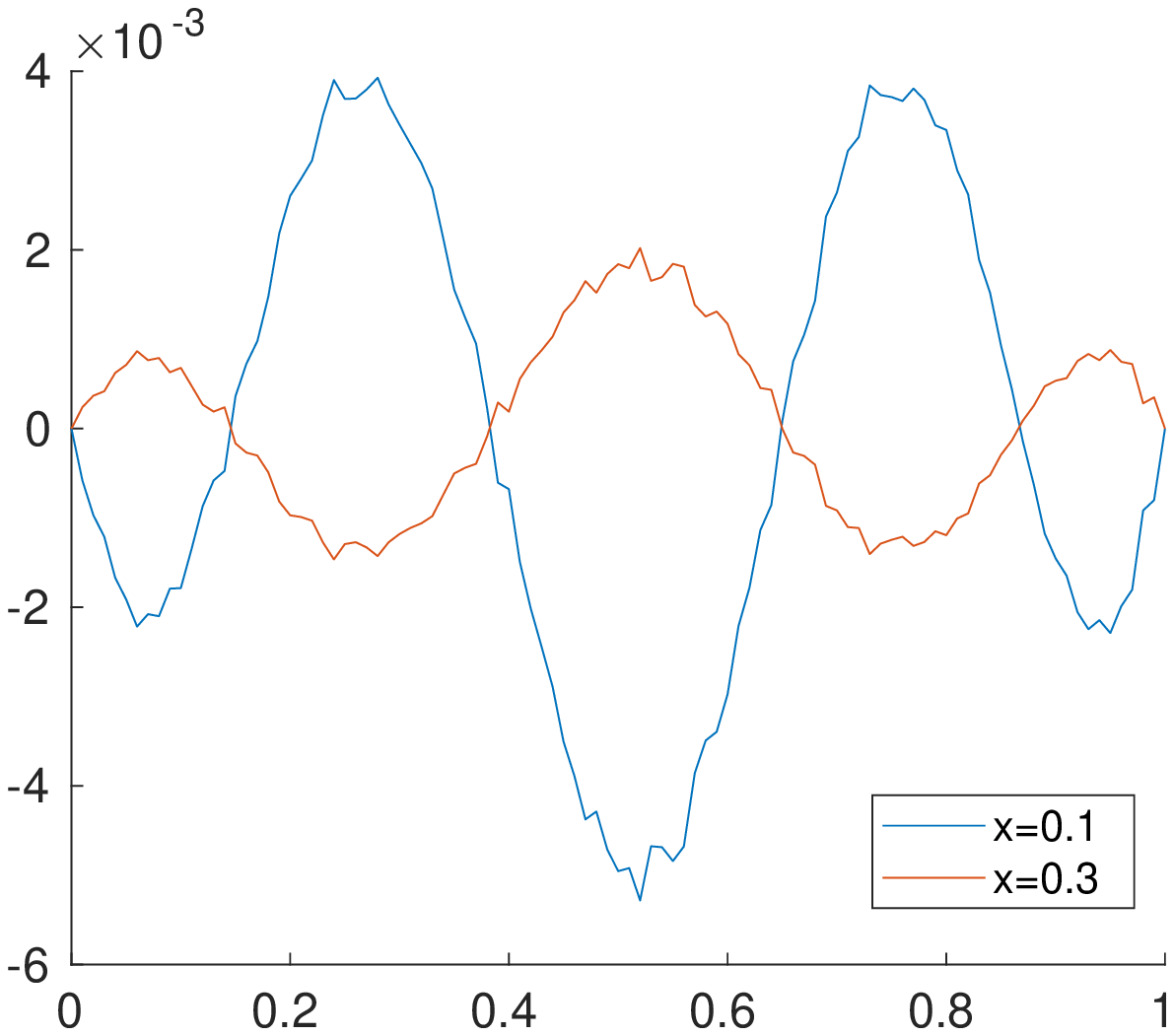}\hfill
\includegraphics[width=.46\textwidth]{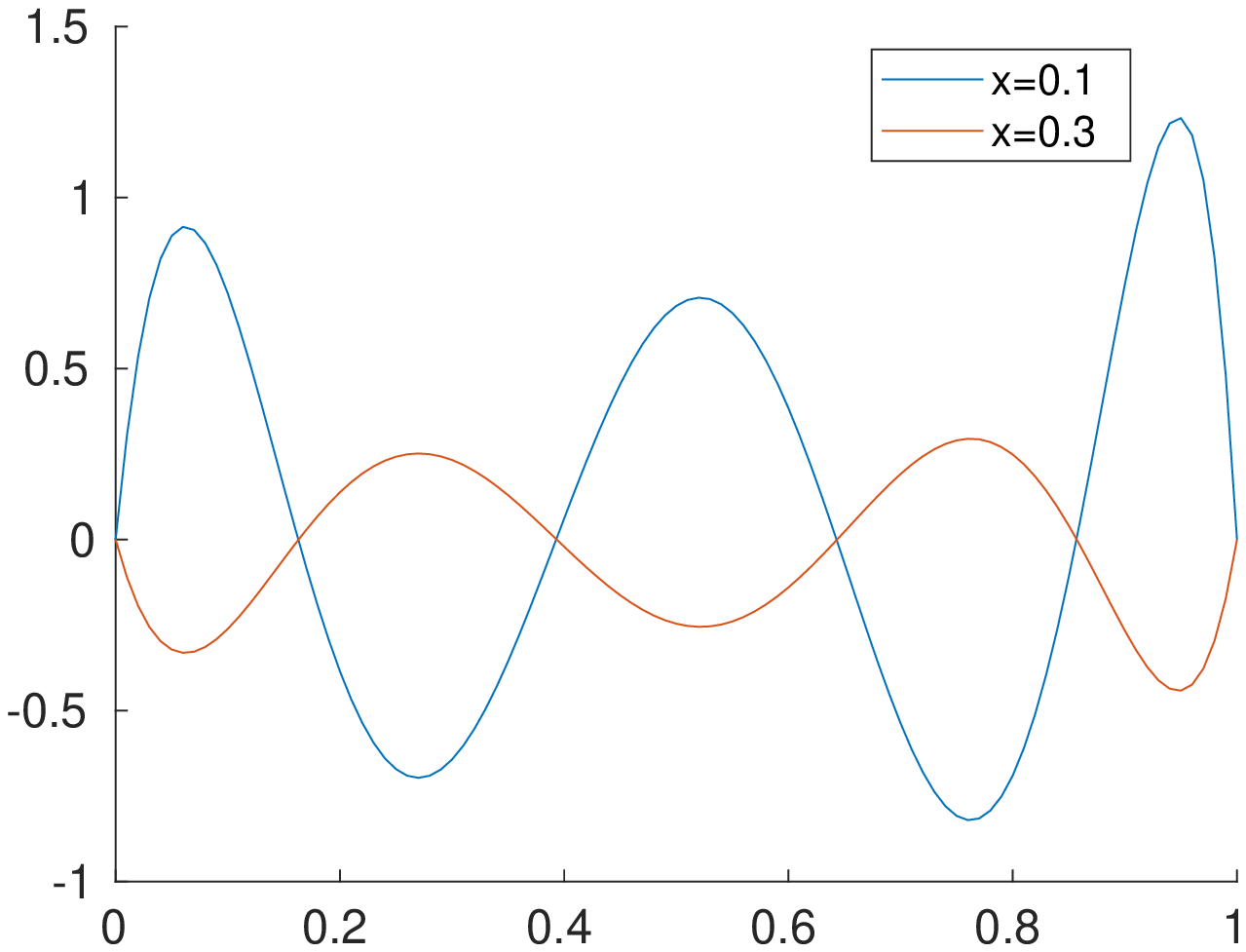}
\caption{Orthonormalized Riesz functions for the system \eqref{ex1}:
$\widehat{\eta}_{1,i}$ (top-left) and $\widehat{\eta}''_{1,i}$ (top-right) for
$x_i = 0.1+0.2(i-1)$, $i = 1,\ldots,5$; $\widehat{\eta}_{2,i}$ (bottom-left)
and $\widehat{\eta}''_{2,i}$ (bottom-right) are displayed only for $x_1$ and
$x_2$.}
\label{fig1cap}
\end{figure}

In order to handle ill-conditioning, as it is customary, we replace the original
problem with a nearby one, whose solution is less sensitive to the error present
in the data. 
The representation \eqref{etaort} is particularly suitable to construct a
regularized solution.
Indeed, according to the ``discrete'' Picard
condition~\cite[Section~4.5]{Hansen}, the numerators in the coefficients
$\widehat{c}_\ell$ should decay to zero faster than the denominators.
Anyway, the presence of noise in the right-hand side $\bm{g}$ will prevent
the projections $\bm{u}_\ell^T\bm{g}$ from decaying when $\ell$ increases,
leading to severe growth in the values of the coefficients.
Truncating the summation in \eqref{etaort} removes the noisy components of
the solution that are enhanced by ill-conditioning.
Moreover, it damps the high frequency components represented by the
$\widehat{\eta}_\ell$ functions with a large index $\ell$.

The association of high frequencies to small singular values cannot be
proved in general. However, in the case of first kind integral equations with a
smooth kernel, it has been observed that singular functions associated with
the smallest singular values oscillate much, while those corresponding to
large singular values are smooth.
For example, Figure~\ref{fig1cap} displays the functions $\widehat{\eta}_\ell$
obtained by applying formula \eqref{etaort2} to the Riesz functions constructed
in Example~\ref{example1}.
In the summation \eqref{etaort2}, the upper bound for the index is fixed at
$N=7$, to preserve the positivity of the eigenvalues.
The graphs in the left column depict the orthonormal functions, and the ones in
the right column their second derivatives.
It is immediate to observe the increasing frequency of the orthonormal basis.

The graphs of the functions $\widehat{\eta}_{2,1}$ and $\widehat{\eta}_{2,2}$
in the bottom-left panel of Figure~\ref{fig1cap} are extremely jagged,
showing that there is a strong error propagation in the numerical
construction of the orthonormal functions.
This deters from employing the orthonormal basis in the real computation,
unless a more stable orthonormalization process is implemented.
Anyway, as we will show, the functions $\widehat{\eta}_\ell$ are only
implicitly used in the construction of the regularized solution.

Indeed, the regularized solution is obtained by choosing an index $\kappa$ 
to truncate the summation in~\eqref{etaort}, i.e., $1\leq \kappa\leq N$, leading
to the expression
\begin{equation}\label{freg}
f^{(\kappa)} = \sum_{\ell=1}^{\kappa} \widehat{c}_\ell \widehat{\eta}_\ell
= \sum_{\ell=1}^{\kappa} \frac{\bm{u}_\ell^T \bm{g}}{\sqrt{\lambda_\ell}}
\sum_{j=1}^{N} \frac{u_{j\ell}}{\sqrt{\lambda_\ell}} \eta_j
= \sum_{j=1}^{N} \sum_{\ell=1}^{\kappa} \frac{\bm{u}_\ell^T 
\bm{g}}{\lambda_\ell} u_{j \ell} \eta_j
=\sum_{j=1}^{N} c^{(\kappa)}_j \eta_j.
\end{equation}
This shows that $f^{(\kappa)}$ can be expressed as a linear combination of the
Riesz representers $\eta_j$ and there is no need to explicitly construct the
singular functions $\widehat{\eta}_\ell$.

The coefficients in the last summation correspond to the \emph{truncated
eigendecomposition} (TEIG) solution of system \eqref{linsystem}
(see~\cite{agrr21,gorr16} for more details) with parameter $\kappa\leq N$,
defined to be the components of the vector
\begin{equation}
\bm{c}^{(\kappa)} = U\Lambda_\kappa^\dagger U^T \bm{g} = \sum_{\ell=1}^\kappa
\frac{\bm{u}_\ell^T\bm{g}}{\lambda_\ell} \bm{u}_\ell,
\label{truncated}
\end{equation}
where $\Lambda_\kappa^\dagger$ denotes the Moore-Penrose
pseudoinverse~\cite{bjo96} of
$\Lambda_\kappa=\diag(\lambda_1,\dots,\lambda_\kappa,0,\ldots,0)$.
We observe that, because of the orthonormality of the functions
$\widehat{\eta}_\ell$, $\|f^{(\kappa)}\|_W \leq \|f^{(\kappa+1)}\|_W \leq 
\|f^\dagger\|_W$.

It is possible to show that the above vector $\bm{c}^{(\kappa)}$ solves the
optimization problem
$$
\begin{cases}
\displaystyle \min_{\bm{c}} \|L\bm{c}\|_2 \\
\displaystyle \bm{c}\in \bigl\{ \arg \min_{\bm{c}} 
\|\cG_\kappa \bm{c}-\bm{g}\|_2 \bigr\},
\end{cases}
$$
where $\cG_\kappa=U\Lambda_\kappa U^T$ is the TEIG of $\cG$.
Therefore, from the algebraic point of view, the computation of $f^{(\kappa)}$
corresponds to selecting the minimal-$L$-norm vector among the solutions of the
best rank-$\kappa$ approximation of system \eqref{linsystem}.
Equation \eqref{normfc} shows that the regularized solution $f^{(\kappa)}$
has minimal-norm in $W$.

A crucial point in the regularization process, in order to get an 
accurate solution, is the estimation of the truncation parameter $\kappa$ in 
\eqref{freg} and \eqref{truncated}.
There exist many methods, either a posteriori or heuristic, aiming at this;
see \cite{Engl,Hansen,rr13}. 
In this paper, we focus our attention on the discrepancy principle and the
L-curve method. 

We assume that the exact right-hand side vector $\bm{g}_\text{exact}$ is
contaminated by an unknown normally distributed noise vector $\bm{e}$, i.e., 
\begin{equation}\label{noise1}
\bm{g} = \bm{g}_\text{exact}+\bm{e}.
\end{equation}
If $\|\bm{e}\|_2$ is known, we can apply the discrepancy principle
\cite{Morozov}, which selects the smallest truncation parameter
$\kappa_\text{d}$ such that
\begin{equation}\label{discrepancy}
\|\cG\bm{c}^{(\kappa_\text{d})}-\bm{g}\|_2 \leq \tau \|\bm{e}\|_2, 
\end{equation}
where $\tau>1$ is a constant independent of the noise level $\|\bm{e}\|_2$.
Note that from \eqref{eig} and \eqref{truncated}, we can write the residual
norm as
\begin{equation}\label{residual}
\|\cG\bm{c}^{(\kappa)}-\bm{g}\|_2^2 
= \|U(\Lambda\Lambda_\kappa^\dagger-I) U^T \bm{g}\|_2^2
= \sum_{j=\kappa+1}^{N_m} (\bm{u}_j^T\bm{g})^2.
\end{equation}
This relation shows that the residual is non-decreasing when $\kappa$ decreases
and it allows reducing the computational load.
Indeed, the projected vector $U^T\bm{g}$ is computed in any case, once the
spectral factorization of $\cG$ is available, since its first $\kappa$
components are required for \eqref{truncated}, but the value of $\kappa$ is not
a priori known.

When the noise level is unknown, we use the L-curve criterion \cite{Ha0,HO},
which selects the regularization parameter $\kappa_\text{lc}$ at the ``corner''
of the curve joining the points
\begin{equation}\label{lcurve}
\left( \log \| \cG\bm{c}^{(\kappa)}-\bm{g} \|_2 , \; 
\log \| f^{(\kappa)} \|_W \right), \qquad \kappa=1,\ldots,N,
\end{equation}
where $f^{(\kappa)}$ is the function defined in \eqref{freg} and
$$
\|f^{(\kappa)}\|_W=\|L\bm{c}^{(\kappa)}\|_2
=\sqrt{(\bm{c}^{(\kappa)})^T \cG \bm{c}^{(\kappa)}}.
$$
When solving discrete ill-posed problems, this curve often exhibits a typical
L-shape. We determine its corner by the method described in~\cite{hjr07} and
implemented in \cite{Hansentools}.

When the exact solution $f$ is available, to ascertain the best possible
performance of the algorithms independently of the strategy adopted for the
estimation of the regularization parameter, in the numerical experiments we
also consider the parameter $\kappa_\text{best}$ which minimizes the norm of
the error, that is,
\begin{equation}\label{kbest}
\kappa_\text{best} = \arg \min_\kappa \|f-f^{(\kappa)}\|_W
= \arg \min_\kappa \|L(\bm{c}-\bm{c}^{(\kappa)})\|_2.
\end{equation}

\begin{remark}\rm
We observe that the operator $\cF_\text{d}$, which assigns to a noisy right-hand
side $\bm{g}$ (see~\eqref{noise1} and~\eqref{noise}) the regularized solution
$f^{(\kappa_\text{d})}$ \eqref{freg} corresponding to the regularization
parameter $\kappa_\text{d}=\kappa_\text{d}(\delta,\bm{g})$ estimated by the
discrepancy principle, is trivially a regularization method in the sense
of~\cite[Definition~3.1]{Engl}.
Indeed, from \eqref{discrepancy} and \eqref{residual}, $\kappa_\text{d}=N_m$
when $\delta\to 0$, and $f^{(N_m)}$ coincides with the minimal-norm solution
$f^\dagger$.
\end{remark}

\section{Numerical tests}\label{sec:tests}

In this section, we report some numerical results obtained by applying our
algorithm to the two examples presented in Section~\ref{sec:method}. 
All the computations were performed on an Intel Xeon E-2244G system with
16Gb RAM, running Matlab 9.10.
The software developed is only prototypal, but it is available from the
authors upon request.

In each numerical test, we consider the exact right-hand side
${\bm{g}}_\text{exact}$ of the linear system \eqref{linsystem},
corresponding to the collocation nodes $x_{\ell,i}$, for $\ell = 1,\ldots,m$
and $i=1,\ldots,n_\ell$. We add Gaussian noise as in \eqref{noise1}, 
where the noise vector $\bm{e}$ is defined by
\begin{equation}\label{noise}
\bm{e} = \frac{\delta}{\sqrt{N_m}}\|{\bm{g}}_\text{exact}\|_2 \bm{w},
\end{equation}
with $N_m$ as in \eqref{Kfj}.
The components of the vector $\bm{w}$ are normally distributed with zero
average and unit variance, and $\delta$ represents the noise level.
For the sake of simplicity, for each system we consider the same collocation
nodes $x_{\ell,i}$ in both equations, so that $m=2$, $n_1=n_2=n$, $N_m=2n$, and 
$x_{1,i}=x_{2,i}$, for $i=1,\dots,n$.

\paragraph{Test problem 1.}\label{sec:test1}

We consider the system \eqref{ex1} described in Example~\ref{example1}.
It consists of two Fredholm integral equations of the
first kind, with $x\in(0,1]$ and exact solution $f(t)=t^2+1$.
In this example we set $x_{\ell,i}= 0.1 + 0.9\,(i-1)/(n-1)$, for $\ell = 1, 
2$ and $i = 1,\ldots,n$.

The corresponding Riesz representers have been computed analytically in 
\eqref{eta1} and \eqref{eta2}.
Note that the analytic expression of $\eta''_{\ell,i}$ defined in 
\eqref{eta1S} and \eqref{eta2S} allows for an accurate computation of the
elements of the Gram matrix \eqref{grammatrix} and for obtaining an explicit
representation of the functions $\eta_{\ell,i}$, providing a fast and accurate
algorithm.

We remind the reader that, by \eqref{xi}, the solution of this problem is
expressed as
\[
f(t) = \xi(t) + \gamma(t),
\]
where $\gamma(t)=t+1$ is the function \eqref{gamma} 
and $\xi(t)$ is the solution of the system \eqref{model3}.

To start with, we depict in Figure~\ref{fig2} the non-regularized 
reconstructions of the solution, obtained for $n=5,10,20$, without noise in
the data, and the corresponding error curves with respect to the exact
solution. 
By ``non-regularized'', we mean that we set $\kappa=N$ in \eqref{freg} and
\eqref{truncated}.
The fact that the errors are so small is remarkable.
Indeed, setting $\delta=0$ in \eqref{noise} only guarantees that the right-hand
side is accurate up to machine precision, that is, roughly $10^{-16}$.
Since the estimation of the condition number of the Gram matrix $\cG$ provided
by the \texttt{cond} function of Matlab for the three problem sizes considered
is $2.2\cdot 10^{18}$, $6.9\cdot 10^{32}$, and $1.1\cdot 10^{19}$,
respectively, the results highlight the stability in the computation,
as well as the effectiveness of the function space setting.

\begin{figure}[ht]
\centering
\includegraphics[width=.47\textwidth]{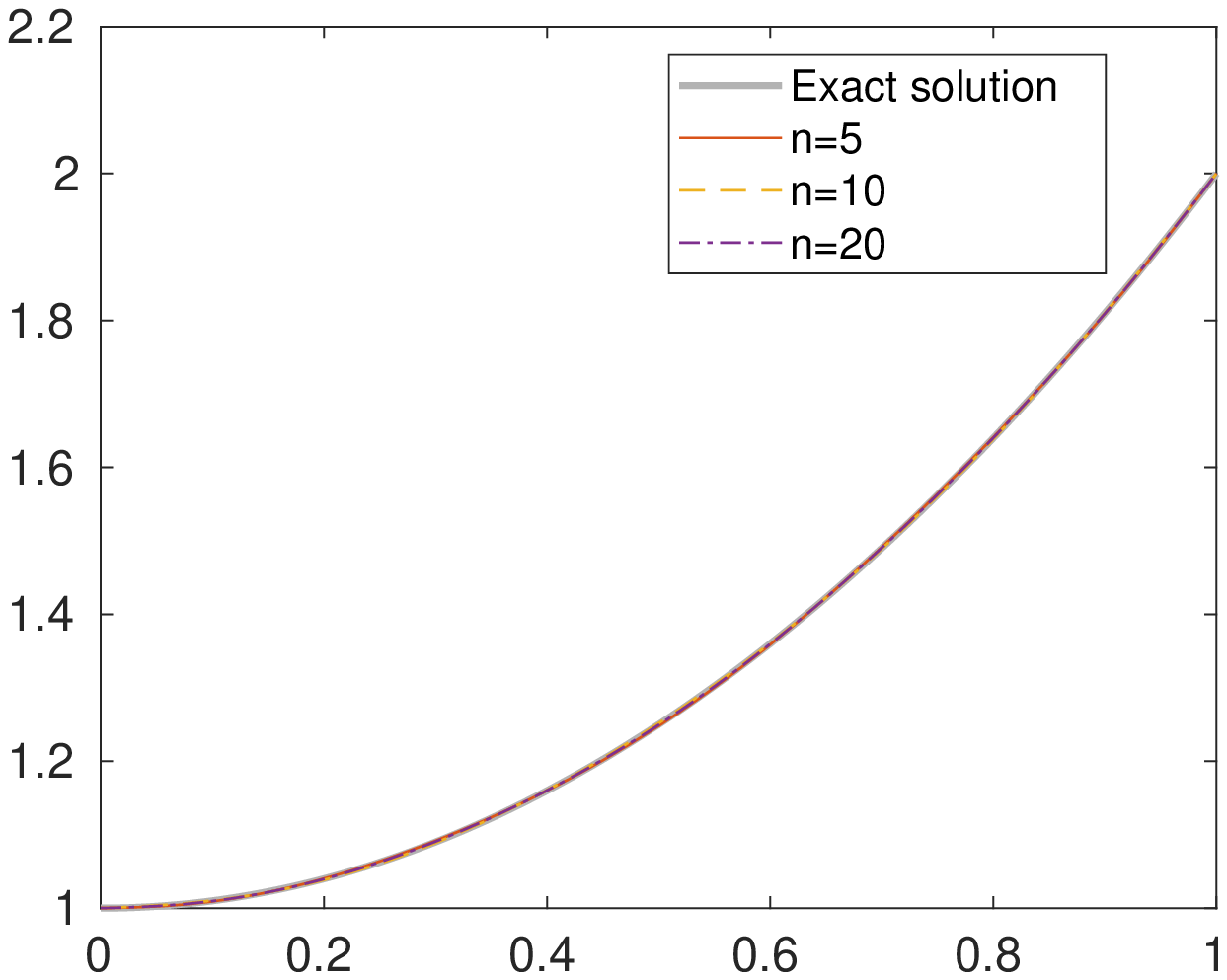}\hfill
\includegraphics[width=.46\textwidth]{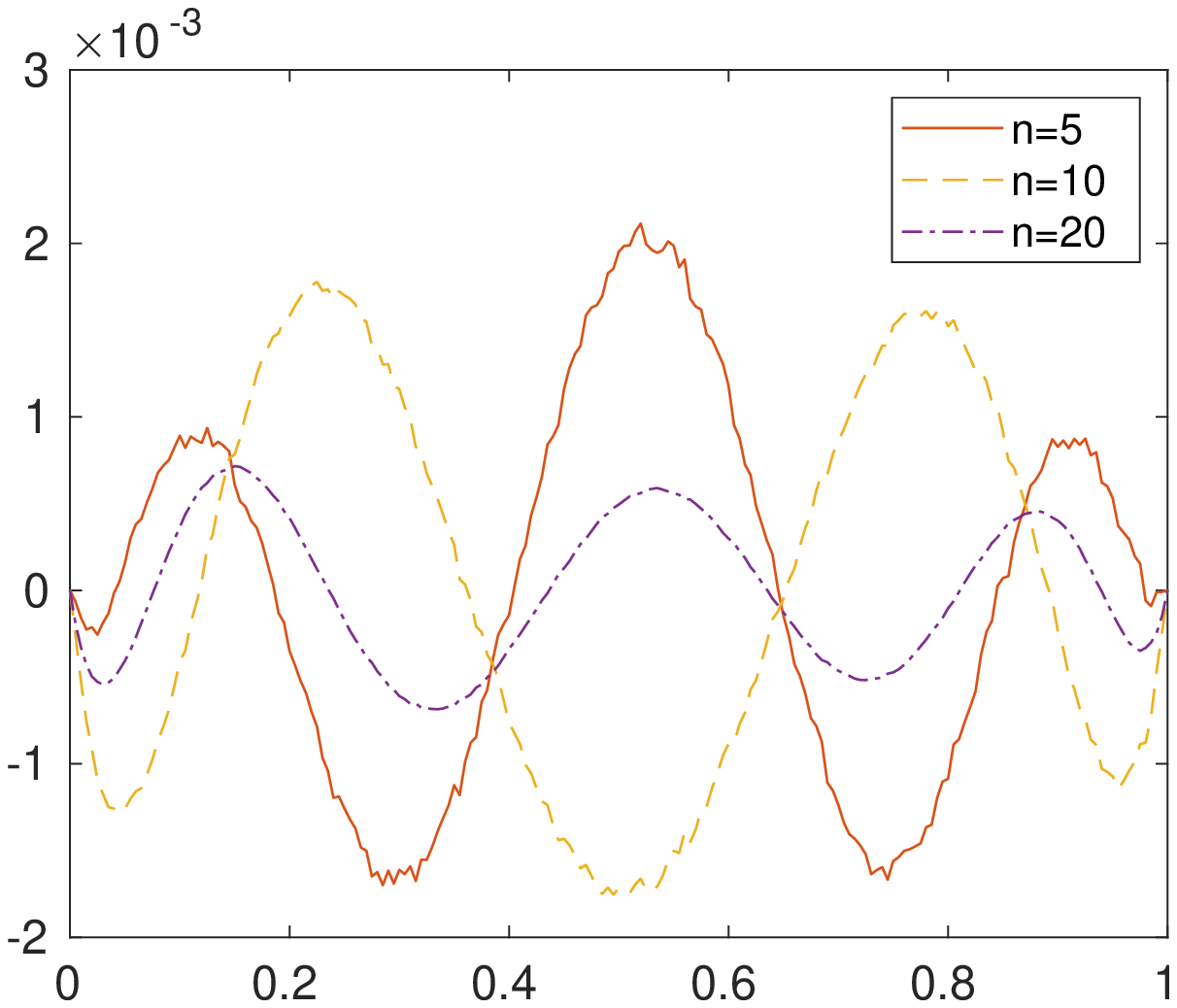}
\caption{Non-regularized reconstructions of the solution of Test problem~1
(left) and corresponding errors (right), for $n=5,10,20$, and without noise.}
\label{fig2}
\end{figure}

Figure~\ref{fig3} shows, in the left pane, the reconstructions obtained without
regularization for $n=5,10,20$, together with the exact solution, when the data
vector is affected by noise with level $\delta=10^{-4}$.
Due to the large condition number, the computed solutions are polluted by noise
propagation to such a point that they oscillate at high frequency away from the
exact solution.
The graph on the right of the same figure displays the results obtained by
computing the regularized solution $f^{(\kappa)}$ defined in \eqref{freg}.
Here, the truncation parameter $\kappa$ coincides with the value
$\kappa_\text{best}$, defined in \eqref{kbest}, corresponding to the best
possible performance of the algorithm.
The quality of the results is excellent.

\begin{figure}[ht]
\centering
\includegraphics[width=.45\textwidth]{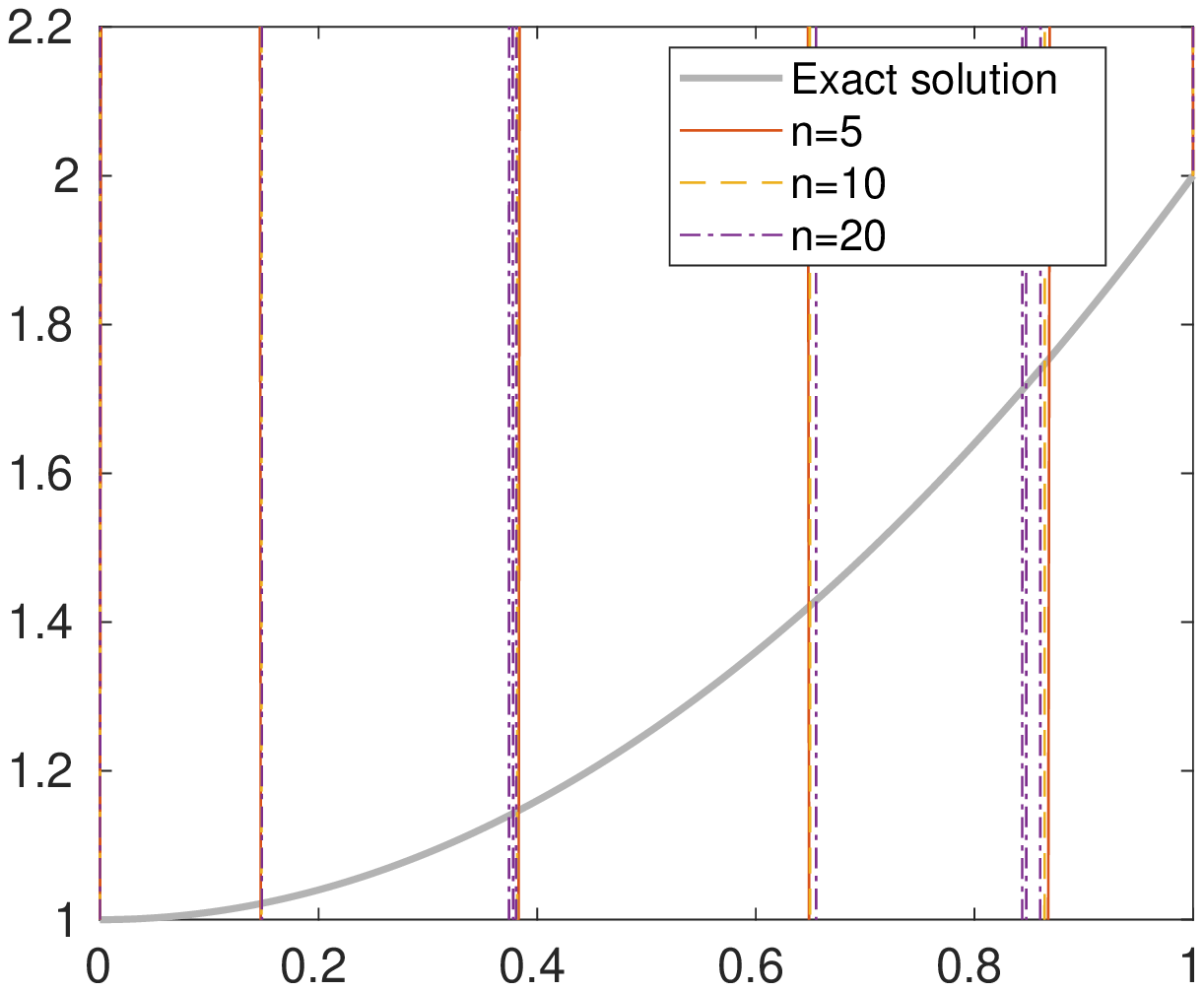}\hfill
\includegraphics[width=.49\textwidth]{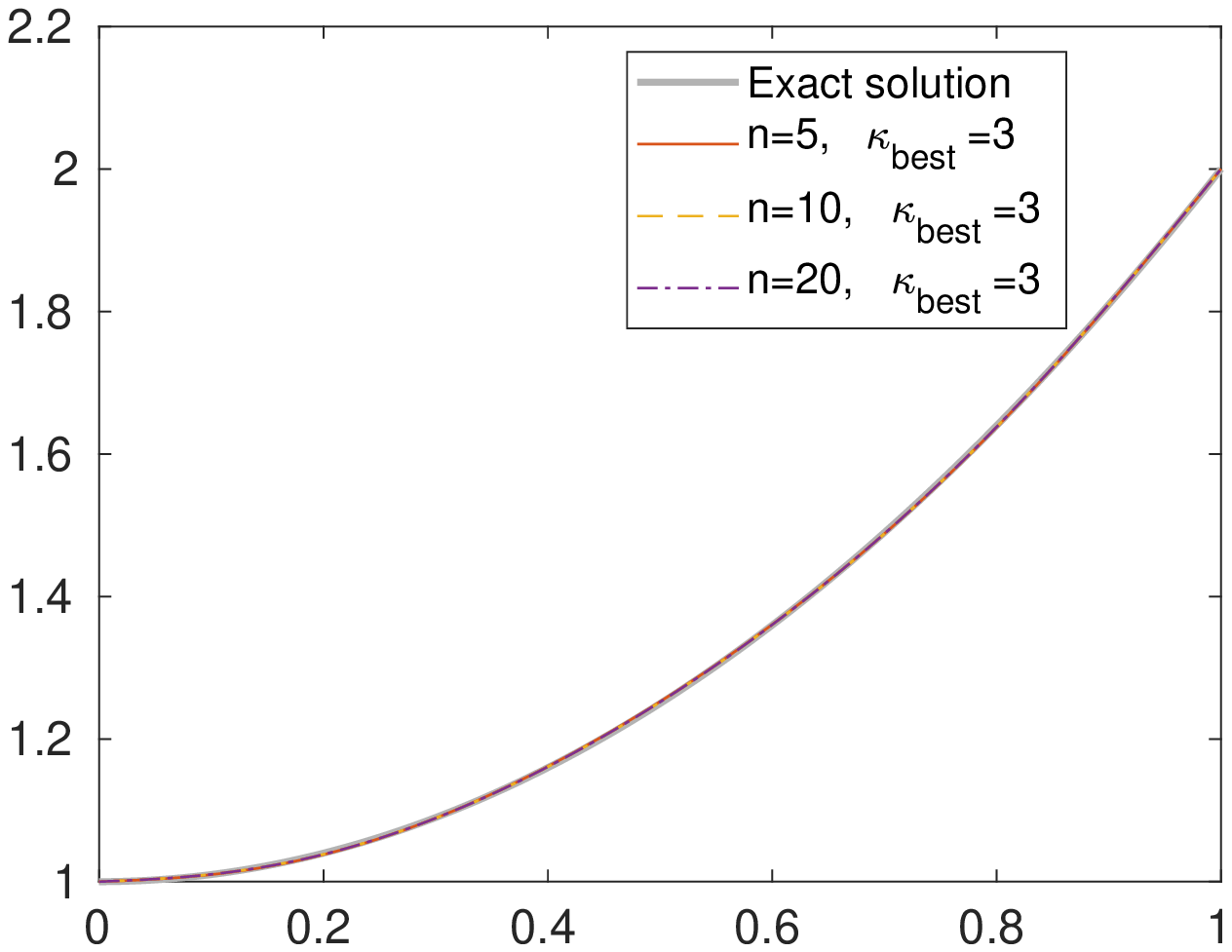}
\caption{On the left: non-regularized solutions of Test problem~1, for
$n=5,10,20$, and noise level $\delta=10^{-4}$. On the right: regularized
solution $f^{(\kappa_\text{best})}(t)$, for $n=5,10,20$, and $\delta=10^{-4}$;
the optimal value $\kappa_\text{best}$ of the regularization parameter is
displayed in the legend.}
\label{fig3}
\end{figure}

The graph on the left of Figure~\ref{fig4} investigates the sensitivity of the
solution to the noise level. It shows the errors obtained for $n=10$ and
$\delta=10^{-8},10^{-4},10^{-2}$. The graph confirms the accuracy and stability
of the proposed regularization method.
In the graph on the right, we compare the ``best'' solution for the noise level 
$\delta=10^{-4}$ to the ones obtained by estimating the regularization
parameter $\kappa_\text{d}$ by the discrepancy principle \eqref{discrepancy},
with $\tau=1.1$, and by the L-curve criterion \eqref{lcurve}, where the
truncation parameter $\kappa_\text{lc}$ is detected by the algorithm described
in \cite{hjr07}. Both estimation techniques are successful.

\begin{figure}[ht]
\centering
\includegraphics[width=.49\textwidth]{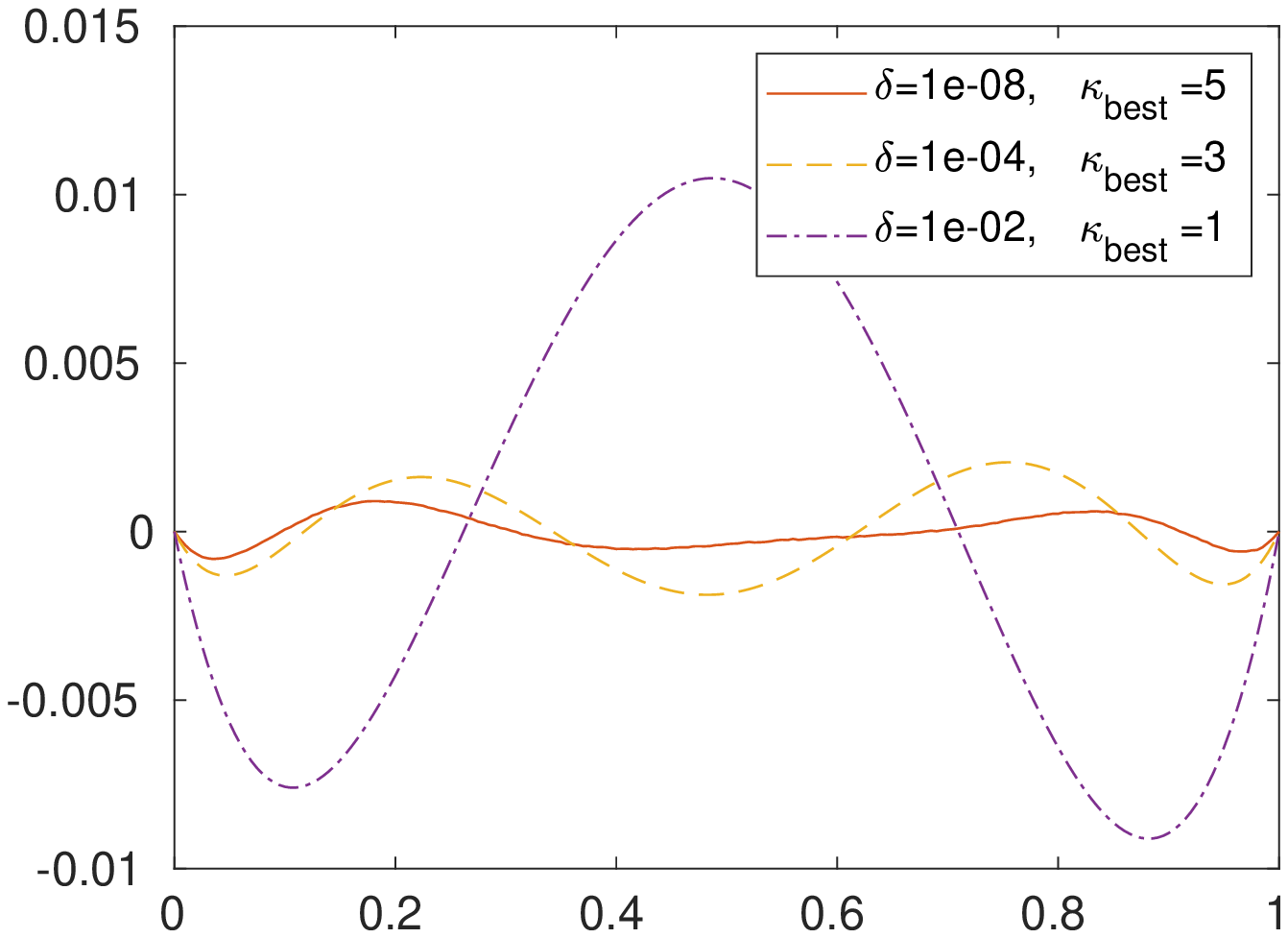}\hfill
\includegraphics[width=.45\textwidth]{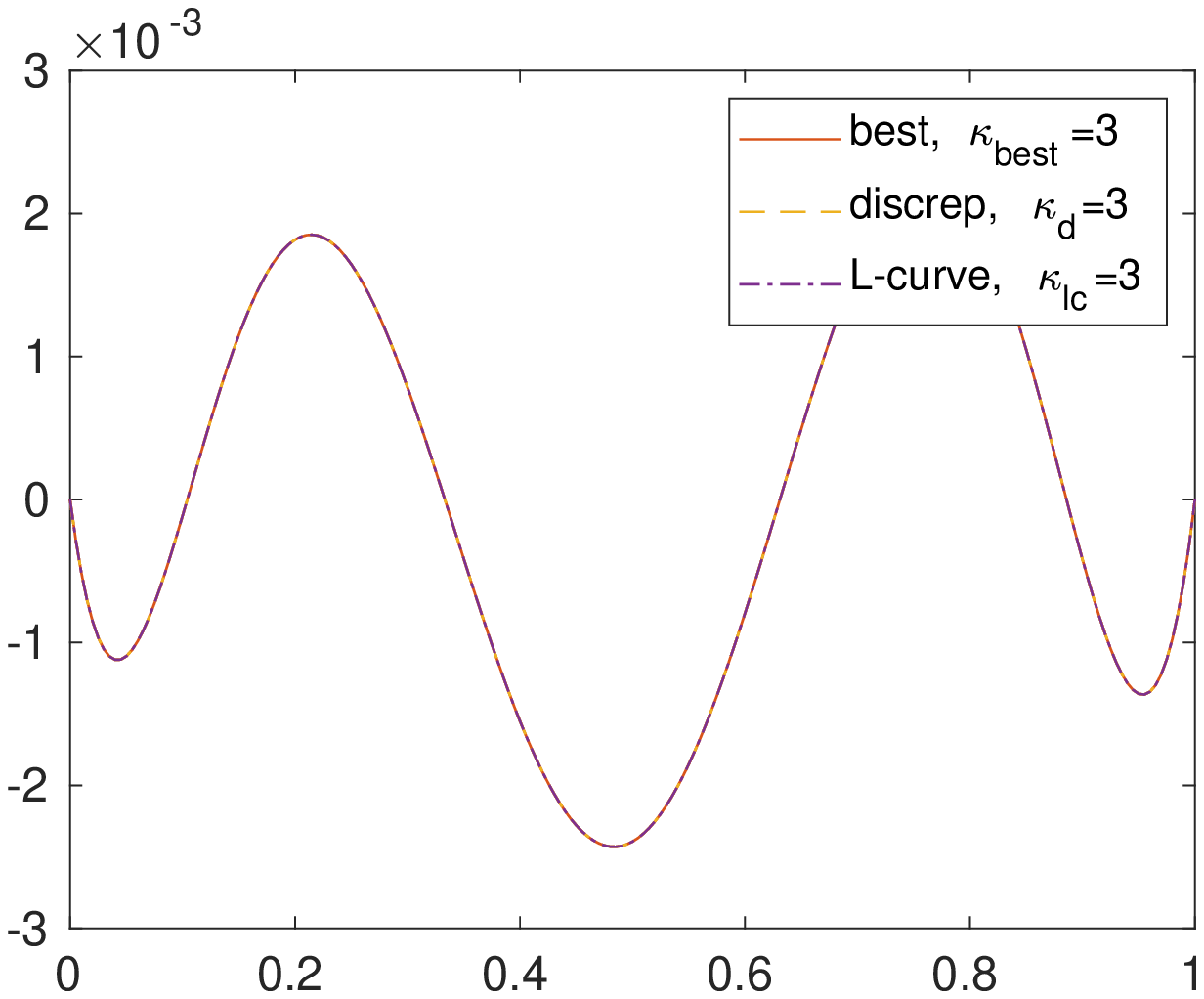}
\caption{On the left: errors corresponding to the regularized solutions
$f^{(\kappa_\text{best})}(t)$ of Test problem~1, for $n=10$ and
$\delta=10^{-8},10^{-4},10^{-2}$. On the right: errors for the solutions
$f^{(\kappa)}(t)$, for $n=20$, $\delta=10^{-4}$, and different estimation
methods for $\kappa$. The values of the regularization parameters
$\kappa_\text{best}$, $\kappa_\text{d}$, and $\kappa_\text{lc}$ are displayed
in the legend.}
\label{fig4}
\end{figure}

\paragraph{Test problem 2.}\label{sec:test2}

Let us now consider the system \eqref{sysbaart} introduced in
Example~\ref{exbaart}, with $x\in(0,\pi/2]$.
It pairs the well-known Baart test problem \cite{Hansentools} to an equation
having the same solution $f(t)=\sin t$.
The collocation points are $x_{\ell,i}=0.1+(\pi/2-0.1)\,(i-1)/(n-1)$, for
$\ell=1,2$ and $i=1,\ldots,n$.

In this example, we were only able to analytically compute the Riesz
representers for the second equation; see \eqref{eta2Sex2} and \eqref{eta2ex2}.
An approximation of the Riesz representers for the first equation was
computed by a Gauss-Legendre quadrature formula. 

\begin{figure}[ht]
\centering
\includegraphics[width=.48\textwidth]{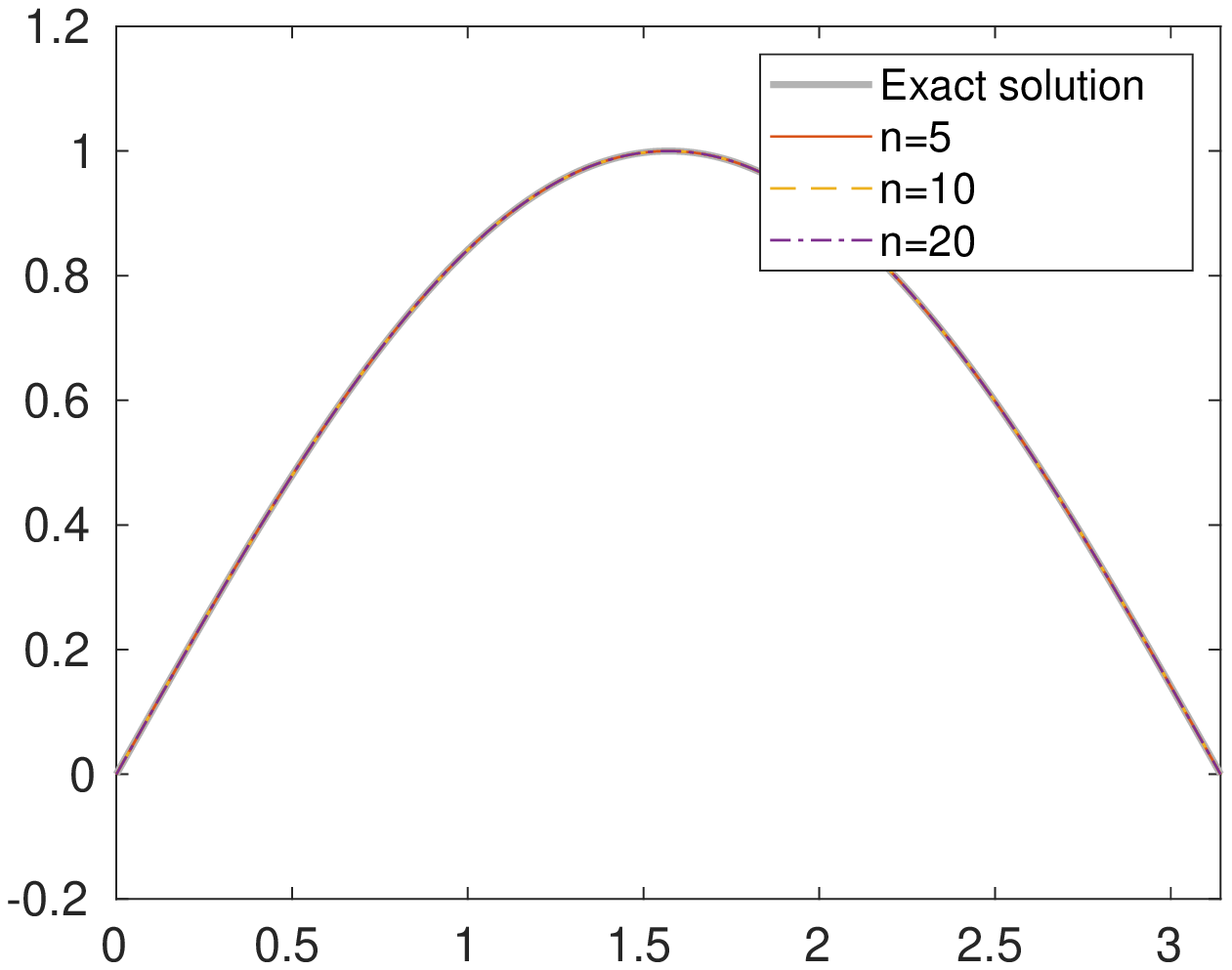}\hfill
\includegraphics[width=.46\textwidth]{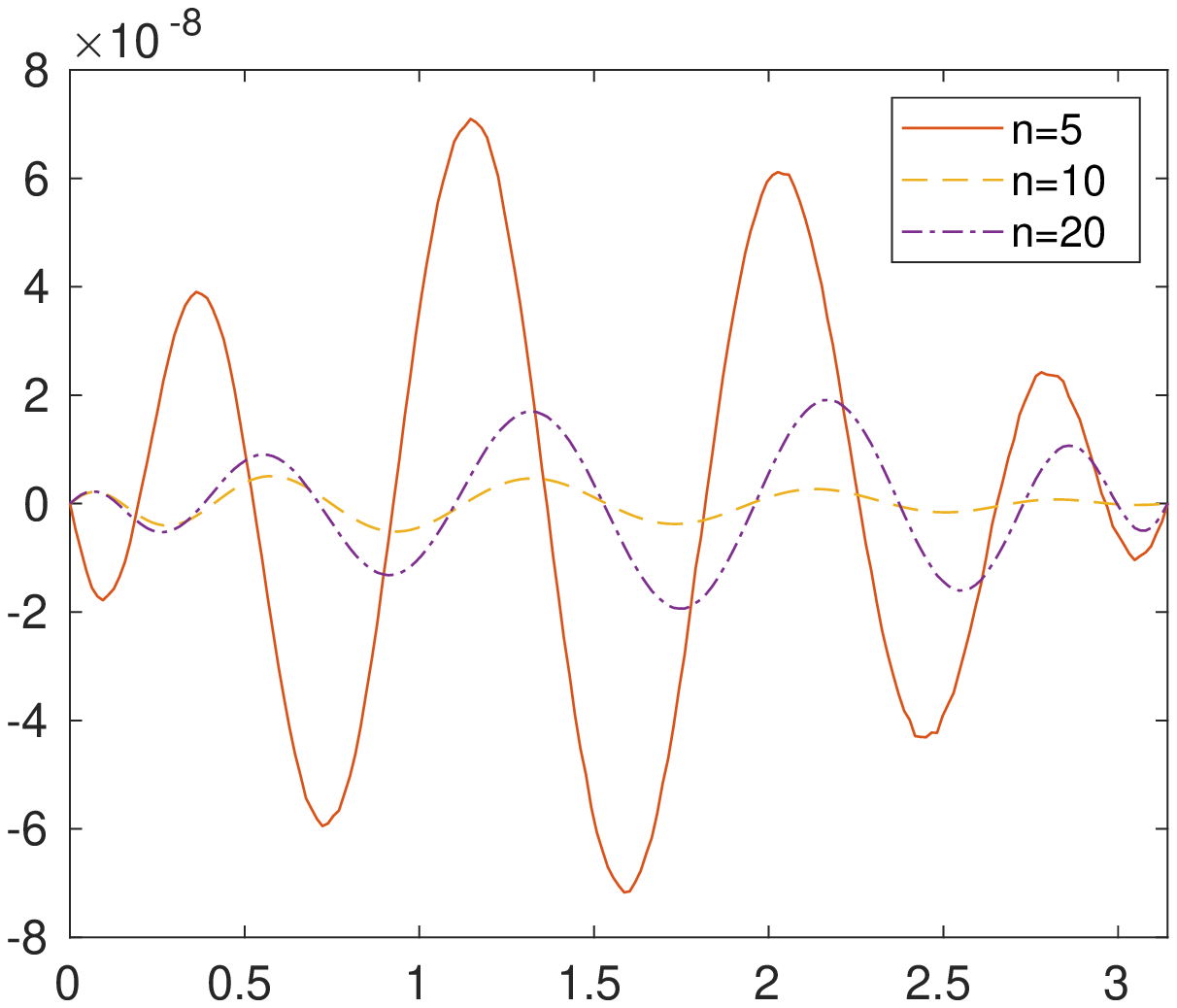}
\caption{Non-regularized reconstructions of the solution of Test problem~2
(left) and corresponding errors (right), for $n=5,10,20$ and without noise.}
\label{fig22}
\end{figure}

%

\begin{figure}[ht]
\centering
\includegraphics[width=.47\textwidth]{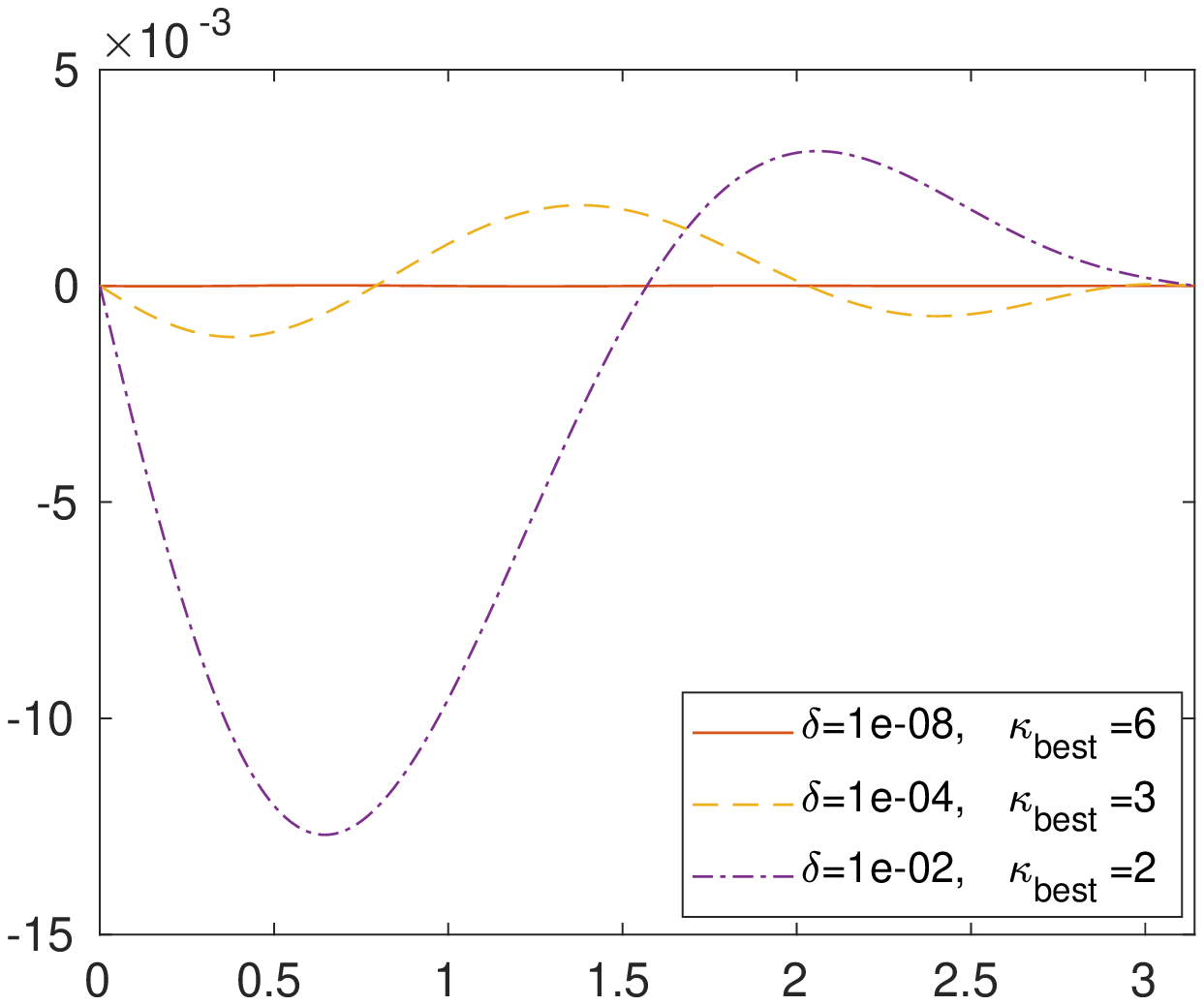}\hfill
\includegraphics[width=.46\textwidth]{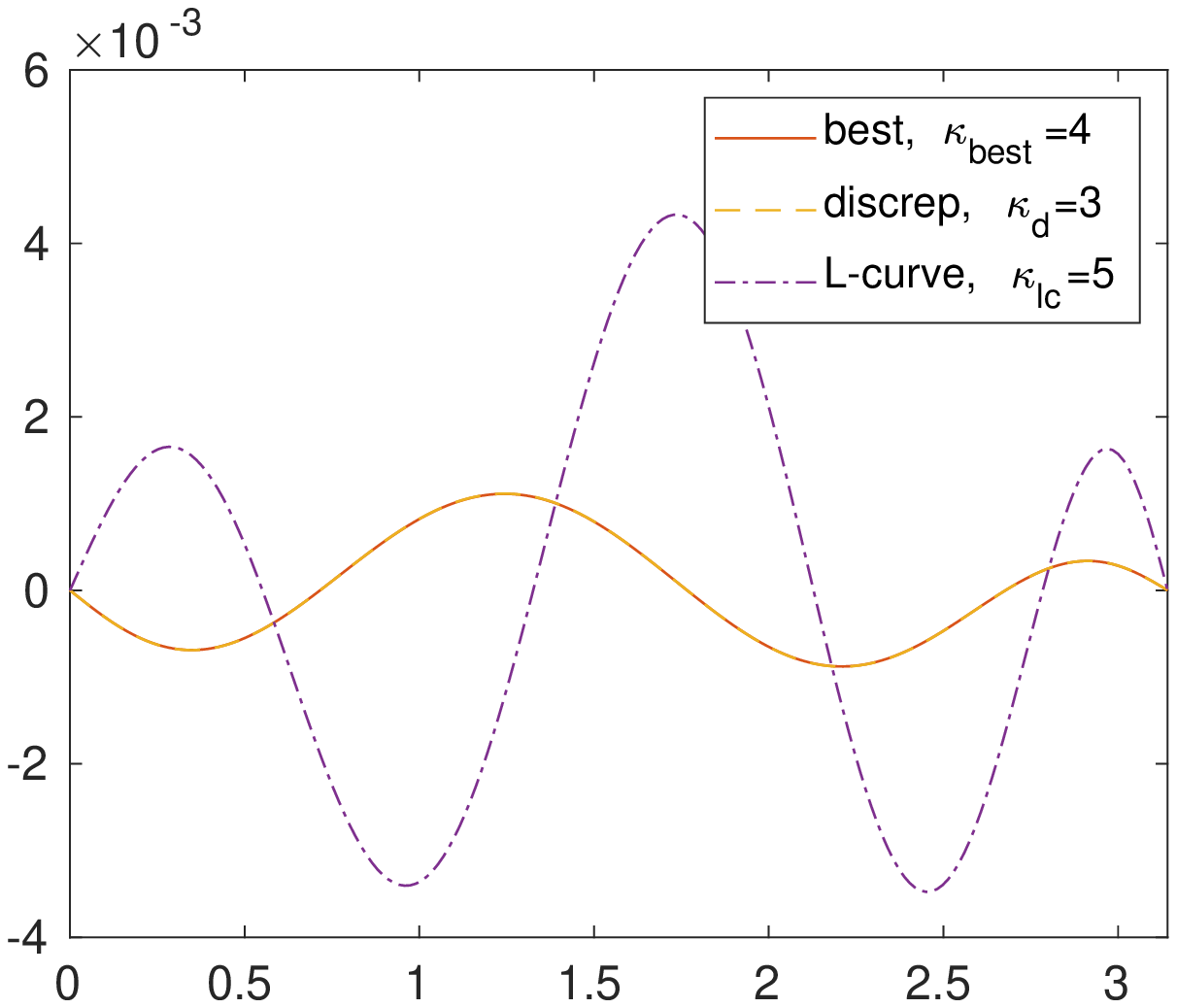}
\caption{On the left: errors corresponding to the regularized solutions
$f^{(\kappa_\text{best})}(t)$ of Test problem~2, for $n=10$ and
$\delta=10^{-8},10^{-4},10^{-2}$. On the right: errors for the solutions
$f^{(\kappa)}(t)$, for $n=20$, $\delta=10^{-4}$, and different estimation
methods for $\kappa$. The values of the regularization parameters
$\kappa_\text{best}$, $\kappa_\text{d}$, and $\kappa_\text{lc}$ are displayed
in the legend.}
\label{fig6}
\end{figure}

Figure~\ref{fig22} shows that, when the data vector is only affected by
rounding errors, the non-regularized solution is very accurate.
On the contrary, as in the previous example, the non-regularized solution is
strongly unstable when a sensible amount of noise is added to the data; we do
not display the results for the sake of brevity.

The graph on the left of Figure~\ref{fig6} depicts the behavior
of the best regularized solutions corresponding to the three noise levels
$\delta=10^{-8},10^{-4},10^{-2}$. All the reconstructions are accurate.
In the second graph, we compare the error corresponding to the optimal
regularization parameter to the ones produced by the discrepancy principle and
the L-curve. Even if the estimated values of the parameter are slightly
different, the results are satisfactory.
We verified that the outcome is not sensibly influenced by the size of the
problem.

A discretization of the Baart integral equation \cite{baart} is implemented in
the Regularization Tools library by P.~C.~Hansen \cite{Hansentools}.
The routine \texttt{baart} adopts a Galerkin discretization based on
orthonormal box functions. We implemented the same discretization method for
the second equation of \eqref{sysbaart}, in order to compare this approach with
the one we propose.

Table~\ref{tabex2} shows the results obtained by the Galerkin approach compared
to the method described in this paper; no regularization method is applied for
the solution of the corresponding linear systems and the right-hand sides are
exact up to machine precision.
The numbers reported in the table represent the infinity norm errors between
the exact and the approximate solutions computed on a discretization of the
interval $[0,\pi]$.
It is clear that there is a strong propagation of rounding errors for the first
method, while the solutions computed by the Riesz approach are very accurate
(see also Figure~\ref{fig22})
thanks to the choice of the function spaces and to the effective use of the
boundary information on the solution.

\begin{table}[ht]\centering
\caption{Infinity norm errors obtained by discretizing Test problem~2 by a
Galerkin method (see the function \texttt{baart} in \cite{Hansentools}) and the
proposed approach based on Riesz theory; the data is free from noise and no
regularization is applied.}\label{tabex2}
\begin{tabular}{ccc}
\hline
$n$ & Galerkin & Riesz \\
\hline
6 & $3.12\cdot 10^{-1\strut}$  & $7.44\cdot 10^{-9}$ \\
10 & $8.78\cdot 10^{-1}$ & $5.16\cdot 10^{-9}$ \\
20 & $4.38\cdot 10^{4\phantom{-}}$  & $1.94\cdot 10^{-8}$ \\
\hline
\end{tabular}
\end{table}

\begin{figure}[ht]
\centering
\includegraphics[width=.47\textwidth]{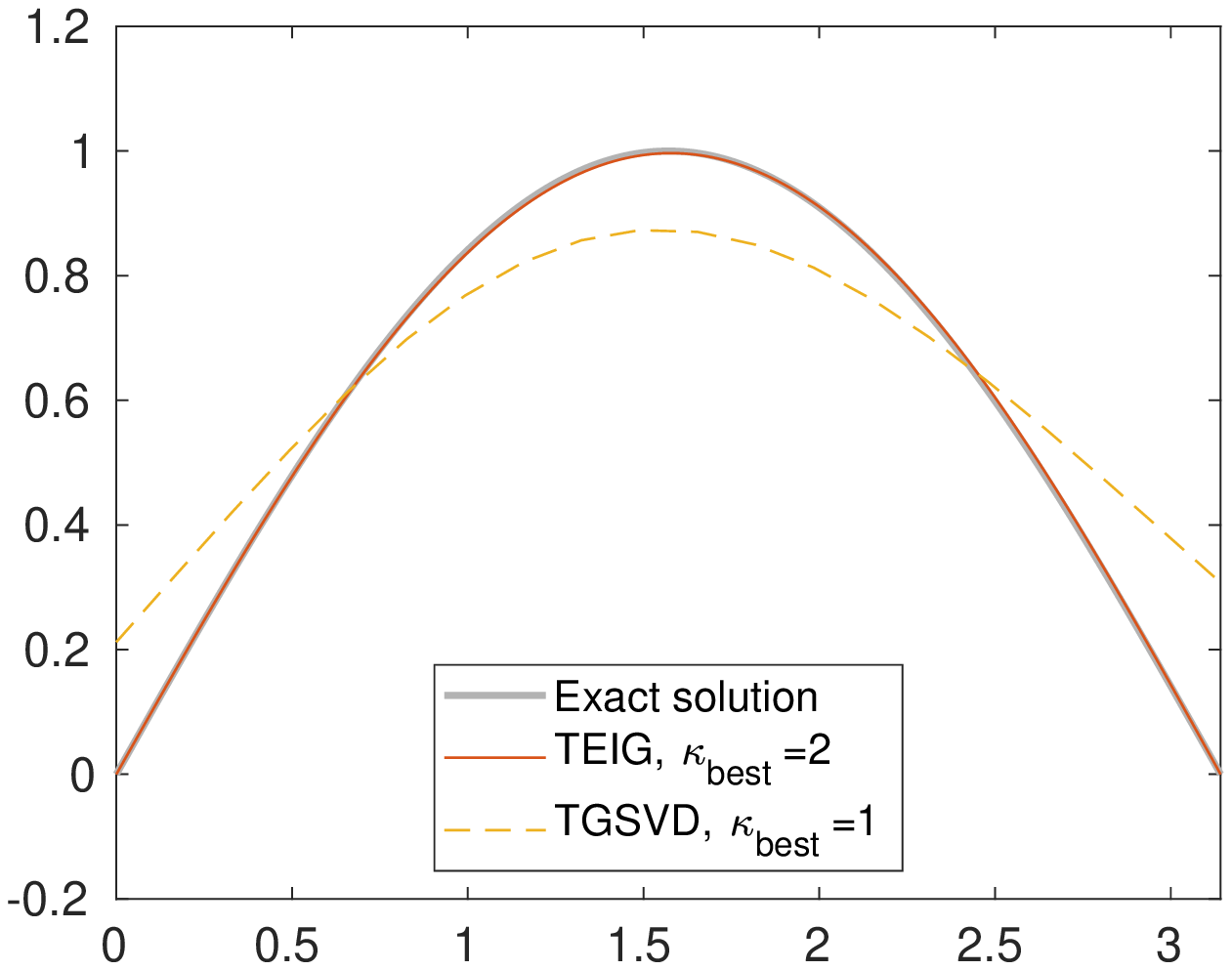}\hfill
\includegraphics[width=.48\textwidth]{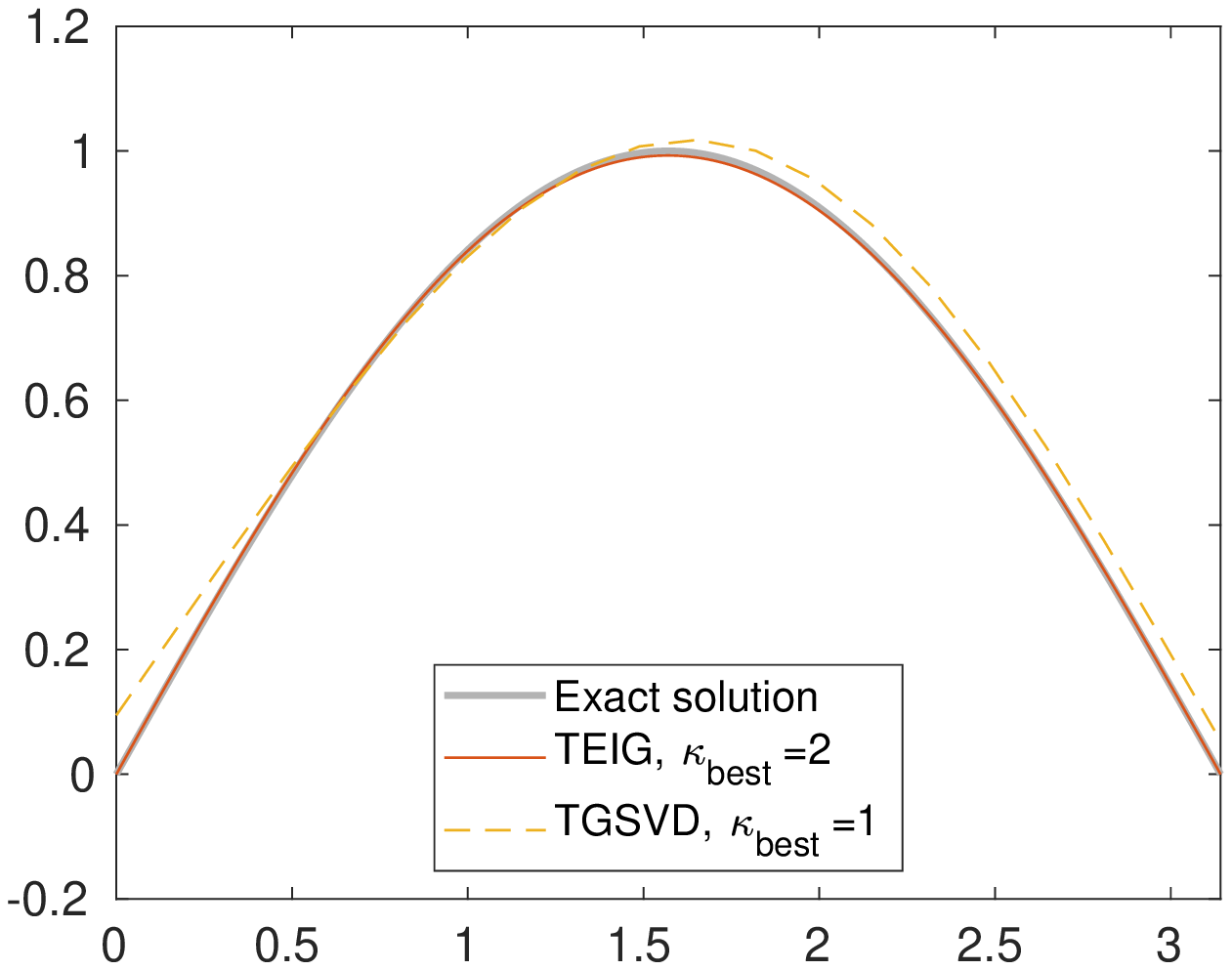}
\caption{Regularized solutions $f^{(\kappa_\text{best})}(t)$ of Test problem~2,
for $n=20$ and $\delta=10^{-2}$, obtained by applying TGSVD to the linear
system resulting from the Galerkin discretization, and by the new proposed
method, labelled as TEIG. On the left, we report the two solutions for the first
equation of system \eqref{sysbaart}, on the right the solutions of the system.
The values of the optimal regularization parameters $\kappa_\text{best}$ are
displayed in the legend.}
\label{figconf}
\end{figure}

We performed a similar comparison in the presence of noise in the data, setting
$\delta=10^{-2}$ in \eqref{noise}, and solving Test problem~2 for $n=20$.
In Figure~\ref{figconf}, the Galerkin approach is regularized by the truncated
generalized singular value decomposition (TGSVD) \cite{Hansen} with a discrete
approximation of the second derivative operator as a regularization matrix.
The obtained solution is compared to the one produced by the method described
in Section~\ref{sec:regularization}.

The graph on the left reports the results for the first equation of system
\eqref{sysbaart}, while the one on the right corresponds to the complete
system.
It is evident that the new method is more accurate than the Galerkin/TGSVD
approach. At the same time, the graphs also show that there is some advantage
in solving the system rather than a single equation.
There is a slight improvement in the error also for the Riesz approach, but
this is not visible in the graph since the order of the infinity norm error is 
$10^{-2}$.

\section{A case study}\label{sec:case study}

Let us consider the following system of integral equations of the first kind 
\begin{equation}\label{linmodel}
\begin{cases}
\displaystyle\int_0^\infty k^V(z+h) \sigma(z)\dz=g^V(h), \qquad 
h \in [0,\infty), \\
\displaystyle \int_0^{\infty\strut} k^H(z+h) \sigma(z)\dz=g^H(h), \qquad 
h \in [0,\infty),
\end{cases}
\end{equation}
proposed in \cite{McNeill} for reproducing the readings of a ground
conductivity meter, a frequency domain electromagnetic (FDEM) induction device;
see Section~\ref{sec:introduction}.
In the above equations,
\begin{equation}\label{kernels}
k^V(z) = \dfrac{4z}{(4z^2+1)^{3/2}}, \qquad
k^H(z) = 2-\dfrac{4z}{(4z^2+1)^{1/2}}
\end{equation}
are the kernel functions corresponding to the vertical and horizontal 
orientation of the coils, respectively, $\sigma(z)\geq 0$ is the unknown
function that represents the electrical conductivity of the subsoil at depth
$z$ below the ground surface, and $g^V(h)$, $g^H(h)$ are given right-hand sides
that represent the apparent conductivity sensed by the device at height $h$
over the ground for the vertical and the horizontal orientation of the coils, 
respectively.
The depth $z$ and the height $h$ are measured in meters, the electrical
conductivity in Siemens per meter.

The conditions for the existence and uniqueness of the solution of system
\eqref{linmodel} have been studied in \cite{dfrv19}, where three collocation
methods were also proposed and compared.
In \cite{dfpr21}, a preliminary version of the method presented in this paper
was applied to the first equation of the model. Here, we extend the
investigation to both equations.

Following \cite{dfrv19}, assuming the a priori information
$\sigma(z)\leq\beta$, for $z>z_0$, we split each integral into the sum 
\begin{equation*}
\int_0^\infty k_\ell(h,z) \sigma(z) \dz =
\int_0^{z_0} k_\ell(h,z) \sigma(z) \dz + 
\int_{z_0}^\infty k_\ell(h,z) \sigma(z) \dz, \qquad \ell=1,2,
\end{equation*}
where $k_1(h,z)=k^V(h+z)$ and $k_2(h,z)=k^H(h+z)$.

Given the expression \eqref{kernels} of the kernels, setting
$\sigma(z)\simeq\beta$, for $z>z_0$ and $z_0$ sufficiently large, the last
integral can be analytically computed.
Then, the system becomes
\begin{equation*}
\int_0^{z_0} k_\ell(h,z) \sigma(z) \dz = g_\ell(h) - 
\beta \int_{z_0}^\infty k_\ell(h,z) \dz, \qquad \ell=1,2,
\end{equation*}
with $g_1(h)=g^V(h)$ and $g_2(h)=g^H(h)$.
In this way, system \eqref{linmodel} is replaced by
\begin{equation}\label{linmodel1}
\begin{cases}
\displaystyle (K_1\sigma)(h) := \displaystyle \int_0^{z_0} k_1(h,z)  
\sigma(z)\dz = g_1(h)-\dfrac{\beta}{\theta(z_0,h)\strut}, \\
\displaystyle (K_2\sigma)(h) := \displaystyle \int_0^{z_0} k_2(h,z)  
\sigma(z)\dz = g_2(h)-\beta\left(\theta(z_0,h) - 2(h+z_0)\right),
\end{cases}
\end{equation}
where (see~\cite{dfrv19})
\begin{equation}\label{theta}
\theta(z,h)=\sqrt{4(z+h)^2+1}.
\end{equation}
We remark that $(-\theta(z,h))^{-1}$ is a primitive function of $k_1(h,z)$,
and $2z-\theta(z,h)$ is that of $k_2(h,z)$.

To determine a solution by applying the theory developed in Sections 
\ref{sec:preliminaries} and \ref{sec:method}, it is necessary to introduce the 
linear function \eqref{gamma}
\begin{equation*}
\gamma(z)=\left(1-\frac{z}{z_0}\right)\alpha+\frac{z}{z_0}\beta,
\end{equation*}
and assume that the values of the electrical conductivity at the endpoints of
the integration interval are known, e.g., $\sigma(0)=\alpha$ and
$\sigma(z_0)=\beta$.
The boundary values can usually be approximated in applications; 
see~\cite{dfrv19}.

By collocating equations \eqref{linmodel1} at the points $h_i$, 
assuming $n_1=n_2=n$ and $h_{1,i}=h_{2,i}=h_i$, for $i=1,\ldots,n$, we obtain
$$
\begin{cases}
\displaystyle\int_0^{z_0} k_1(h_i,z) \phi(z)\dz=\psi_1(h_i), \quad 
i=1,\dots,n, \\
\displaystyle \int_0^{z_0\strut} k_2(h_i,z) \phi(z)\dz=\psi_2(h_i), \quad 
i=1,\dots,n,
\end{cases}
$$
where 
$$
\phi(z)=\sigma(z)-\gamma(z)
$$
is the new unknown function, and
\begin{align*}
\psi_1(h_i)&=g_1(h_i)-\frac{\beta}{\theta(z_0,h_i)\strut} - \int_0^{z_0} 
k_1(h_i,z) \gamma(z) \dz \\
&= g_1(h_i) - \frac{\alpha}{\theta(0,h_i)} - 
\frac{\alpha-\beta}{2z_0}\bigl[\arcsinh(2h_i) -
\arcsinh(2(z_0+h_i)) \bigr], \\
\psi_2(h_i)&=g_2(h_i) - \beta\left(\theta(z_0,h_i) - 2(h_i+z_0)\right) - 
\int_0^{z_0} k_2(h_i,z) \gamma(z) \dz \\
&= g_2(h_i) -\left[\frac{(\alpha-\beta)h_i}{2z_0}+\alpha\right]\theta(0,h_i)
+ \frac{\alpha-\beta}{2} \left[\frac{h_i}{z_0}+1\right] \theta(z_0,h_i) \\
&\phantom{=\ } + 2\beta h_i - z_0(\alpha-\beta) - 
\frac{a-b}{4z_0}\left[\arcsinh(2h_i) - \arcsinh(2(z_0+h_i)) \right]
\end{align*}
are the new right-hand sides.

The second derivative of the Riesz representers can be computed analytically.
Indeed, from \eqref{etasecondo}, it follows that
\begin{equation*}
\begin{aligned}
\eta_{1,i}''(x)
& = \displaystyle \int_0^{z_0} G''_z(x) k_1(h_i,z) \dz \\
& = \frac{1}{2} \left[ \left(1-\frac{x}{z_0}\right)\arcsinh(2h_i) 
-\arcsinh(2(x+h_i)) \right. \\
& \phantom{XXX} \left. + \frac{x}{z_0}\arcsinh(2(z_0+h_i))\right]
\end{aligned}
\end{equation*}
and
\begin{equation*}
\begin{aligned}
\eta_{2,i}''(x)& = \int_0^{z_0} G''_z(x) k_2(h_i,z) \dz \\
& = \frac{1}{2}\left[ 2x(x-z_0)+ 
x\left(1+\frac{h_i}{z_0}\right)\theta(z_0,h_i) - 
(x+h_i)\theta(x,h_i) \right. \\
& \phantom{XXX} \left. + h_i\left(1-\frac{x}{z_0}\right)\theta(0,h_i) 
+\eta_{1,i}''(x)\right],
\end{aligned}
\end{equation*}
where $\theta(x,h)$ is the function defined in \eqref{theta}. 
From the above second derivatives, we can compute the Riesz functions 
\begin{align*}
\eta_{1,i}(y)& = \int_0^{z_0} G''_y(x) \eta_{1,i}''(x) \dx \\
& = \frac{3}{16}\biggl[ (y+h_i)\theta(y,h_i)
-y\left(1+\frac{h_i}{z_0}\right)\theta(z_0,h_i) + h_i\left(
\frac{y}{z_0}-1\right)\theta(0,h_i) \biggr] \\
&\phantom{=} +\frac{1}{2}\left\{ \biggl[\frac{1}{2}\left(\frac{y}{z_0}-1
\right)\left(\frac{1}{8}-h_i^2-\frac{y^2}{3} \right) \right. 
+\frac{y}{3}(y-z_0)\biggr]\arcsinh(2h_i) \\
&\phantom{=} + \biggl[ -\frac{y}{2z_0}\left( \frac{1}{8}-h_i^2-\frac{y^2}{3}
\right) +y\left(h_i+\frac{z_0}{3}\right)
\biggr]\arcsinh(2(z_0+h_i)) \\
&\phantom{=} \left. + \frac{1}{2}\left[\frac{1}{8}-(y+h_i)^2
\right]\arcsinh(2(y+h_i))
\right\}
\end{align*}
and
\begin{align*}
\eta_{2,i}(y)& = \int_0^{z_0}  G''_y(x) \eta_{2,i}''(x) \dx \\
& = \frac{1}{192z_0}\Bigl\{
z_0\left[h_i\left(13-8(3h_iy+h_i^2+3y^2)\right)
+y(13-8y^2)\right]\theta(y,h_i) \Bigl. \\ 
&\phantom{=} \left. +y\left[h_i\left(8(3h_iz_0+h_i^2+2y^2+z_0^2) 
-13\right) 
+z_0\left(8(2y^2-z_0^2)-13 \right) \right]\theta(z_0,h_i) \right. \\
&\phantom{=} \left. +h_i\left[ z_0\left(8(h_i^2+6y^2 -4yz_0)-13 \right)
+y\left( 13-8(h_i^2+2y^2)\right)\right]\theta(0,h_i) \right. \\
&\phantom{=} \Bigr. +16yz_0\left(y^3-2y^2z_0+z_0^3\right)
\Bigr\}\\
&\phantom{=} +\frac{1}{128z_0}\left\{
\left(y-z_0\right) \left[1-16\left(h_i^2+\frac{y^2}{3}-\frac{2yz_0}{3}\right) 
\right] \arcsinh(2h_i) \right. \\
&\phantom{=} \left. + z_0 \left[1-16\left(y+h_i\right)^2\right] 
\arcsinh(2(y+h_i) \right. \\
&\phantom{=} \left. - 
y\left[1-16\left(h_i^2+\frac{y^2}{3}+2h_iz_0+\frac{2z_0^2}{3}\right) 
\right] \arcsinh(2(z_0+h_i))
\right\}.
\end{align*}

\begin{figure}[ht]
\centering
\includegraphics[width=.46\textwidth]{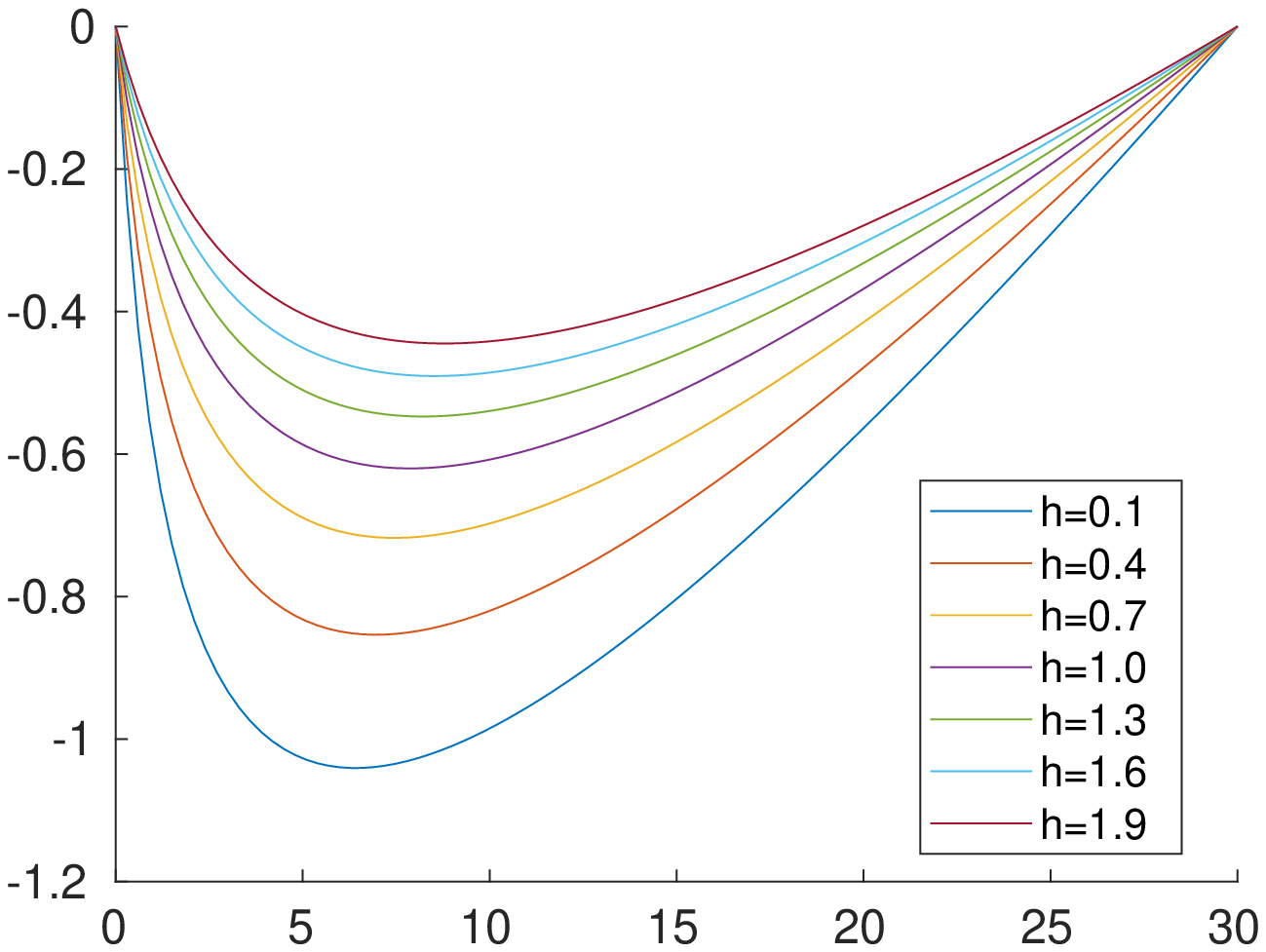}\hfill
\includegraphics[width=.46\textwidth]{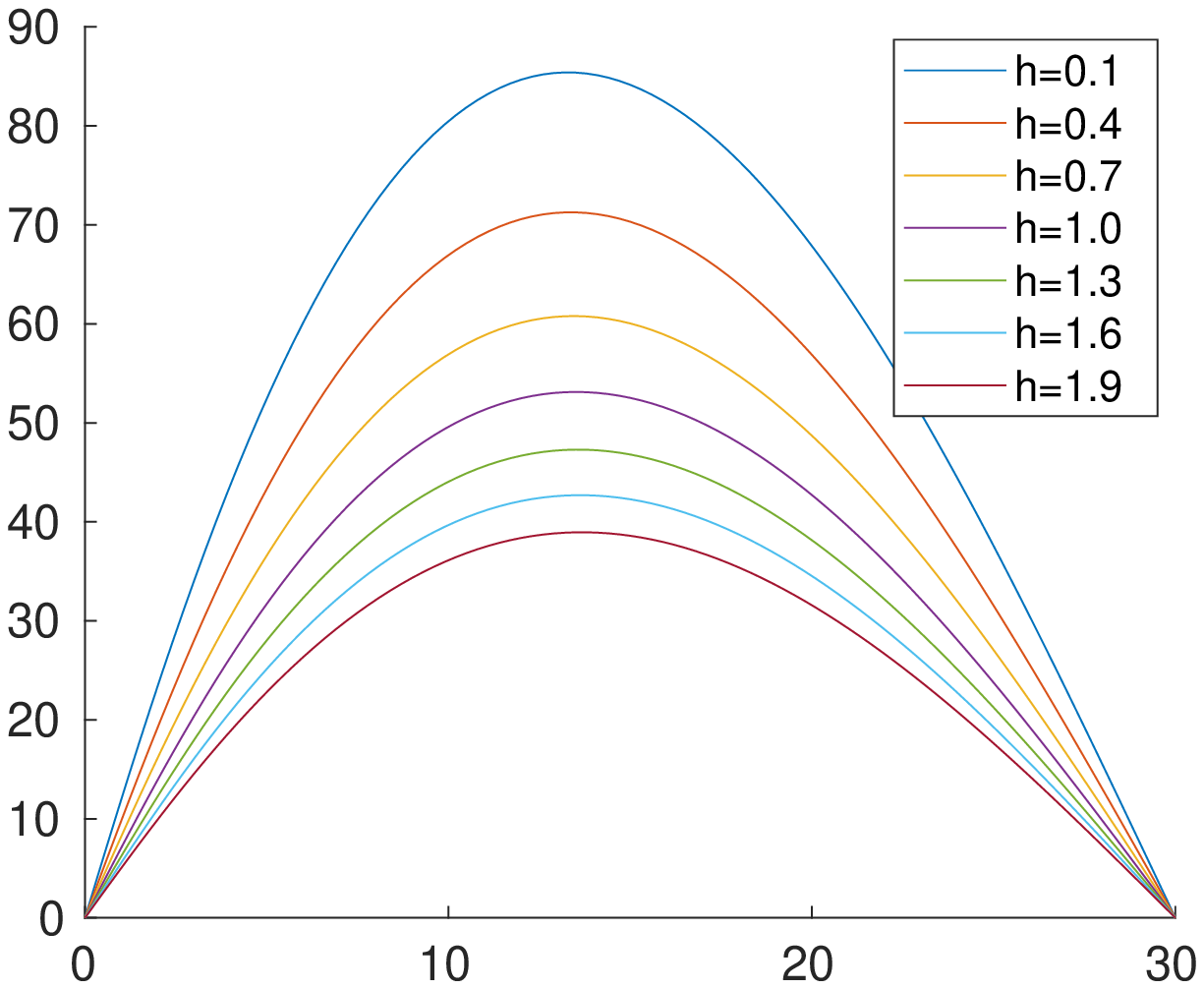}\\ \vspace{.5cm}
\includegraphics[width=.46\textwidth]{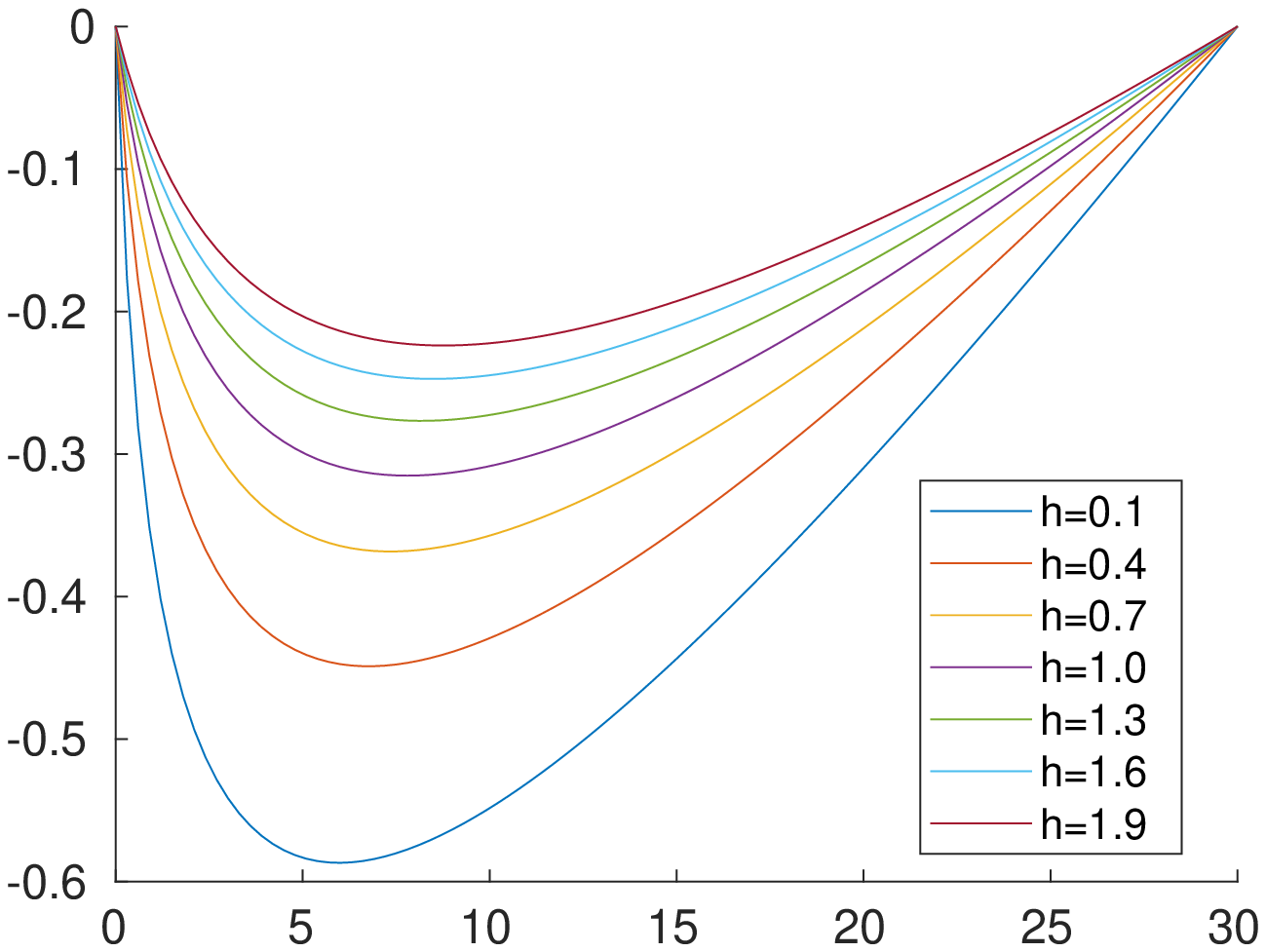}\hfill
\includegraphics[width=.46\textwidth]{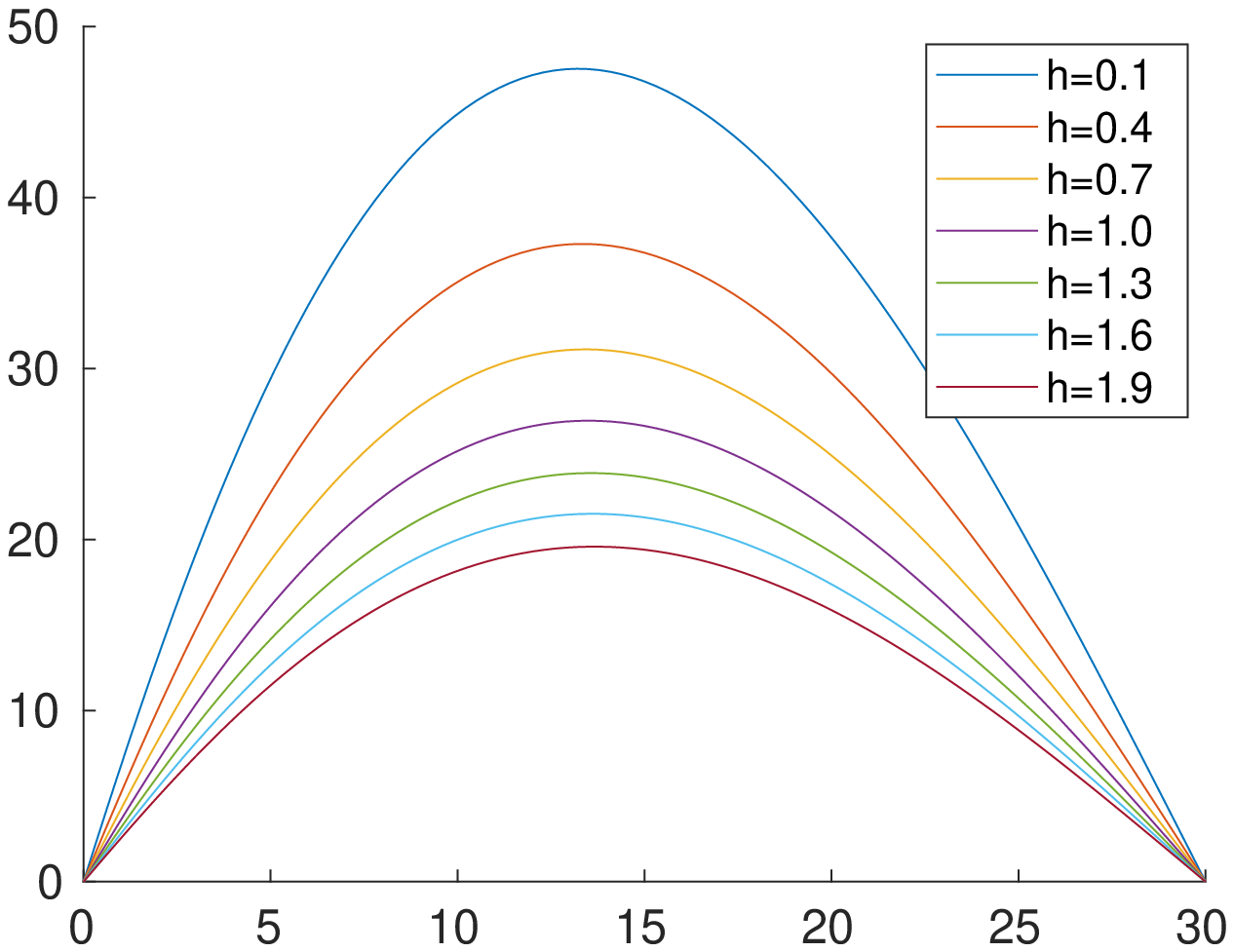} 
\caption{The functions $\eta_{1,i}''$ (top-left), $\eta_{1,i}$ (top-right),
$\eta_{2,i}''$ (bottom-left), and $\eta_{2,i}$ (bottom-right),
with $h_i = 0.1 + (i-1)\frac{3}{10}$ and $i = 1,\ldots,7$.}
\label{fig10}
\end{figure}

Figure \ref{fig10} shows the behavior of $\eta_{1,i}''$ and $\eta_{1,i}$ for 
different values of $h_i$, in the case $z_0 = 30$.
We also report the graphs of the Riesz 
representers for the horizontal orientation $\eta_{2,i}''$ and $\eta_{2,i}$.
Figure~\ref{fig10cap} displays the orthonormal functions $\widehat{\eta}_{1,i}$ 
and $\widehat{\eta}_{2,i}$ defined in \eqref{etaort2}, together with their
second derivatives $\widehat{\eta}_{1,i}''$ and $\widehat{\eta}_{2,i}''$. In 
the summation \eqref{etaort2}, the upper bound is set to $N=12$ for preserving 
the positivity of the eigenvalues.

\begin{figure}[ht]
\centering
\includegraphics[width=.46\textwidth]{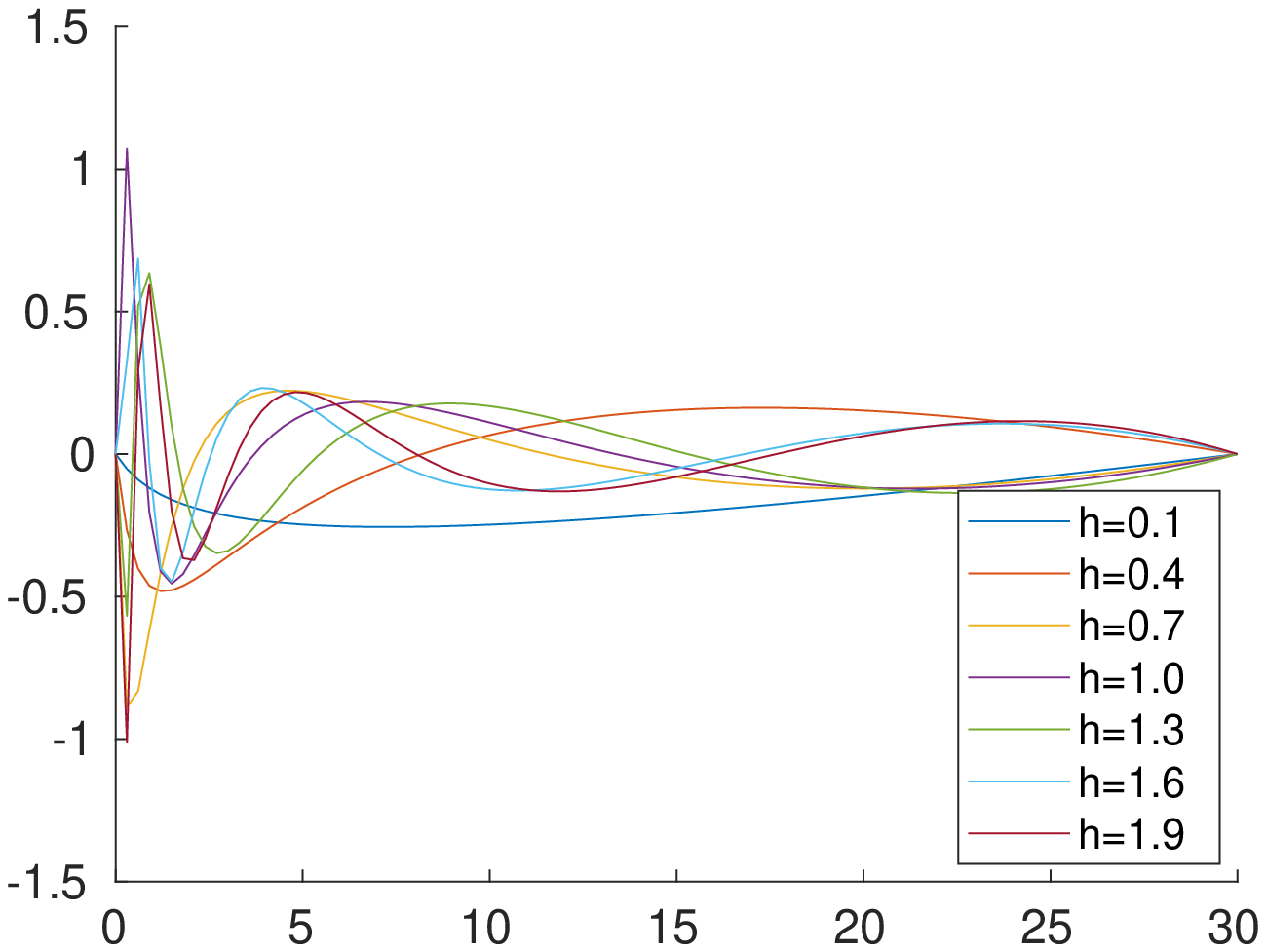}\hfill
\includegraphics[width=.46\textwidth]{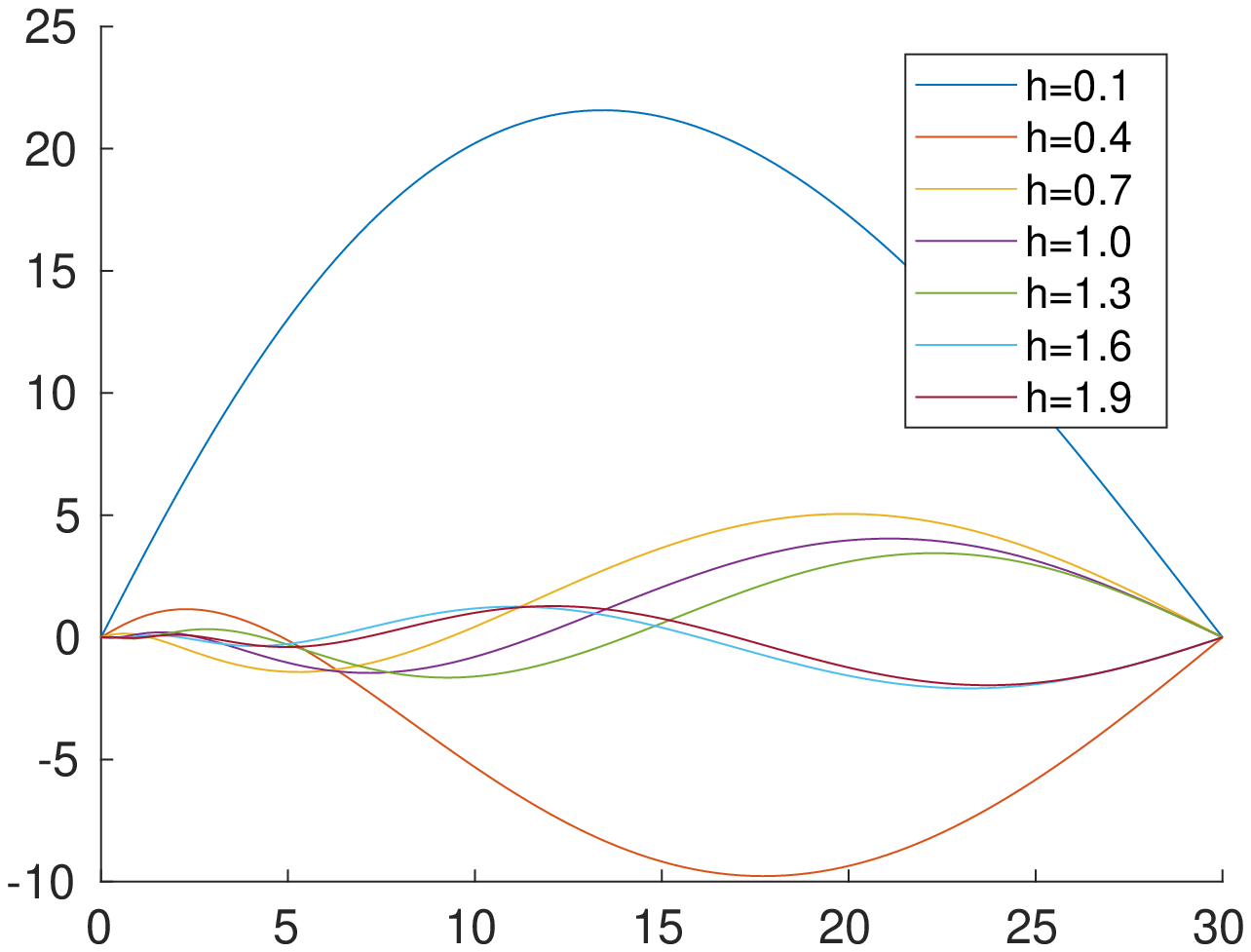}\\ \vspace{.5cm}
\includegraphics[width=.46\textwidth]{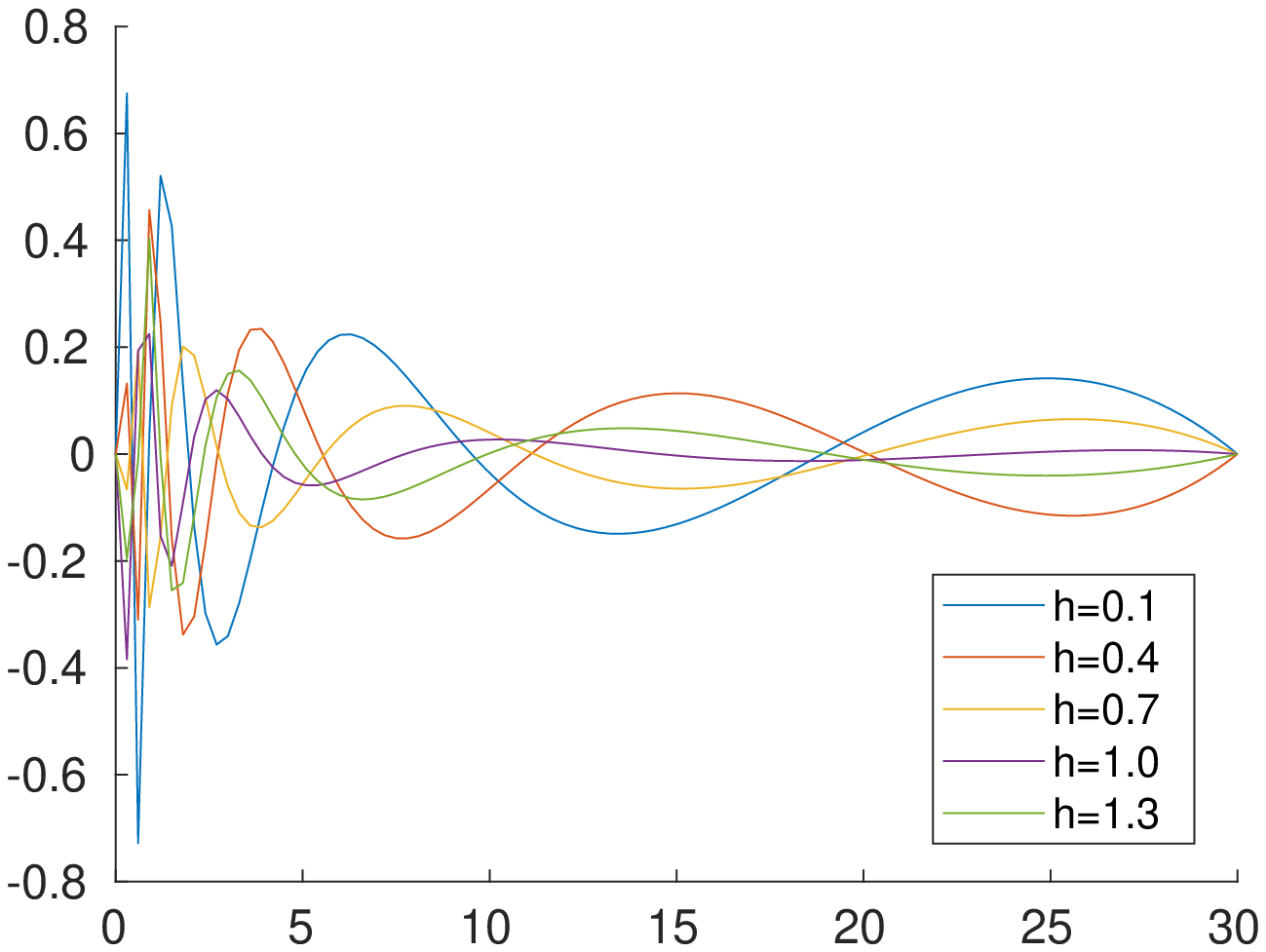}\hfill
\includegraphics[width=.46\textwidth]{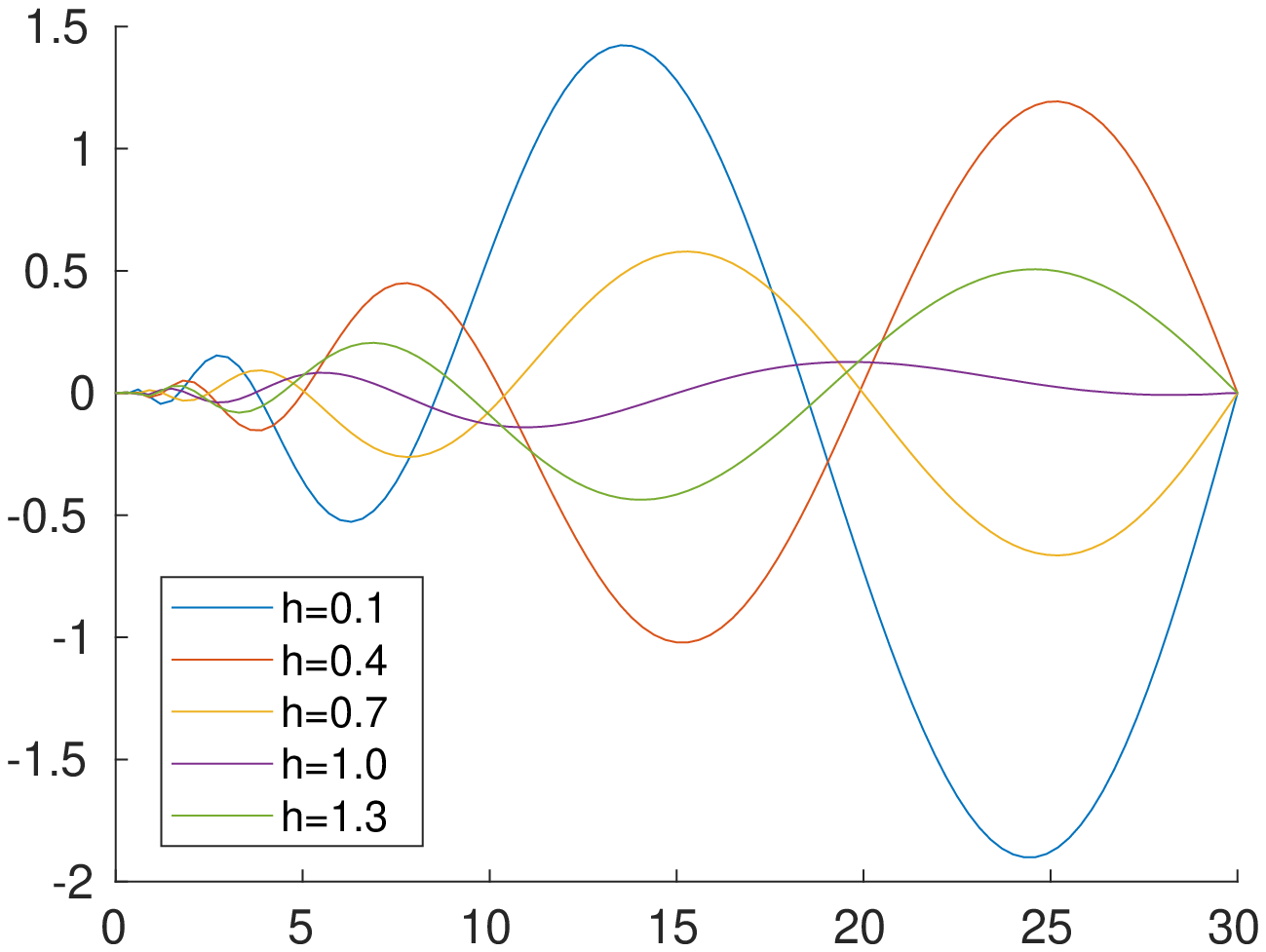} 
\caption{The orthonormal functions $\widehat{\eta}_{1,i}''$ (top-left),
$\widehat{\eta}_{1,i}$ (top-right), $\widehat{\eta}_{2,i}''$ (bottom-left), and
$\widehat{\eta}_{2,i}$ (bottom-right), with $h_i = 0.1 + (i-1)\frac{3}{10}$ and
$i=1,\ldots, 7$.}
\label{fig10cap}
\end{figure}

In order to ascertain the accuracy of our method, when applied to the case
study presented in this section, we consider three different profiles for the
electrical conductivity $\sigma(z)$.
Then, for each test function, we compute the data vector
$\psib_{\text{exact}}$, setting $z_0=4$ and $h_i=0.1 + 0.9\,(i-1)/(n-1)$,
$i=1,\dots,n$, for a chosen dimension $n$.

The computation of the exact data vector is performed by the \texttt{quadgk}
function of Matlab, which implements an adaptive Gauss-Kronrod quadrature
formula.

In applications, the available data is typically contaminated by errors. The 
perturbed data vector $\psib$ is determined by adding to $\psib_{\text{exact}}$ 
a noise-vector $\bm{e}$, obtained by substituting in \eqref{noise}
$\psib_{\text{exact}}$ to $\bm{g}_{\text{exact}}$ and setting $N_m=2n$.
The noise level is determined by the parameter $\delta$.

\paragraph{Test function 1.}
In the first example, we assume a smooth profile for the exact solution of
\eqref{linmodel}
\begin{equation*}
\sigma_1(z)=\ee^{-(z-1)^2}+1.
\end{equation*}
We set $\alpha=\sigma_1(0)=\ee^{-1}+1$ and $\beta =\sigma_1(z_0)=\ee^{-9}+1$.

We remark that this test function is extremely smooth, so the function 
$\phi_1(z) = \sigma_1(z) - \gamma_1(z)$ can be assumed to approximately belong 
to $\cN(\bm{K})^\perp=\Span\{\eta_{1},\dots,\eta_{N_m}\}$,
the space which contains the minimal-norm solution.

\begin{figure}[ht]
\centering
\includegraphics[width=.46\textwidth]{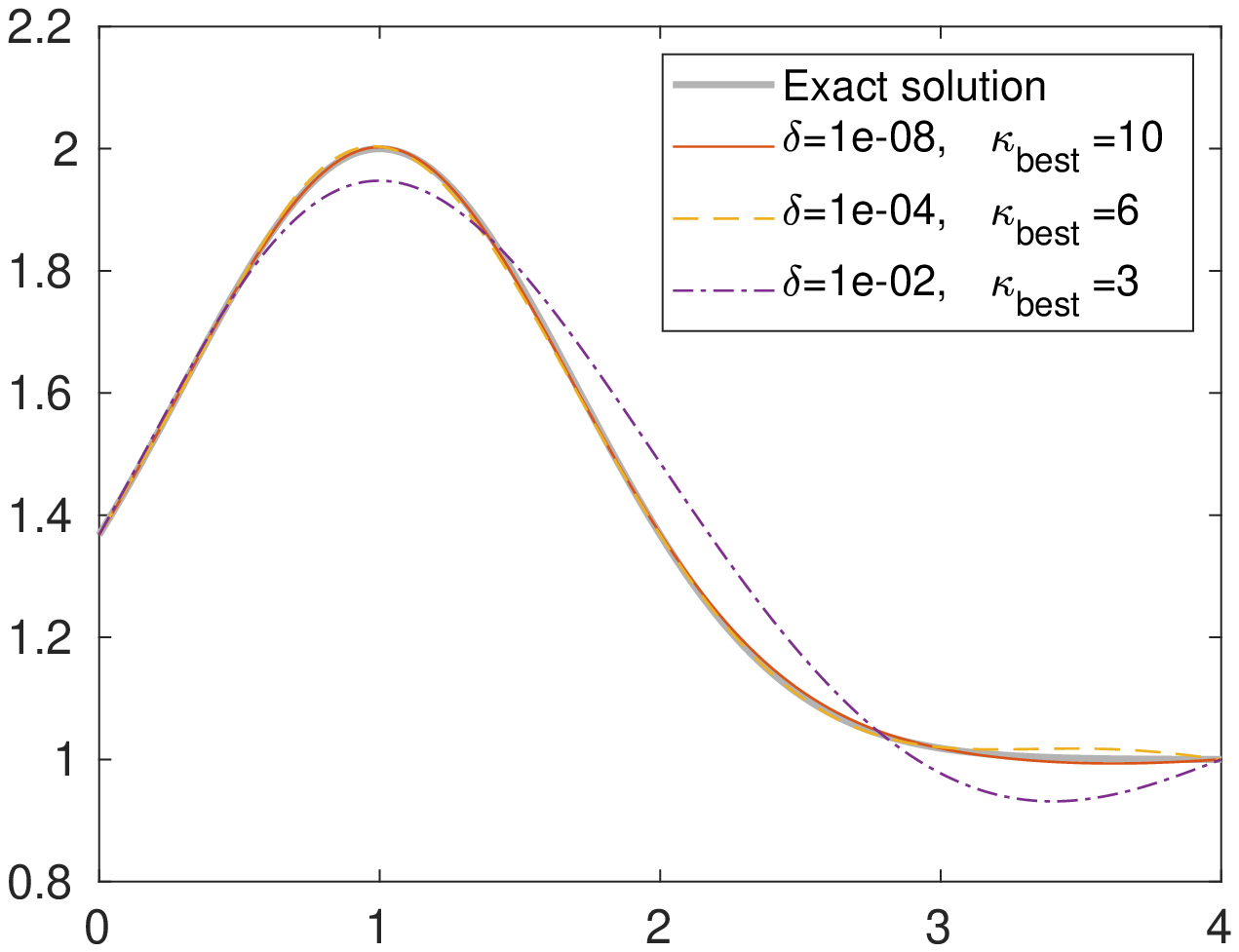}\hfill
\includegraphics[width=.46\textwidth]{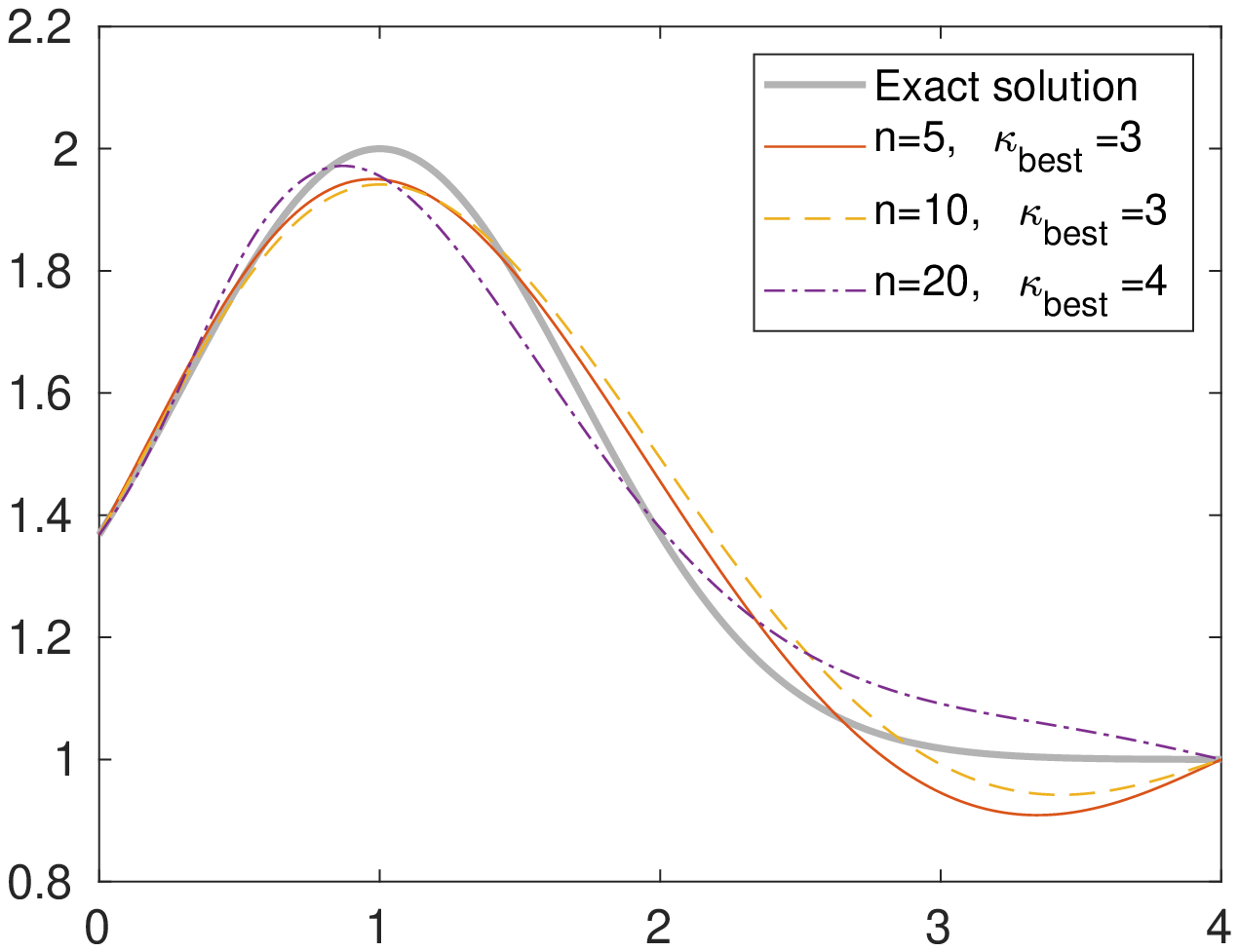}
\caption{On the left: regularized solution
$\sigma_1^{(\kappa_{\text{best}})}(z)$ for noise levels $\delta =
10^{-8}, 10^{-4}, 10^{-2}$, and $n=10$. On the right: regularized solution
$\sigma_1^{(\kappa_{\text{best}})}(z)$ for $n=5, 10, 20$, and noise
level $\delta = 10^{-2}$; the optimal regularization parameter
$\kappa_\text{best}$ is displayed in the legend.}
\label{fig7}
\end{figure}

Figure \ref{fig7} displays the results obtained by applying the method 
described in this paper to the electromagnetic integral model
\eqref{linmodel} with the optimal regularization parameter.
On the left-hand side, we report the approximation of the solution for
different noise levels $\delta = 10^{-8}, 10^{-4}, 10^{-2}$, and $n=10$;
on the right-hand side, the results for $n=5, 10, 20$ and $\delta=10^{-2}$
are depicted.
All the reconstructions are accurate and identify with sufficient accuracy 
the
maximum value of the conductivity and its depth localization.
The graph on the left shows that, even for an increasing noise level, the
method is still able to produce reliable results.
On the other hand, from the graph on the right we deduce that both the 
reconstructions and the optimal value of the regularization parameter are not
very sensitive on the size of the data vector.

In order to test the method in realistic conditions, in Figure~\ref{fig12} we
compare the optimal solution to the approximate solutions corresponding to the
parameters $\kappa_\text{d}$ and $\kappa_\text{lc}$, estimated by the
discrepancy principle with $\tau=1.3$ and by the L-curve criterion,
respectively. In this case, we have fixed $n=10$ and a noise level
$\delta=10^{-4}$. Both estimation techniques appear to be effective.

\begin{figure}[ht]
\centering
\includegraphics[scale=0.5]{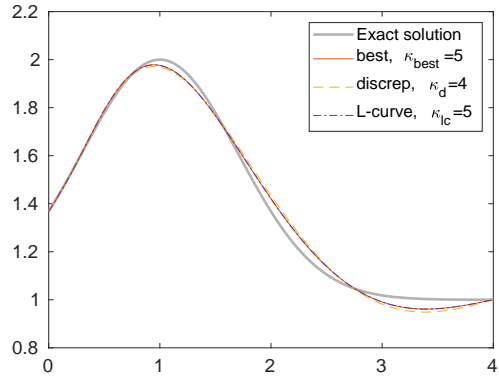}
\caption{Regularized solution $\sigma_1^{(\kappa)}(z)$ with $n=10$ and
$\delta=10^{-4}$; the optimal regularization parameter $\kappa_{\text{best}}$
is compared to those determined by the discrepancy principle $\kappa_\text{d}$
and by the L-curve $\kappa_\text{lc}$.}
\label{fig12}
\end{figure}

\begin{figure}[ht]
\centering
\includegraphics[width=.46\textwidth]{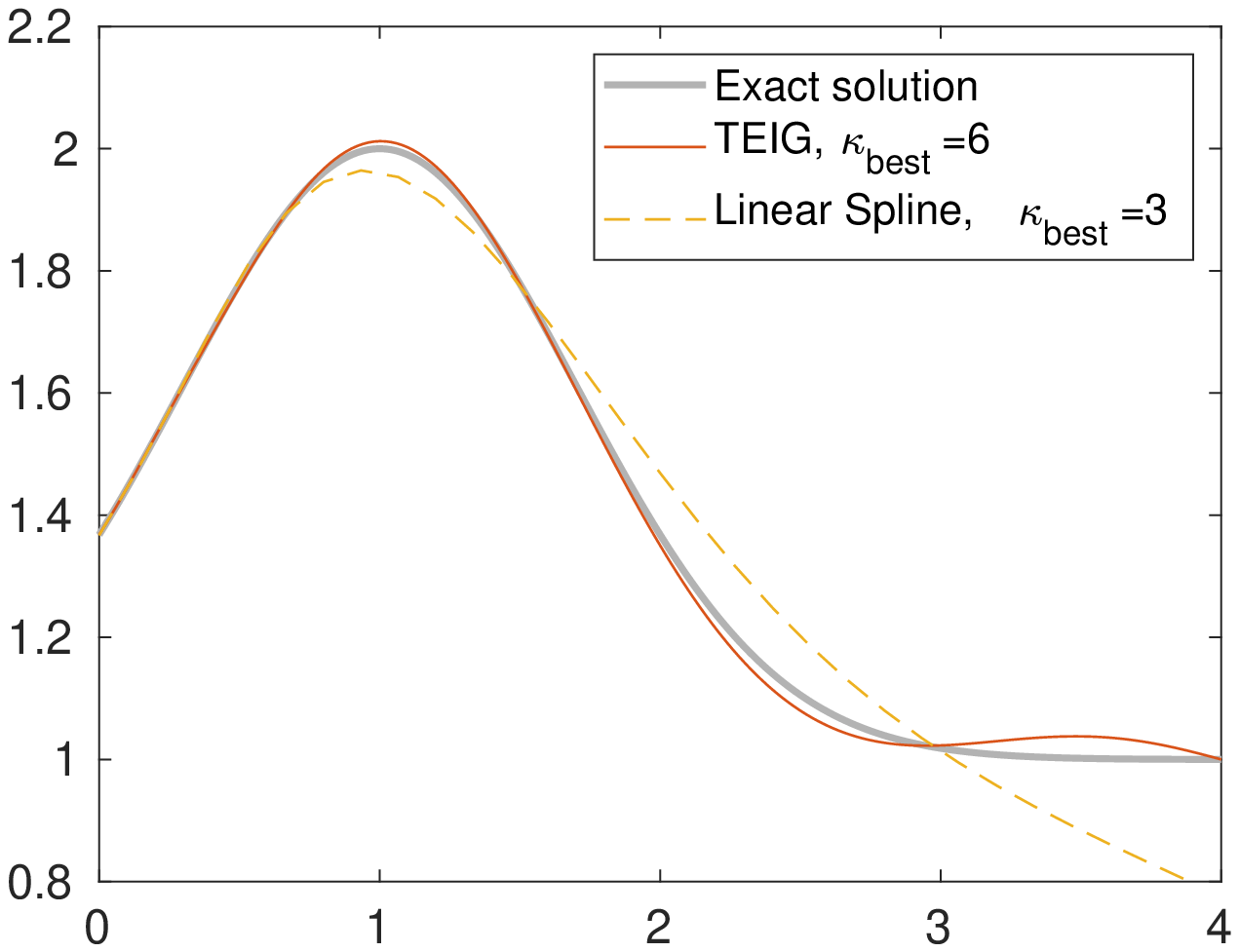}\hfill
\includegraphics[width=.46\textwidth]{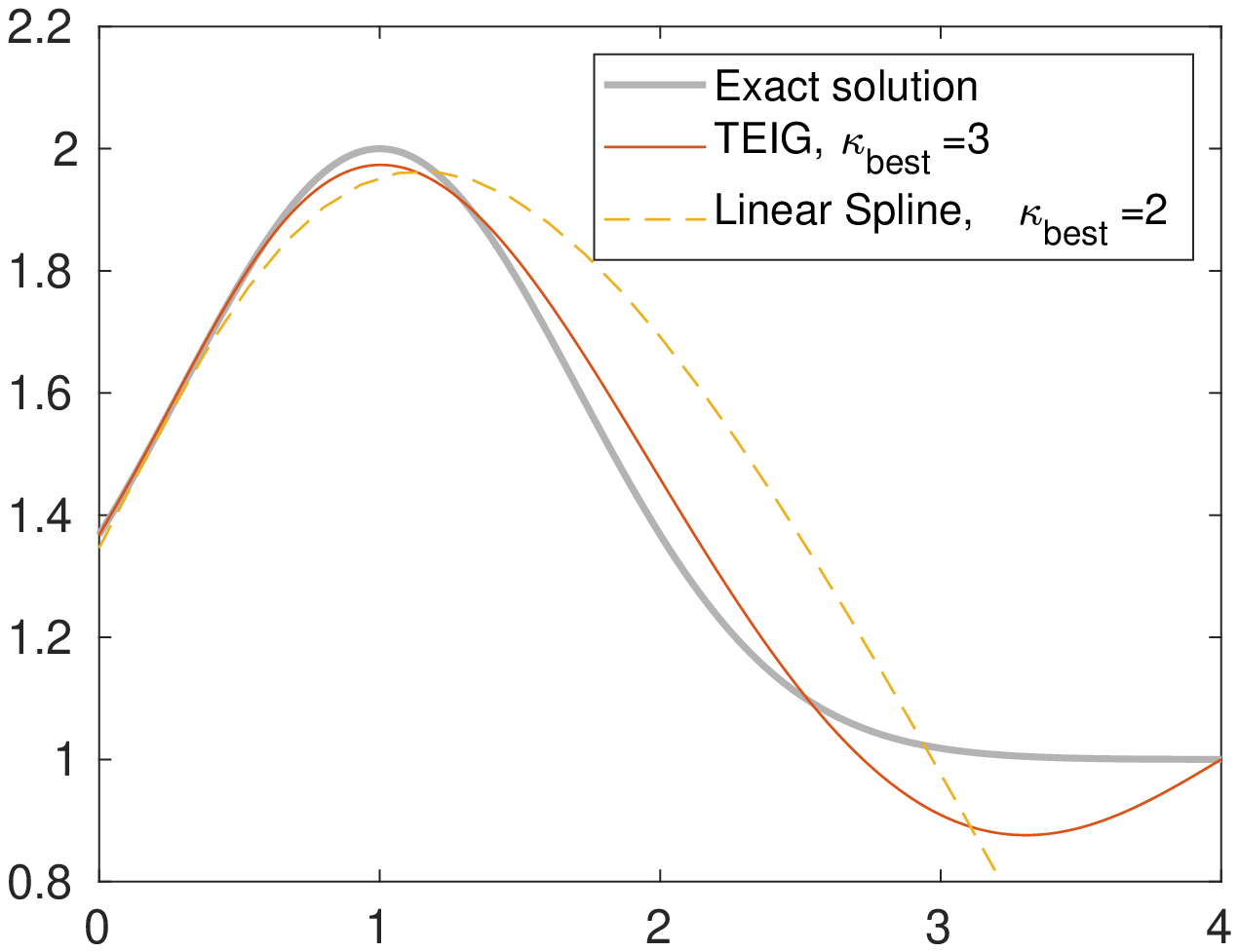}
\caption{Regularized solutions $\sigma_1^{(\kappa_{\text{best}})}(z)$ for
$n=20$ by using an approach presented in \cite{dfrv19} and the new proposed
method. On the left noise level $\delta=10^{-4}$, on the right 
$\delta=10^{-2}$. The values of the regularization parameters 
$\kappa_\text{best}$ are displayed in the legend.}
\label{figconflin}
\end{figure}

As already remarked, the linear model \eqref{linmodel} has been analyzed in
\cite{dfrv19}, where some collocation methods were discussed.
In Figure~\ref{figconflin} we compare the most effective technique presented in
\cite{dfrv19}, based on a linear spline approximation coupled to a TGSVD
regularization of the resulting linear system, to our new approach.
The two graphs report the solutions obtained with two different noise levels, 
$\delta=10^{-4}$ and $\delta=10^{-2}$, when $n=20$ and choosing the ``best''
regularization parameter.
In both cases,  the new method produces more accurate solutions than
the linear spline approach.
In particular, in the graph on the left, the spline solution is not able to
recognize the flattening of the conductivity below 3m depth. In the one on
the right, it does not even identify the correct depth of the maximum.

\paragraph{Test function 2.}
In the second experiment, we select the following model function
\begin{equation*}
\sigma_2(z)= \begin{cases}
0.8z+0.2, & z \in [0,1], \\
0.8\ee^{-(z-1)}+0.2, & z \in (1,\infty),
\end{cases}
\end{equation*}
and set $\alpha=0.2$ and $\beta=0.2+0.8\ee^{-3}$.

The graph in the left pane of Figure~\ref{fig8} reports the optimal regularized
solutions corresponding to the noise levels $\delta = 10^{-8}$, $10^{-4}$, 
$10^{-2}$, and $n = 10$.
The optimal parameter is displayed in the legend. The reconstruction is not 
accurate as in the previous test, because the solution is non-differentiable
and, consequently, it does not belong to $\cN(\bm{K})^\perp$.
Anyway, the algorithm correctly identifies the position of the maximum of the
electrical conductivity at 1m depth.

\begin{figure}[ht]
\centering
\includegraphics[width=.46\textwidth]{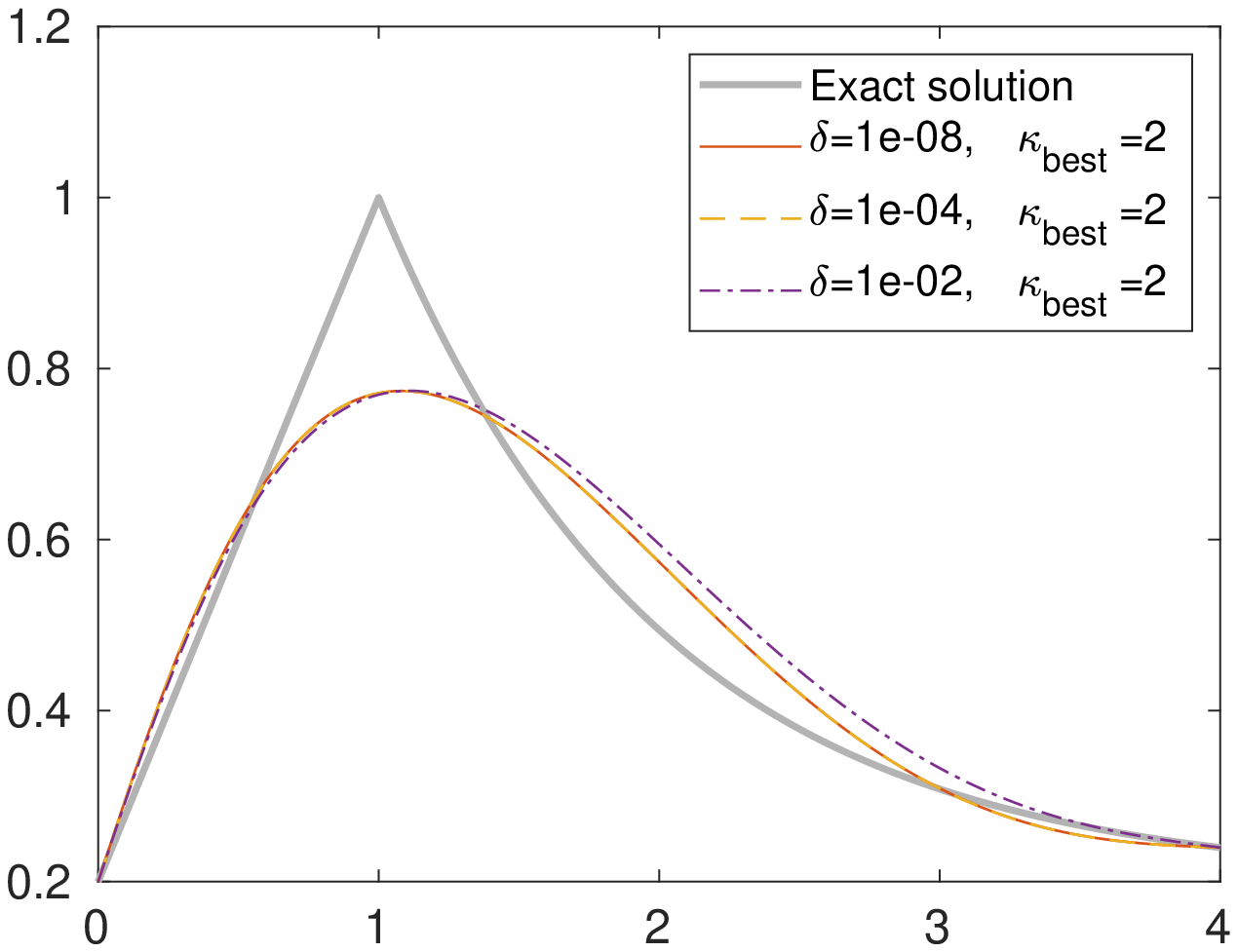}\hfill
\includegraphics[width=.47\textwidth]{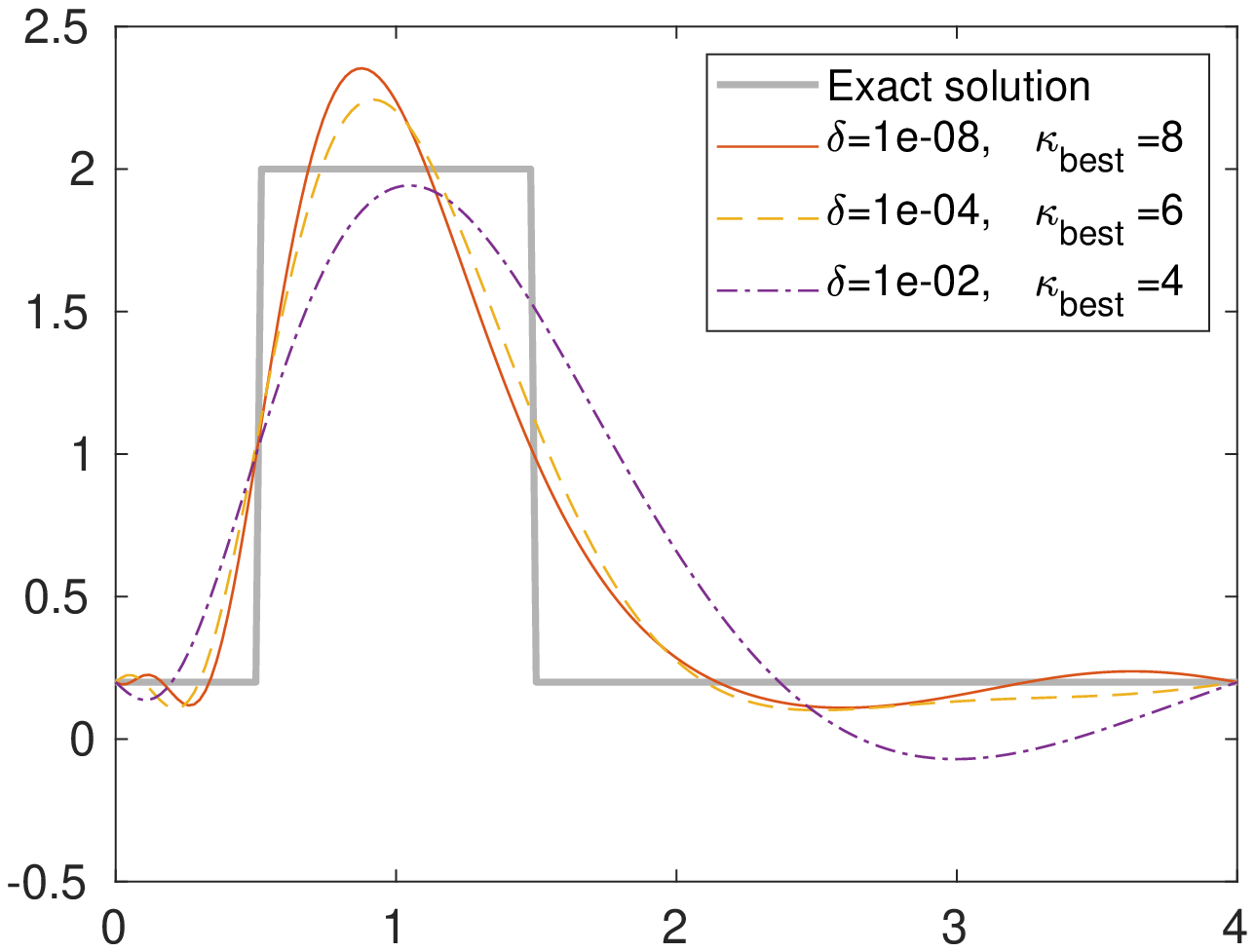}
\caption{Regularized solution $\sigma_2^{(\kappa_{\text{best}})}(z)$ (left) and
$\sigma_3^{(\kappa_{\text{best}})}(z)$ (right), for $n=10$ and different noise
levels $\delta = 10^{-8}, 10^{-4}, 10^{-2}$; the optimal regularization
parameter $\kappa_\text{best}$ is displayed in the legend.}
\label{fig8}
\end{figure}


\paragraph{Test function 3.}
The third model function is the step function
\begin{equation*}
\sigma_3(z)= \begin{cases}
0.2, & z \in (0,0.5), \\
2, & z \in [0.5,1.5], \\
0.2, & z \in (1.5, \infty),
\end{cases}
\end{equation*}
with $\alpha=\beta=0.2$.

The graph on the right-hand side of Figure~\ref{fig8} reports the optimal
regularized solutions for $\delta = 10^{-8}$, $10^{-4}$, $10^{-2}$, and 
$n=10$.
Since the function is discontinuous, comments similar to the previous
example are valid.


\medskip

\section*{Acknowledgements}
The authors would like to thank an anonymous referee for his insightful 
comments that lead to improvements of the presentation.

Luisa Fermo, Federica Pes, and Giuseppe Rodriguez are partially supported by
Regione Autonoma della Sardegna research project ``Algorithms and Models for
Imaging Science [AMIS]'' (RASSR57257, intervento finanziato con risorse FSC
2014-2020 - Patto per lo Sviluppo della Regione Sardegna). Luisa Fermo is
partially supported by INdAM-GNCS 2020 project ``Approssimazione multivariata ed
equazioni funzionali per la modellistica numerica''. Patricia D\'iaz de Alba,
Federica Pes, and Giuseppe Rodriguez are partially supported by INdAM-GNCS 2020
project ``Tecniche numeriche per l'analisi delle reti complesse e lo studio dei
problemi inversi''. Patricia D\'iaz de Alba gratefully acknowledges Fondo
Sociale Europeo REACT EU - Programma Operativo Nazionale Ricerca e Innovazione
2014-2020 and Ministero dell'Universit\`a e della Ricerca for the financial
support. Federica Pes gratefully acknowledges CRS4 (Centro di Ricerca, Sviluppo
e Studi Superiori in Sardegna) for the financial support of her Ph.D.
scholarship.

\bibliographystyle{siam}
\bibliography{biblio}

\end{document}